\documentclass[10pt,a4paper]{article}
\usepackage{hyperref}
\usepackage{amsmath,graphicx}
\usepackage{tabularx}
\usepackage{psfrag}
\usepackage{xcolor}
\usepackage{amsfonts}
\usepackage{amssymb}
\usepackage{subcaption}
\usepackage[ruled,vlined]{algorithm2e}
\usepackage{booktabs}
\usepackage{multirow}
\usepackage{enumerate}
\usepackage{cite}		
\newtheorem{theorem}{Theorem}
\newtheorem{proposition}{Proposition}                     
\newtheorem{proof}{Proof}                                                                  
\newtheorem{definition}{Definition}
\graphicspath{{figures/}}
\definecolor{rouge}{rgb}{0.7294,0.0392,0.0392}

\begin{document}

\title{
How Joint Fractal Features Estimation and Texture Segmentation can be cast into a Strongly Convex Optimization Problem ?
}

\author{Barbara Pascal, Nelly Pustelnik and 
        Patrice Abry
\thanks{B. Pascal, N. Pustelnik and P. Abry are with the Universit\'{e} de Lyon, ENS de Lyon, CNRS,  Laboratoire de Physique, 
Lyon, France,  e-mail: { \footnotesize \tt firstname.lastname@ens-lyon.fr}}
\thanks{This work was supported by Defi Imag'in SIROCCO and by ANR-16-CE33-0020
MultiFracs, France.}}

\maketitle

\section{Introduction}
\label{sec:introduction}

\noindent {\bf Context: unsupervised texture segmentation.}  
Textured images appear as \emph{natural} models for a large variety of real-world applications very different in nature.
Often, fractal attributes are used to relevantly characterize such real-world textures.
This is the case with
biologic tissues \cite{mccann2014images}, tomography-based pathology diagnostic \cite{ibrahim2017identification, marin_mammographic_2017}, art painting expertise~\cite{Wendt2013a}, microfluidics~\cite{pascal2018joint}, to name a few examples.

Often in these applications, texture segmentation (i.e., partitioning images into regions with homogeneous features) remains an on-going and major challenge. 
In computer vision or scene analysis, there exist numerous well-established and efficient methods to partition images mostly relying on their geometrical properties (cf. e.g., \cite{ray1999determination,getreuer2012chan,mumford1989optimal,jung-ssvm-11}). 
For textured images, segmentation remains challenging as geometry is more difficult to capture, relying essentially on the statistics of the texture features.

Most classical texture segmentation approaches rely on two-step procedures, consisting of first computing a priori chosen texture features and of, second, grouping features into regions with homogeneous features statistics.
Attempts to improve these traditional approaches address their tow main limitations: arbitrary a priori feature selection and suboptimality of the two-step procedure. 
For instance, deep learning and neural network approaches have strongly contributed to renew image segmentation, avoiding notably the a priori selection of specific features, and combining together feature estimation and segmentation (cf. e.g., \cite{andrearczyk2016using}). 
However, their use remains designed mostly for the specific context of supervised segmentation: 
A (usually large) training dataset needs to be available, together with large computer/memory powers. 
Often, such databases are not available as expert annotations may be too expensive (in time and money) to implement or technically beyond reach. 
Besides 
technical 
and financial issues, assembling such databases may allow raise non trivial ethical problems (cf. e.g.,~\cite{cimpoi2016deep}).
Along that line, a severe drawback of neural network approaches lies in the lack of \emph{interpretability} or even of \emph{identification} of the features decisions are based on: Doctors for instance might legitimately remain reluctant to base diagnostic on non-identified features. 
Therefore, despite the potential of deep learning, in contexts of absence of documented database, of importance of accuracy in estimated boundaries, and of a need of understandability, unsupervised segmentation strategies remain of critical importance. 

The present work is thus dedicated to the unsupervised segmentation of textures into regions well-characterized by piecewise constant fractal features (local regularity and local variance) formulated as a nonsmooth strongly convex optimization problem, which can be solved with fast and scalable algorithms. \\

\noindent {\bf Related work.} 
Amongst classical features commonly enrolled in texture segmentation, one can list e.g., Gabor or short-term Fourier transform coefficients~\cite{jain1991unsupervised, dunn1994texture}, fractal dimension~\cite{chaudhuri1995texture}, Amplitude/Frequency Modulation models~\cite{kokkinos2009texture}.
More recently, fractal (or scale-free) features were also involved in texture segmentation (cf. e.g., \cite{Wendt2009b}). 
Notably, \emph{local regularity} was used to quantify the fluctuations of roughness along the texture. 
Local regularity is quantified as an \emph{optimal}~local power-law behavior across scales for some multiscale quantities \cite{Wendt2009b}.
Modulus of wavelet coefficients \cite{Mallat_S_1997_book_wav_tsp} were initially used as multiscale quantities, followed by continuous wavelet transform modulus maxima \cite{Mallat_S_1992_j-ieee-pami_cha_sme,Muzy_J_1991_j-prl_wav_mfs}. 
More recently, \textit{wavelet leaders} (local suprema of discrete wavelet coefficients), used here, 
were shown to permit a theoretically accurate and practically robust estimation of local regularity \cite{Wendt2009b},
and successfully involved in texture characterization in e.g., \cite{Wendt2013a,nelson2016semi,marin_mammographic_2017}. 
Wavelet leaders are used to estimate local regularity at each pixel mostly by linear regressions in log-log coordinates.

For the segmentation step, local Gabor coefficient histograms were grouped using matrix factorization \cite{yuan_factorization-based_2015},~;  
 textons were combined with brightness and color features to yield a multiscale contour detection procedure \cite{arbelaez2011contour}.  
Further, with fractal features, pixels sharing similar estimates of local regularity are
grouped together via a functional minimization strategy \cite{pustelnik_combining_2016}, using Rudin-Osher-Fatemi (ROF) model~\cite{rudin1992nonlinear} (i.e. \textit{total variation} (TV) denoising). 

Further, in~\cite{pustelnik_combining_2016}, it was also attempted to combine both steps into a single one by incorporating the regression weight estimation into the optimization procedure. 
The high computational burden implied by iteratively solving optimization problems as well as by tuning the regularization parameter has tentatively been addressed by \emph{block splitting approaches} as suggested in~\cite{repetti_parallel_2015} and explored in \cite{pascal_block_2018}. 
Strong convexity \cite{chambolle2011first} constitutes another recently proposed track to address iterative optimization acceleration.
However, the proposed joint features estimation and segmentation procedures, despite showing satisfactory and state-of-the-art segmentation performance suffered from major limitations:
Their high computing cost prevents their use of large images and databases~; While based on a key fractal feature, local regularity, they neglect changes in power, 
a potentially relevant information for texture segmentation, notably to extract accurate region boundaries.\\

\noindent {\bf Goals, contributions and outline.} 
Aiming to address the above limitations, the overall goals of the present contribution are to investigate the potential benefits for texture segmentation brought by
\begin{itemize}
\item the use of joint estimates of fractal features (local variance and local regularity)~;
\item the formulation of feature estimation and segmentation as a single step taking the form of a convex minimization~;
\item the derivation of fast and scalable iterative algorithms to solve the optimization problem, permitted by explicitly proving (and measuring) strong convexity of  the data fidelity term in the objective function~;
\item the comparisons of several optimization formulations, favoring changes in features that are either co-localized in space or independent~;
\item the derivation of an explicit stochastic process, modeling piecewise fractal textures, thus permitting to conduct large-size Monte-Carlo simulation for performance benchmarking and comparisons. 
\end{itemize} 

To that end, key concepts and definitions 
related to local regularity, wavelet leaders and corresponding state-of-the-art linear regression followed by TV-based estimation procedures are recalled in Section~\ref{sec:local_regularity}. 
Two alternative objective functions
for the combined estimation/segmentation of local variance and regularity, respectively referred to as \emph{joint}~and \emph{coupled}, are proposed in Section~\ref{sec:optimization}, based on the same data fidelity term but on two different total-variation based regularization strategies. 
Studying strong convexity and duality gaps, two classes of fast iterative algorithms (\emph{Primal-Dual} and \emph{Dual Forward-backward}), are devised. 
Piecewise homogeneous fractal textures are defined and studied in Section~\ref{sec:synthesis}, thus enabling performance assessment from Monte-Carlo simulations on well-controlled synthetic textures, as reported in Section~\ref{sec:performance}. 
Conducting such simulations requires to address issues related to regularization parameter selection, iterative algorithm stopping criterion, and performance metrics. 
It permits relevant answers to the final aim of assessing the actual benefits of using both local variance and regularity in texture segmentation with respect to issues such as sensitivity to fractal parameter changes, computational costs, impact of the different optimization formulations.  
Finally, the potential of the proposed segmentation approaches is illustrated at work on real-world piecewise homogeneous textures taken from a publicly available and documented texture database.
Performance are also compared against a state-of-the-art segmentation strategy. 

A {\sc matlab} toolbox implementing the analysis and synthesis procedures devised here will be made freely and publicly available at the time of publication.

\section{Local regularity}
\label{sec:local_regularity}
\subsection{Local regularity and H\"older exponent}

Local regularity can classically be assessed by means of H\"older exponent~\cite{jaffard2004wavelet,pLeadersPartII2015}, defined as follows:   
\begin{definition} Let $f\colon \Omega   \to \mathbb{R}$ denote a 2D real field defined on an open set $\Omega \subset \mathbb{R}^2$. 
The H\"older exponent $h(\underline{z}_0)$ at location $\underline{z}_0 \in \Omega $  is defined as the largest $\alpha > 0$ such that there exists a constant $\chi>0$, a polynomial $\mathcal{P}_{\underline{z}_0}$ of degree lower than $\alpha$ and a neighborhood $\mathcal{V}(\underline{z}_0)$ satisfying:  \\
$$\big(\forall \underline{z} \in \mathcal{V}(\underline{z}_0)\big) \quad \left\lvert f(\underline{z}) - \mathcal{P}_{\underline{z}_0}(\underline{z}) \right\rvert \leq \chi \left\lVert \underline{z} - \underline{z}_0 \right\rVert^{\alpha}$$
 where $\lVert \cdot \rVert$ denotes the Euclidian norm.
\label{def:holder}
\end{definition}
Definition~\ref{def:holder} does not provide a practical way to estimate local regularity. 
Thus, the practical assessment of $h(\underline{z}_0)$ usually relies on the use of multiscale quantities, such as wavelet coefficients, or wavelet leaders~\cite{jaffard2004wavelet,pLeadersPartII2015,Wendt2009b}.

\subsection{Local regularity and wavelet leaders}
\label{sec:wavleaders}

Because the practical aim is to analyse digitized images, definitions are given in a discrete setting. 
Let $X =(X_{\underline{n} })_{\underline{n} \in \Upsilon}\in \mathbb{R}^{\lvert \Upsilon \rvert}$ denote the digitized version of the 2D real field $f$ on a finite grid $\Upsilon =\lbrace1, \hdots, N \rbrace^2 $. 
Let $d_j = \textbf{W}_j X$ denote the discrete wavelet transform (DWT) coefficients of $X$,  at resolution $j\in \{j_1,\ldots,j_2\}$, 
with $ \textbf{W}_j \colon \mathbb{R}^{N \times N}\rightarrow\mathbb{R}^{M_j \times M_j}$ the operator formulation of the DWT.

Further, let the wavelet leader $\mathcal{L}_{j,\underline{k}}$, at scale $2^j$ and location $\underline{n} = 2^{j}\underline{k}$, be defined as the local supremum of modulus of wavelet coefficients in a small neighborhood across all finer scales,~\cite{jaffard2004wavelet,pLeadersPartII2015,Wendt2009b} $\mathcal{L}_{j, \underline{k}} = \underset{
\lambda_{j',\underline{k}'} \subset 3\lambda_{j, \underline{k}}}{\text{sup}} | 2^{j} d_{j',\underline{k}'} |$, with $  \lambda_{j, \underline{k}} = [\underline{k}2^j, (\underline{k}+1)2^j) $ and $3\lambda_{j, \underline{k}} = \underset{\underline{p} \in \lbrace -1, 0, 1 \rbrace^2}{\bigcup} \lambda_{j, \underline{k} + \underline{p}}$.

It was proven in~\cite{jaffard2004wavelet,pLeadersPartII2015,Wendt2009b} that wavelet leaders provide multiscale quantities intrinsically tied to H\"older exponents, insofar as when $X$ has H\"older exponent $h_{\underline{n}} $ at location $\underline{n}$, then the wavelet leaders $\mathcal{L}_{j,\underline{k}}$ satisfy: 
\begin{align}
\label{eq:model}
(\forall j) \, \, \mathcal{L}_{j,\underline{k}} \simeq \eta_{\underline{n}} 2^{j h_{\underline{n}}} \, \, \, \text{as} \, \, 2^j \rightarrow 0, \, \, \, \mathrm{with} \, \, \underline{n} = 2^j\underline{k} \in \lambda_{j,k},
\end{align}
with $\eta_{\underline{n}}$ proportional to local variance of $X$ around pixel $\underline{n}$.

\subsection{Local regularity estimation}
\label{sec:estimation}

\subsubsection{Linear regression} Eq.~\eqref{eq:model} naturally leads to estimate $\boldsymbol{h}$ and $\boldsymbol{v} = \log_2  \boldsymbol{\eta} $ by means of linear regression in log-log coordinates~\cite{jaffard2004wavelet,pLeadersPartII2015,Wendt2009b}, denoted $\left(\widehat{\boldsymbol{v}}_{\mathrm{LR}}, \widehat{\boldsymbol{h}}_{\mathrm{LR}}\right)$: 
\begin{align}
\label{eq:linsol}
(\forall \underline{n}\in \Upsilon)\qquad \begin{pmatrix}
\widehat{v}_{\mathrm{LR},\underline{n}} \\
\widehat{h}_{\mathrm{LR},\underline{n}}
\end{pmatrix} = \textbf{J}^{-1} \begin{pmatrix}
\mathcal{S}_{\underline{n}} \\
\mathcal{T}_{\underline{n}}
\end{pmatrix} 
\end{align}
\begin{equation}
\label{eq:sm}
 \mbox{with} \quad 
 \textbf{J} = \begin{pmatrix}
R_0 & R_1 \\
R_1 & R_2
\end{pmatrix} \quad \mbox{and} \quad R_m = \sum_{j} j^m,
\end{equation}
\begin{equation}
\label{eq:ST}
\makebox{ and} \quad  \mathcal{S}_{\underline{n}} = \sum_j \log_2 \mathcal{L}_{j, \underline{n}}, \quad 
 \mathcal{T}_{\underline{n}} = \sum_j j \log_2 \mathcal{L}_{j,\underline{n}}.
\end{equation}
The sum $\sum_j $ implicitly stand for $\sum_{j=j_1}^{j_2}$ with $(j_1,j_2)$ the range of octaves involved in the estimation.
Though linear regressions provide estimates both for  $\boldsymbol{h}$ and $\log_2 \boldsymbol{\eta}$, the later remains to date rarely used.
Estimates of $\boldsymbol{h}$ were for instance used in~\cite{pustelnik_combining_2016,pascal_block_2018}.
Bayesian based extension of linear regressions were also and alternatively proposed to estimate $\eta$ and $h$~\cite{GaMRF2018}.  

A sample of piecewise homogeneous fractal texture is displayed in Fig.~\ref{subfig:ex_txt_2}, synthesized using the mask  sketched in Fig.~\ref{subfig:mask}. 
The corresponding linear regression based estimates 
$\widehat{\boldsymbol{\eta}}_{\mathrm{LR}} \equiv 2^{\widehat{\boldsymbol{v}}_{\mathrm{LR}}} $,
and  $\widehat{\boldsymbol{h}}_{\mathrm{LR}} $ (Fig.~\ref{fig:reglin}(c-d)) show too poor performance (high variability) to permit to detect the two regions and the corresponding boundary.

\begin{figure}[h!]
\begin{subfigure}{0.32\linewidth}
\centering
\includegraphics[width = 0.8\linewidth]{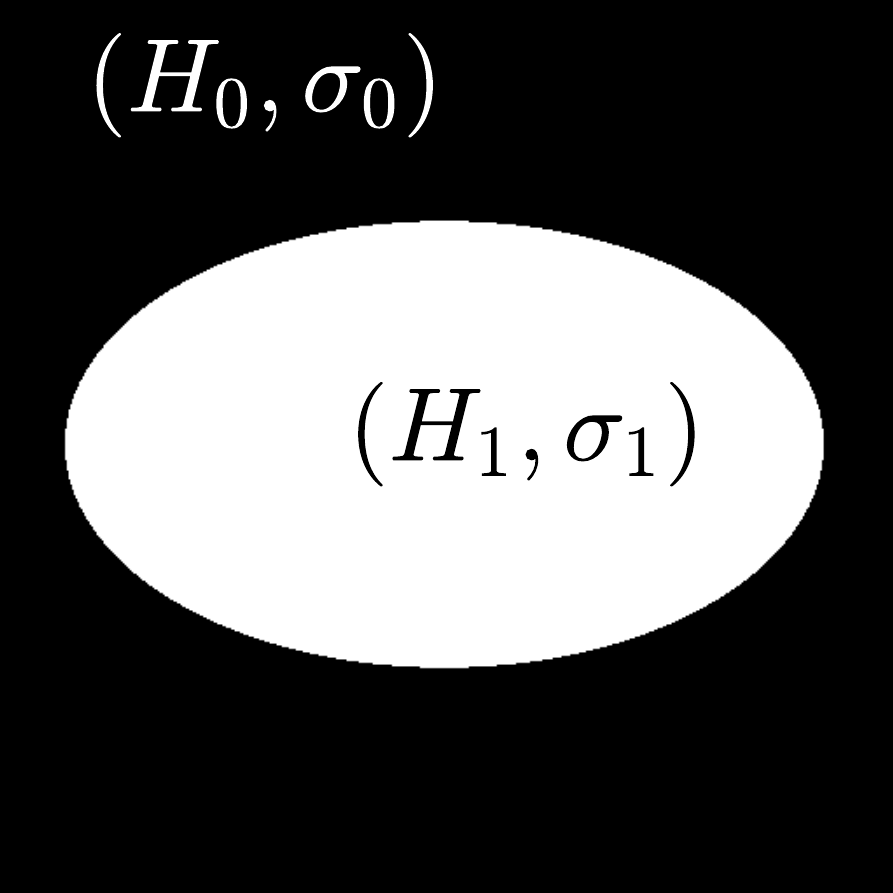}
\subcaption{\label{subfig:mask}Mask}
\end{subfigure} \hfill
\begin{subfigure}{0.32\linewidth}
\centering
\includegraphics[width = 0.8\linewidth]{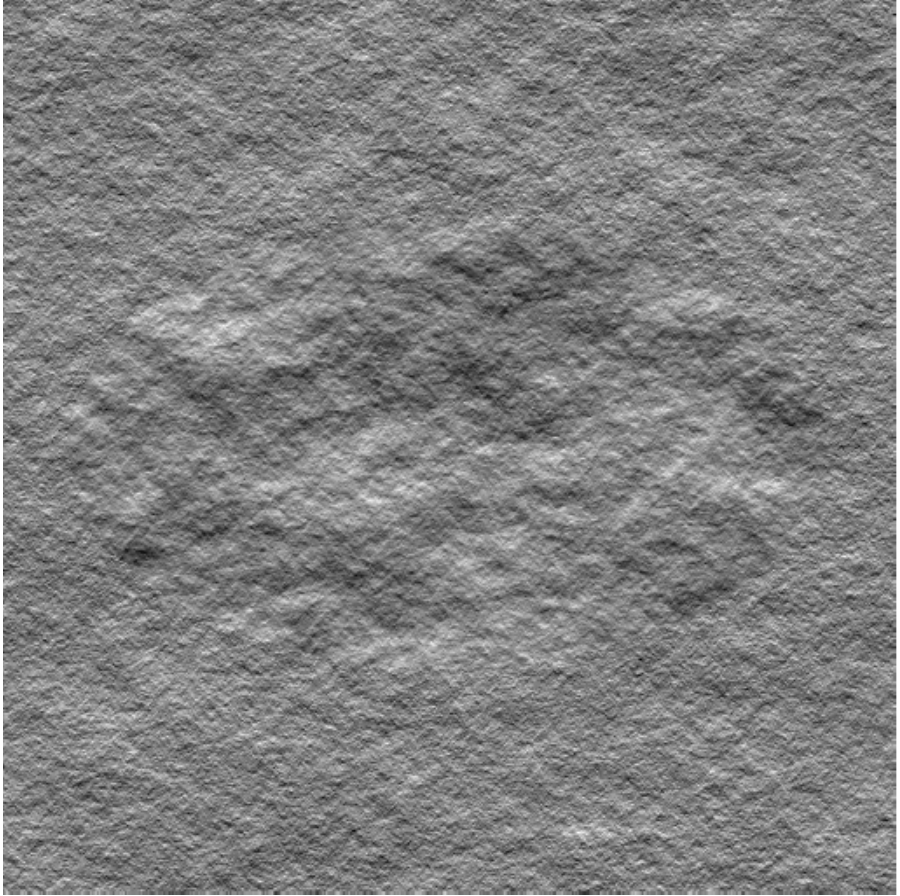}
\subcaption{\label{subfig:ex_txt_2}Texture $X$}
\end{subfigure}\hfill
\begin{subfigure}{0.32\linewidth}
\centering
\includegraphics[width = 0.8\linewidth]{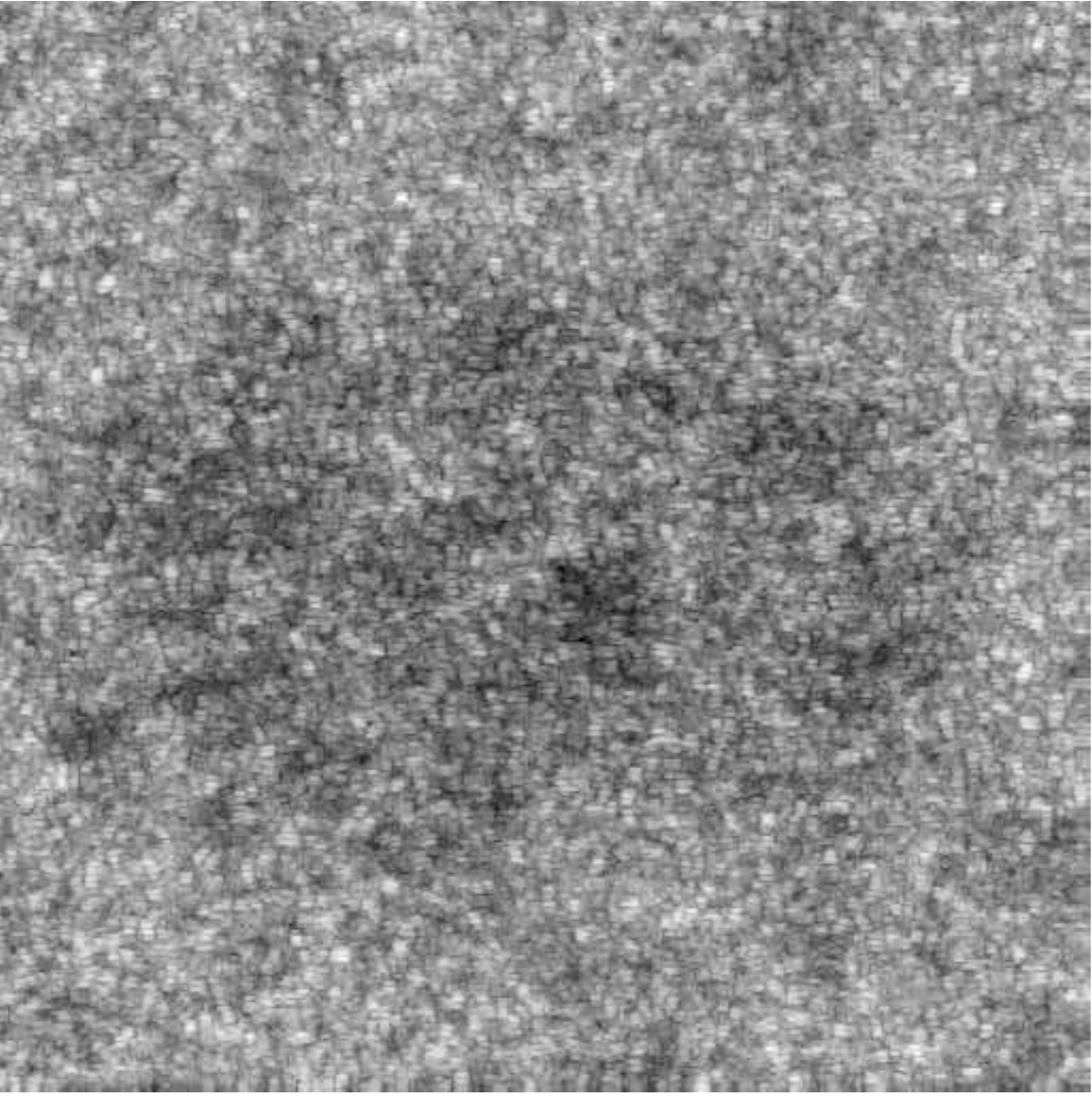}
\subcaption{\label{subfig:ex_etareg_2}$\widehat{\boldsymbol{v}}_{\mathrm{LR}}$}
\end{subfigure}\hfill
\begin{subfigure}{0.32\linewidth}
\centering
\includegraphics[width = 0.8\linewidth]{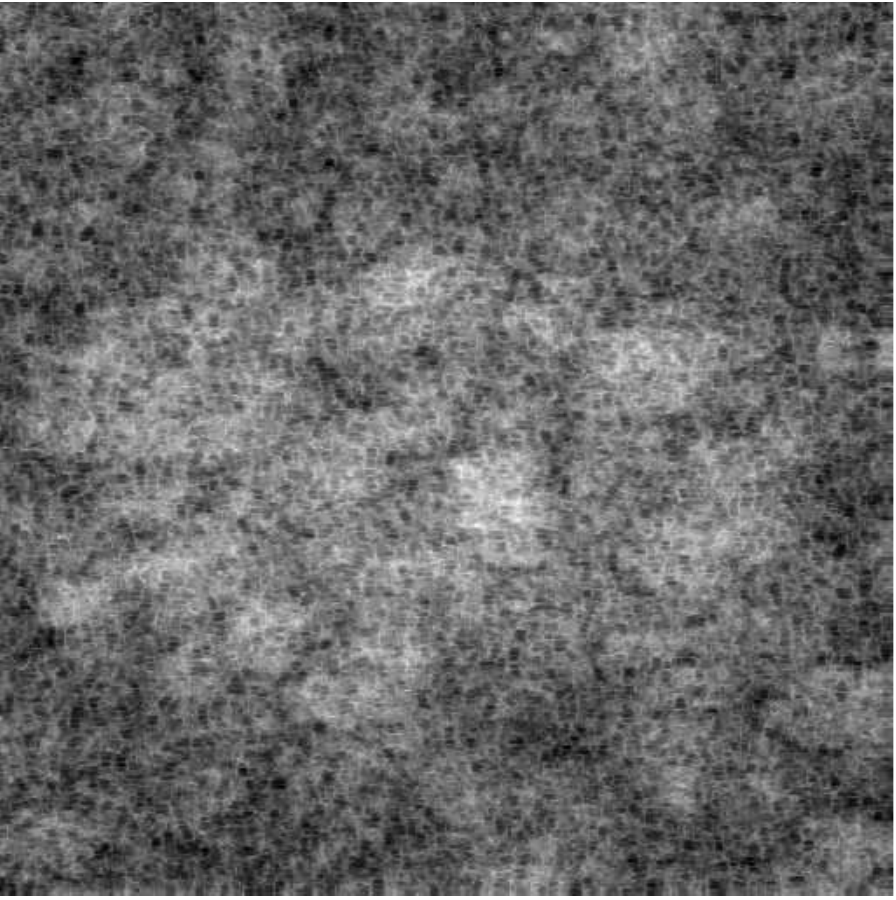}
\subcaption{\label{subfig:ex_hreg_2}$\widehat{\boldsymbol{h}}_{\mathrm{LR}}$}
\end{subfigure}
\begin{subfigure}{0.32\linewidth}
\centering
\includegraphics[width = 0.8\linewidth]{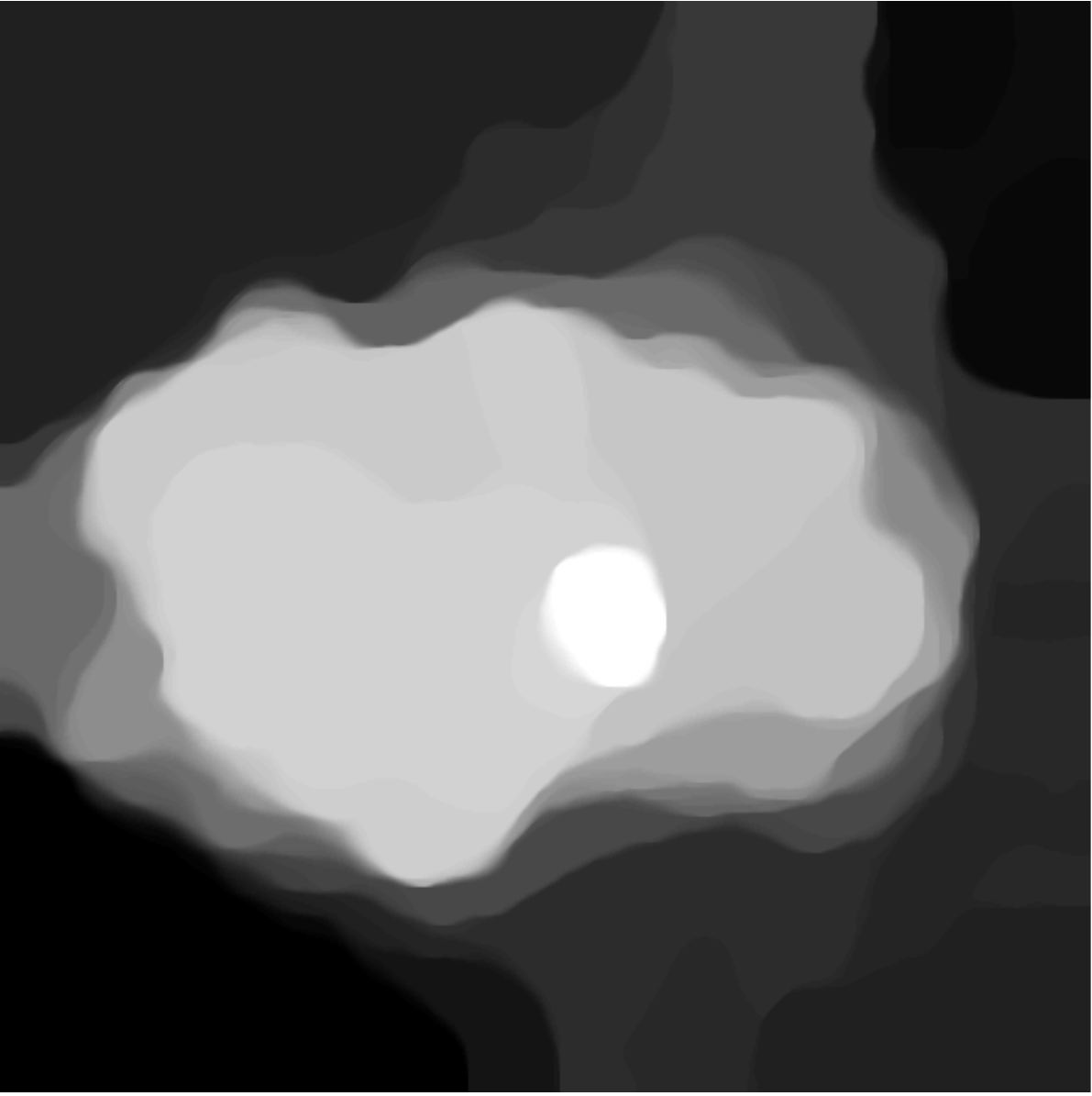}
\subcaption{\label{subfig:ex_hrof_2}$\widehat{\boldsymbol{h}}_{\mathrm{ROF}}$}
\end{subfigure}\hfill
\begin{subfigure}{0.32\linewidth}
\centering
\includegraphics[width = 0.8\linewidth]{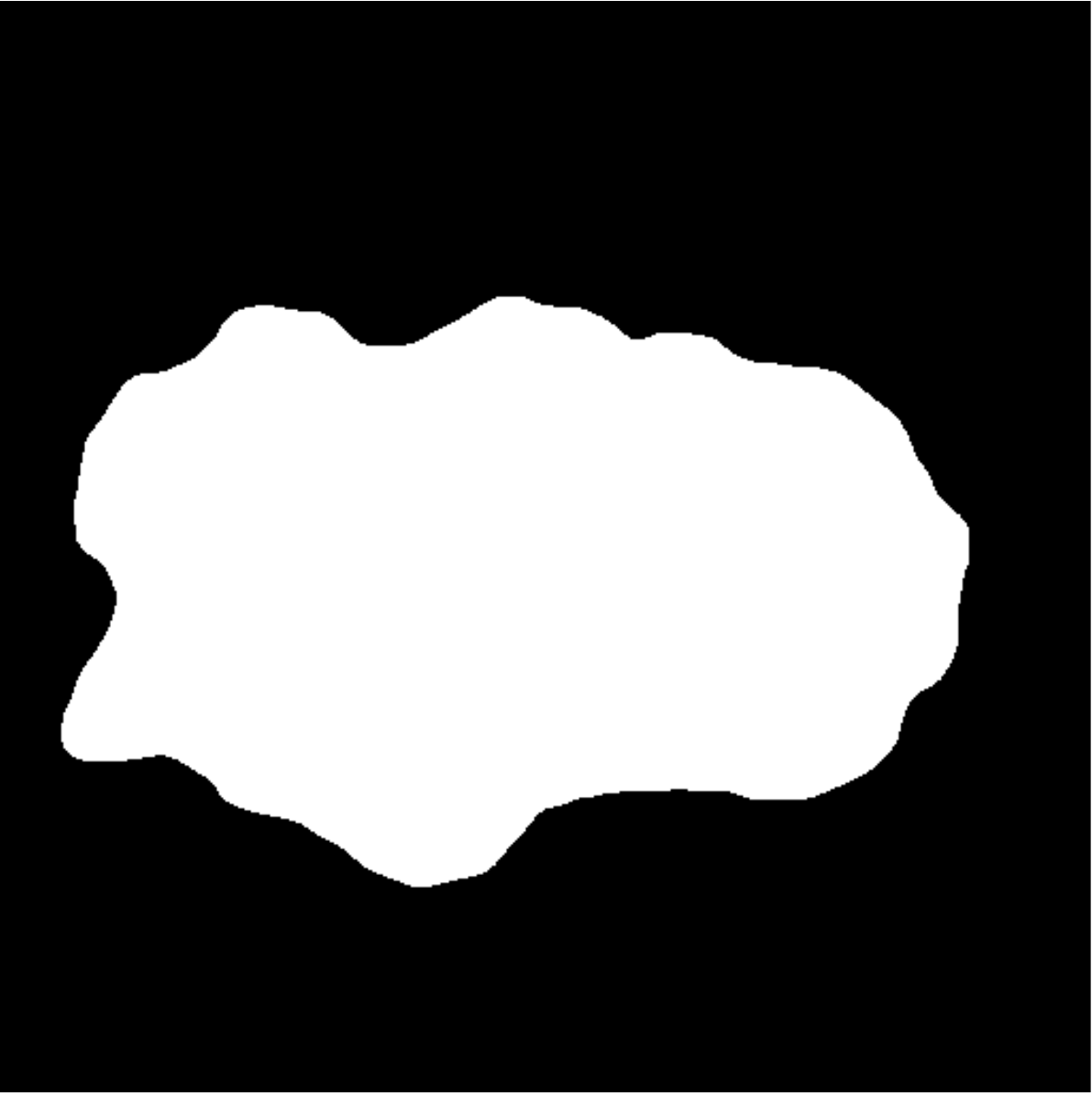}
\subcaption{\label{subfig:ex_trof_2}$\widehat{\boldsymbol{M}}_{\mathrm{ROF}}$}
\end{subfigure}
\caption{\label{fig:reglin} {\bf Piecewise homogeneous fractal texture and local variance and regularity estimates.} (a) Synthesis mask $(H_0, \Sigma_0)=(0.5,0.6)$ and $(H_1, \Sigma_1)=(0.8, 0.65)$;  (b) sample texture (see Section~\ref{sec:synthesis}), (c) and (d) linear regression based estimates of local variance and regularity $\widehat{\boldsymbol{v}}_{\mathrm{LR}}$ and  $\widehat{\boldsymbol{h}}_{\mathrm{LR}}$; (e) and (f) Total variation based estimates of local regularity $\widehat{\boldsymbol{h}}_{\mathrm{ROF}}$ and segmentation $\widehat{\boldsymbol{M}}_{\mathrm{ROF}}$ obtained with Alg.~\ref{alg:trof}.}
\label{fig:xd}
\end{figure}

\subsubsection{Total variation based estimates}
\label{ssec:trof}

To address the poor estimation performance achieved by linear regressions for local estimation, it has been proposed to \emph{denoise} $\widehat{\boldsymbol{h}}_{\mathrm{LR}}$ by using an optimisation framework involving TV regularization,
as a mean to enforce piecewise constant estimates~\cite{pustelnik_combining_2016, nelson2016semi}: 
\begin{align}
\label{eq:tvsep}
\widehat{\boldsymbol{h}}_{\mathrm{ROF}} = \underset{\boldsymbol{h} \in \mathbb{R}^{\lvert \Upsilon \rvert}}{\mathrm{argmin}} \, \frac{1}{2} \lVert \boldsymbol{h} - \widehat{\boldsymbol{h}}_{\mathrm{LR}} \rVert_2^2 + \lambda \mathrm{TV}(\boldsymbol{h}).
\end{align}
The above functional balances a data fidelity term versus a total variation penalization defined as:
\begin{align}
\label{eq:deftv}
\mathrm{TV}(\boldsymbol{h}) = \lVert \textbf{D}\boldsymbol{h} \rVert_{2,1},
\end{align}
where the operator $\textbf{D} : \mathbb{R}^{\lvert \Upsilon \rvert} \rightarrow \mathbb{R}^{2 \times \lvert \Upsilon \rvert}$, consists of horizontal and vertical increments: 
\begin{align}
\label{eq:defD}
\left( \textbf{D} \boldsymbol{h} \right)_{n_1, n_2} = \begin{pmatrix}
h_{n_1,n_2+1} - h_{n_1,n_2} \\
h_{n_1 + 1,n_2} - h_{n_1,n_2}
\end{pmatrix}
\end{align}
and where the mixed norm, $\lVert \cdot \rVert_{2,1}$, defined for $\boldsymbol{z}  = [\boldsymbol{z}_1 ; \ldots ; \boldsymbol{z}_I]^\top \in \mathbb{R}^{I \times \lvert \Upsilon \rvert}$ as: 
\begin{align}
\label{eq:norm21o}
\lVert \boldsymbol{z} \rVert_{2,1} = \sum_{n_1 = 1}^{N_1-1} \sum_{n_2 = 1}^{N_2-1} \sqrt{\sum_{i=1}^I(z_i)^2_{n_1,n_2}}
\end{align}
both ensures an isotropic and piecewise constant behaviour of the estimates~\cite{pustelnik_combining_2016}.

While showing substantial benefits compared to linear regression based estimates and althought fast, the TV-based procedure above, proposed in~\cite{pustelnik_combining_2016}, suffers from several significant limitations: 
i) It is restricted to the estimation of local regularity only and neglects local variance~;
ii) It consists of a two-step process (first apply linear regression to obtain $\widehat{\boldsymbol{h}}_{\mathrm{LR}}$, second apply TV to  $\widehat{\boldsymbol{h}}_{\mathrm{LR}}$ to obtain $\widehat{\boldsymbol{h}}_{\mathrm{ROF}} $) and is hence potentially not optimal.
Section~\ref{sec:optimization} below will propose several solutions that address these two limitations by performing 
the estimate of both local variance and regularity in a single pass at similar computational cost than the two-step procedure.

\subsubsection{A posteriori segmentation and local regularity global estimation} 
\label{sec:posterior}

When the number of regions is known \textit{a priori}, it has been shown in~\cite{cai2013multiclass,cai2018linkage} that a fast iterative thresholding post-processing procedure can be applied to $\widehat{\boldsymbol{h}}_{\mathrm{ROF}} $ to obtain a segmentation of the TV denoised estimate. 
In~\cite{cai2013multiclass,cai2018linkage}  were provided theoretical guarantees linking this Threshold-ROF (T-ROF) strategy to Mumford-Shah-like segmentation. 
Alg.~\ref{alg:trof} explicitly describes the procedure applied to $\widehat{\boldsymbol{h}}_{\mathrm{ROF}} $ to obtain a two-region  T-ROF~segmentation $\widehat{\boldsymbol{M}}_{\mathrm{ROF}}$ with $\widehat{M}_{\mathrm{ROF,\underline{n}}} = 0 $ if  $ \underline{n} \in \Upsilon_0$ and $1$ otherwise. 
Figs.~\ref{subfig:ex_hrof_2},~\ref{subfig:ex_trof_2} report examples of TV denoised estimate $\widehat{\boldsymbol{h}}_{\mathrm{ROF}} $ and T-ROF~segmentation $\widehat{\boldsymbol{M}}_{\mathrm{ROF}}$.

\begin{algorithm}
\textbf{Input:} $\boldsymbol{h} $\\
\textbf{Initialization:} 
\, $\mathrm{m}_0^{[0]} = \underset{\underline{n} \in \Upsilon}{\min} \, h_{\underline{n}}$, \, 
$\mathrm{m}_1^{[0]} = \underset{\underline{n} \in \Upsilon}{\max} \, h_{\underline{n}}$.\\
\For{$t \in \mathbb{N}^*$}{
\underline{\textit{Compute the threshold:}}\\
$ \mathrm{T}^{[t-1]} = \left( \mathrm{m}_0^{[t-1]} + \mathrm{m}_1^{[t-1]}\right)/2 $\\
\underline{\textit{Threshold $\boldsymbol{h}$:}}\\
$\Upsilon_0^{[t]} = \lbrace \underline{n} \, | \, h_{\underline{n}} \leq \mathrm{T}^{[t]} \rbrace$, \,
$\Upsilon_1^{[t]} = \lbrace \underline{n} \, | \, h_{\underline{n}} > \mathrm{T}^{[t]} \rbrace$ \\
\underline{\textit{Update region mean:}}\\
$\displaystyle \mathrm{m}_0^{[t]} = 1/ \lvert \Upsilon_0 \rvert \sum_{\underline{n} \in \Upsilon_0}h_{\underline{n}}$, \,
$\displaystyle \mathrm{m}_1^{[t]} = 1/ \lvert \Upsilon_1 \rvert \sum_{\underline{n} \in \Upsilon_1}h_{\underline{n}}$.\\
}
\textbf{Output:} $\Upsilon_0 = \Upsilon_0^{[\infty]}$, \, $\Upsilon_1 = \Upsilon_1^{[\infty]}$
\caption{\label{alg:trof}T-ROF: iterative thresholding of $\widehat{\boldsymbol{h}}_{\mathrm{ROF}}$}
\end{algorithm}

From T-ROF~segmentation, Texture $X$ can be interpreted as the concatenation of two fractal textures $X_0$ and $X_1$, each with its own uniform regularity and variance. 
Therefore, a posteriori global estimates, $\widehat{H}_{0,\mathrm{ROF}}$ and $\widehat{H}_{1, \mathrm{ROF}}$, for the local regularity of each region, $\Upsilon_0$ and $\Upsilon_1$, can be obtained by performing linear regressions applied to the logarithm of
 multiscale quantities $\boldsymbol{\mathcal{L}}_j$ averaged over $\Upsilon_0$ (resp. $\Upsilon_1$) (cf.~\cite{Wendt2009b}).

\section{Total variation based estimation of local variance and regularity}
\label{sec:optimization}

\subsection{Design of the objective function}

\subsubsection{Linear regression as functional minimization} Linear regression can be viewed as an optimization scheme. 
Indeed, setting $\boldsymbol{v} = \log_2 \boldsymbol{\eta}$
the following strictly convex functional (in variables $(\boldsymbol{v},\boldsymbol{h})$):
\begin{align}
\label{eq:defF}
\Phi( \boldsymbol{v}, \boldsymbol{h} ; \boldsymbol{\mathcal{L}}) = \frac{1}{2} \sum_{j}  \left\lVert  \boldsymbol{v}  +j  \boldsymbol{h} - \log_2 \boldsymbol{ \mathcal{L}}_{j}\right\rVert_{\mathrm{Fro}}^2,
\end{align}
 has a unique minimum corresponding to the linear regression estimates in \ref{eq:linsol}.

\subsubsection{Penalization}
To improve on the poor performance of linear regression, a generic optimization framework can be proposed as 
$(\widehat{\boldsymbol{v}}, \widehat{\boldsymbol{h}}) = \underset{\boldsymbol{v} \in \mathbb{R}^{\lvert \Upsilon \rvert}, \boldsymbol{h} \in \mathbb{R}^{\lvert \Upsilon \rvert}}{\arg \min} \Phi(\boldsymbol{v},\boldsymbol{h}; \boldsymbol{\mathcal{L}}) + \Psi(\boldsymbol{v},\boldsymbol{h}),
$
where the data fidelity term $\Phi$ is that of the linear regression (cf.\eqref{eq:defF}), and where 
the penalization term  $\Psi$ favors piecewise constancy of $\boldsymbol{v}$ and $\boldsymbol{h}$.
Two different strategies are proposed:
\begin{enumerate}[(i)]
\item \emph{Joint estimation}:
\begin{equation}
\label{eq:jointtv}
(\widehat{\boldsymbol{v}}_{\mathrm{J}}, \widehat{\boldsymbol{h}}_{\mathrm{J}}) = \underset{\boldsymbol{v}, \boldsymbol{h} \in \mathbb{R}^{\lvert \Upsilon \rvert}\times \mathbb{R}^{\lvert \Upsilon \rvert}}{\arg \min} \;\Phi(\boldsymbol{v},\boldsymbol{h}; \boldsymbol{\mathcal{L}}) + \Psi_J(\boldsymbol{v},\boldsymbol{h}),
\end{equation}
\begin{align}
\label{eq:deftvj}
\makebox{ with } \quad \Psi_{\mathrm{J}}(\boldsymbol{v},\boldsymbol{h}) = \lambda \left( \mathrm{TV}(\boldsymbol{v}) + \alpha \mathrm{TV}(\boldsymbol{h}) \right), 
\end{align}
that does not favor changes in $\boldsymbol{v}$ and $\boldsymbol{h}$ that occur at same location;
\item \emph{Coupled estimation}: 
\begin{equation}
\label{eq:coupledtv}
(\widehat{\boldsymbol{v}}_{\mathrm{C}}, \widehat{\boldsymbol{h}}_{\mathrm{C}}) = \underset{\boldsymbol{v}, \boldsymbol{h} \in \mathbb{R}^{\lvert \Upsilon \rvert}\times \mathbb{R}^{\lvert \Upsilon \rvert}}{\arg \min}\Phi(\boldsymbol{v},\boldsymbol{h}; \boldsymbol{\mathcal{L}}) + \Psi_{\mathrm{C}}(\boldsymbol{v},\boldsymbol{h}),
\end{equation}
\begin{align}
\label{eq:deftvc}
\makebox{ with } \quad  \Psi_{\mathrm{C}}(\boldsymbol{v},\boldsymbol{h})& = \lambda \mathrm{TV}_{\alpha}(\boldsymbol{v},\boldsymbol{h}),\nonumber\\
&=\lambda\big\Vert \big[ \textbf{D}\boldsymbol{v} ; \alpha\textbf{D}\boldsymbol{h} \big]^\top\big\Vert_{2,1}
\end{align}
where $\mathrm{TV}_{\alpha}$ couples spatial variations of $v_{\underline{n}}$ and $h_{\underline{n}}$ and thus favor their occurrences at same location. 
\end{enumerate}
For both constructions, $\lambda > 0$ and $\alpha > 0$ constitute regularization hyperparameters that need to be selected. 

Further, the thresholding procedure in~\cite{cai2013multiclass,cai2018linkage} is generalized to ``\textit{joint}" and ``\textit{coupled}" estimation strategies: Alg.~\ref{alg:trof}, applied to $\widehat{\boldsymbol{h}}_{\mathrm{J}}$ (resp. $\widehat{\boldsymbol{h}}_{\mathrm{C}}$) provides ``\textit{joint}" (resp. ``\textit{coupled}") segmentation $\widehat{\boldsymbol{M}}_{\mathrm{J}}$ (resp. $\widehat{\boldsymbol{M}}_{\mathrm{C}}$).

Finally, from the two-region segmentation $\widehat{\boldsymbol{M}}_{\mathrm{J}}$ (resp. $\widehat{\boldsymbol{M}}_{\mathrm{C}}$) obtained with Alg.~\ref{alg:trof}, one can use global techniques~\cite{Wendt2009b} to obtain \textit{a posteriori} estimates of the H\"older exponents of each region $\widehat{H}_{0,\mathrm{J}}$ and $\widehat{H}_{1,\mathrm{J}}$ (resp. $\widehat{H}_{0,\mathrm{C}}$ and $\widehat{H}_{1,\mathrm{C}}$) (cf. Sec.~\ref{ssec:trof}).

\subsection{Minimization algorithms}

The data fidelity function $\Phi$ is Lipschitz differentiable.
This is the case neither for $\Psi_{\mathrm{J}}$ nor for $\Psi_{\mathrm{C}}$. 
Therefore, gradient descent methods are not appropriate to solve \eqref{eq:jointtv} or \eqref{eq:coupledtv}. 
Instead, the use of proximal algorithms is adapted to the minimization of nonsmooth functions~\cite{Combettes2011}. 
While the non-differentiability stems from the mixed norm $\lVert \cdot \rVert_{2,1}$, appearing both in $\Psi_{\mathrm{J}}$ and in $\Psi_{\mathrm{C}}$, the linear operator $\textbf{D}$ into the nonsmooth term makes the computation of the proximal operator of total variation difficult~\cite{Chaux_C_2007_j-inv-prob_var_ffb, Pustelnik_N_2011_j-ieee-tip_par_pxa}.
Below two algorithms are devised: the dual FISTA and the primal-dual. 

\subsubsection{Dual FISTA} The well-known Fast Iterative Shrinkage Thresholding Algorihtm (FISTA) algorithm~\cite{chambolle_how_2014} is here customized to Problems~\eqref{eq:jointtv} and~\eqref{eq:coupledtv}, to achieve faster convergences than with the basic dual forward-backward. 
Corresponding iterations are detailed in Algorithms~\ref{alg:fistaj} and~\ref{alg:fistac}.
Convergence guarantees are specified in Theorems~\ref{th:fistaj} and~\ref{th:fistac}.
\begin{theorem}[Convergence of $\mathrm{FISTA}_{\mathrm{J}}$]
\label{th:fistaj}
The sequence $(\boldsymbol{v}^{[t]},\boldsymbol{h}^{[t]})$ generated by Algorithm~\ref{alg:fistaj} converges towards a solution of the \textit{joint estimation problem~\eqref{eq:jointtv}}.
\end{theorem}
\begin{theorem}[Convergence of $\mathrm{FISTA}_{\mathrm{C}}$]
\label{th:fistac}
The sequence $(\boldsymbol{v}^{[t]},\boldsymbol{h}^{[t]})$ generated by Algorithm~\ref{alg:fistac} converges towards a solution of \textit{coupled estimation problem~\eqref{eq:coupledtv}}.
\end{theorem}
\begin{proof}
\textit{The proofs of Theorems~\ref{th:fistaj}~and~\ref{th:fistac} stem directly from the choice of descent parameters proposed in Algs.~\ref{alg:fistaj}~and~\ref{alg:fistac} following the reasoning in~\cite{chambolle_how_2014}.}
\end{proof}

\begin{algorithm}
$\begin{array}{ll} \textbf{Initialization:} &\mbox{Set} \; \boldsymbol{u}^{[0]} \in \mathbb{R}^{2 \times \lvert \Upsilon \rvert}, \,\bar{ \boldsymbol{u}}^{[0]} =  \boldsymbol{u}^{[0]};\\
&\mbox{Set} \;  \boldsymbol{\ell}^{[0]} \in \mathbb{R}^{2 \times \lvert \Upsilon \rvert}, \, \bar{\boldsymbol{\ell}}^{[0]} = \boldsymbol{\ell}^{[0]}; \\
& \mbox{Let} \, \left(\mathcal{S}_{\underline{n}}, \mathcal{T}_{\underline{n}}\right) \, \mbox{defined in \eqref{eq:ST}};\\
& \mbox{Let} \, \textbf{J}\, \mbox{defined in \eqref{eq:sm}};\\
&\mbox{Set} \;   (\forall \underline{n})\; \left(v_{\underline{n}}^{[0]}, h_{\underline{n}}^{[0]}\right)^{\top}  = \textbf{J}^{-1}\left(\mathcal{S}_{\underline{n}}, \mathcal{T}_{\underline{n}}\right)^{\top};\\
&  \mbox{Set} \; b>2\;\mbox{and}\;\tau_0 = 1;\\
& \mbox{Set} \; \alpha>0\;\mbox{and}\;\lambda>0;\\
&  \mbox{Set} \; \gamma>0\, \mbox{such that}\, \gamma \, \lVert \textbf{J}^{-1} \rVert \lVert \textbf{D} \rVert^2 < 1;\\
\end{array}$\\
\For{$t \in \mathbb{N}$}{
	\,	
	\underline{\textit{Dual variable update:}}\\ 
	\,
	$\boldsymbol{u}^{[t+1]} = \mathrm{prox}_{\gamma \left( \lambda \lVert . \rVert_{2,1}\right)^*}\left( \bar{\boldsymbol{u}}^{[t]} + \gamma \textbf{D} \boldsymbol{v}^{[t]}\right)$\\
	$\boldsymbol{\ell}^{[t+1]} = \mathrm{prox}_{\gamma \left( \lambda \alpha \lVert . \rVert_{2,1}\right)^*}\left( \bar{\boldsymbol{\ell}}^{[t]} + \gamma \textbf{D} \boldsymbol{h}^{[t]}\right)$\\
	\,
	\underline{\textit{FISTA parameter update}}\\
	\,
	$\tau_{t+1} = \frac{t + b}{b}$\\
	\,
	\underline{\textit{Auxiliary variable update}}\\
	\,
	$\bar{\boldsymbol{u}}^{[t+1]} = \boldsymbol{u}^{[t+1]} + \frac{\tau_t - 1}{\tau_{t+1}} \left( \boldsymbol{u}^{[t+1]} - \boldsymbol{u}^{[t]}\right)$\\
	\,	
	$\bar{\boldsymbol{\ell}}^{[t+1]} = \boldsymbol{\ell}^{[t+1]} + \frac{\tau_t - 1}{\tau_{t+1}} \left( \boldsymbol{\ell}^{[t+1]} - \boldsymbol{\ell}^{[t]}\right)$\\
	\,
	\underline{\textit{Primal variable update}}\\
	\,
	$\begin{pmatrix}
	\boldsymbol{v}^{[t+1]}\\
	\boldsymbol{h}^{[t+1]}
\end{pmatrix} = \begin{pmatrix}
\boldsymbol{v}^{[t]}\\
\boldsymbol{h}^{[t]}
\end{pmatrix} - \textbf{J}^{-1} \begin{pmatrix}
\textbf{D}^*(\boldsymbol{u}^{[t+1]}-\boldsymbol{u}^{[t]})\\
\textbf{D}^*(\boldsymbol{\ell}^{[t+1]} - \boldsymbol{\ell}^{[t]})
\end{pmatrix}$
    }
	\caption{\label{alg:fistaj} $\mathrm{FISTA}_{\mathrm{J}}$: Joint estimation (Pb.~\eqref{eq:jointtv})}
\end{algorithm}

\begin{algorithm}
$\begin{array}{ll} \textbf{Initialization:} &\mbox{Set} \; \boldsymbol{u}^{[0]} \in \mathbb{R}^{2 \times \lvert \Upsilon \rvert}, \,\bar{ \boldsymbol{u}}^{[0]} =  \boldsymbol{u}^{[0]};\\
&\mbox{Set} \;  \boldsymbol{\ell}^{[0]} \in \mathbb{R}^{2 \times \lvert \Upsilon \rvert}, \, \bar{\boldsymbol{\ell}}^{[0]} = \boldsymbol{\ell}^{[0]}; \\
& \mbox{Let} \, \left(\mathcal{S}_{\underline{n}}, \mathcal{T}_{\underline{n}}\right) \, \mbox{defined in \eqref{eq:ST}};\\
& \mbox{Let} \, \textbf{J}\, \mbox{defined in \eqref{eq:sm}};\\
&\mbox{Set} \;   (\forall \underline{n})\; \left(v_{\underline{n}}^{[0]}, h_{\underline{n}}^{[0]}\right)^{\top}  = \textbf{J}^{-1}\left(\mathcal{S}_{\underline{n}}, \mathcal{T}_{\underline{n}}\right)^{\top};\\
&  \mbox{Set} \; b>2\;\mbox{and}\;\tau_0 = 1;\\
& \mbox{Set} \; \alpha>0\;\mbox{and}\;\lambda>0;\\
&  \mbox{Set} \; \gamma>0\, \mbox{s. t.}\, \gamma  \max( 1, \alpha)  \, \lVert \textbf{J}^{-1} \rVert \lVert \textbf{D} \rVert^2 < 1;\\
\end{array}$\\
\For{$t \in \mathbb{N}$}{
	\,	
	\underline{\textit{Dual variable update:}}\\ 
	\,
	$\begin{pmatrix}
	 \boldsymbol{u}^{[t+1]} \\
	 \boldsymbol{\ell}^{[t+1]}
	\end{pmatrix} = \mathrm{prox}_{\gamma \left( \lambda \lVert . \rVert_{2,1}\right)^*} \begin{pmatrix}
	\bar{ \boldsymbol{u}}^{[t]} + \gamma \textbf{D}  \boldsymbol{v}^{[t]}\\
	\bar{ \boldsymbol{\ell}}^{[t]} + \gamma \alpha \textbf{D}  \boldsymbol{h}^{[t]}
	\end{pmatrix}  $\\
	\,
	\underline{\textit{FISTA parameter update}}\\
	\,
	$\tau_{t+1} = \frac{t + b}{b}$\\
	\,
	\underline{\textit{Auxiliary variable update}}\\
	\,
	$\bar{ \boldsymbol{u}}^{[t+1]} =  \boldsymbol{u}^{[t+1]} + \frac{\tau_t - 1}{\tau_{t+1}} \left(  \boldsymbol{u}^{[t+1]} -  \boldsymbol{u}^{[t]}\right)$\\
	\,
	$\bar{ \boldsymbol{\ell}}^{[t+1]} =  \boldsymbol{\ell}^{[t+1]} + \frac{\tau_t - 1}{\tau_{t+1}} \left(  \boldsymbol{\ell}^{[t+1]} -  \boldsymbol{\ell}^{[t]}\right)$\\
	\,
	\underline{\textit{Primal variable update}}\\
	\,
	$\begin{pmatrix}
	\boldsymbol{v}^{[t+1]}\\
	\boldsymbol{h}^{[t+1]}
\end{pmatrix} = \begin{pmatrix}
\boldsymbol{v}^{[t]}\\
\boldsymbol{h}^{[t]}
\end{pmatrix} - \textbf{J}^{-1} \begin{pmatrix}
\textbf{D}^*(\boldsymbol{u}^{[t+1]}-\boldsymbol{u}^{[t]})\\
\alpha\textbf{D}^*(\boldsymbol{\ell}^{[t+1]} - \boldsymbol{\ell}^{[t]})
\end{pmatrix}$
    }
	\caption{\label{alg:fistac} $\mathrm{FISTA}_{\mathrm{C}}$: Coupled estimation (Pb.~\eqref{eq:coupledtv})}
\end{algorithm}

\subsubsection{Primal-dual} Primal-dual algorithms~\cite{chambolle2011first, condat_primal-dual_2013, Vu_B_2013_j-acm_spl_adm, Komodakis_N_2015_j-ieee-spm_pla_dor} can also be customized to solve Problems~\eqref{eq:jointtv}~and~\eqref{eq:coupledtv}. 

This requires the derivation of a closed form expression for the proximal operator associated with the quadratic data fidelity term $\Phi$, provided in Proposition~\ref{prop:proxf}.
\begin{proposition}[Computation of $\textrm{prox}_{\Phi}$]
\label{prop:proxf}
For every $(\boldsymbol{v}, \boldsymbol{h}) \in \mathbb{R}^{\vert \Upsilon \vert} \times \mathbb{R}^{\vert \Upsilon \vert}$, denoting $\left(\boldsymbol{p}, \boldsymbol{q}\right)=\mathrm{prox}_{\Phi}(\boldsymbol{v},\boldsymbol{h}) \in \mathbb{R}^{\vert \Upsilon \vert} \times \mathbb{R}^{\vert \Upsilon \vert}$ one has
\begin{equation*}
\left\lbrace \begin{array}{ll}
\boldsymbol{p} =\frac{ (1+R_2)(\boldsymbol{\mathcal{S}} + \boldsymbol{v}) - R_1 (\boldsymbol{\mathcal{T}} +\boldsymbol{h})}{(1 +
R_0)(1+R_2) - R_1^2},\\
\boldsymbol{q} = \frac{ (1+R_0)(\boldsymbol{\mathcal{T}} + \boldsymbol{h}) - R_1 (\boldsymbol{\mathcal{S}}+\boldsymbol{v}) }{(1 +
R_0)(1+R_2) - R_1^2}.
\end{array} \right.
\end{equation*}
with $\boldsymbol{\mathcal{S}}$ and $\boldsymbol{\mathcal{T}}$ 
defined in \eqref{eq:ST} 
and  $R_0, R_1, R_2$ defined in \eqref{eq:sm}.
\end{proposition} 

Further, in~\cite{chambolle2011first}, it was described how to take advantage of the strong convexity\footnote{A function $\varphi : \mathcal{H} \rightarrow \mathbb{R}$, defined on an Hilbert space $\mathcal{H}$, is said to be $\mu-$strongly convex, for a given $\mu > 0$ if the function $y \mapsto \varphi(y) - \frac{\mu}{2}\lVert y \rVert_2^2$ is convex. When the function $\varphi$ is differentiable, $\varphi$ is $\mu-$strongly convex if and only if $
\left( \forall (y,z) \in \mathcal{H} \times \mathcal{H} \right) \quad \langle \nabla \varphi (y) - \nabla \varphi (z), y - z  \rangle \geq \mu \lVert y - z\rVert^2 $
where $\langle \cdot \, , \cdot\rangle$ denotes the Hilbert scalar product and $\lVert \cdot \rVert$ the associated scalar product.} of the objective function to obtain fast implementation for primal-dual algorithms, with linear convergence rates.
In Proposition~\ref{prop:convex} below, we prove that the data fidelity term $\Phi$ is strongly convex and derive the closed form expression of the strong convexity coefficient. 
\begin{proposition}
\label{prop:convex}
Function $ \Phi(\boldsymbol{v},\boldsymbol{h};\mathcal{L} )$ is $\mu$-strongly convex w.r.t the variables $(\boldsymbol{v},\boldsymbol{h})$, with $\mu = \chi$ where $\chi>0$ is the lowest eigenvalue of the symmetric and positive definite matrix $\textbf{J}$ defined in Eq.~\eqref{eq:sm}.
\end{proposition}
Indeed, since $\nabla \Phi(\boldsymbol{v},\boldsymbol{h};\mathcal{L} ) = \textbf{J} \left(
\boldsymbol{v},\boldsymbol{h}
\right)^{\top}$ is linear, the condition for $\Phi$ to be $\mu-$strongly convex can be recasted as: 
\begin{align}
 \langle \nabla \Phi(\boldsymbol{v},\boldsymbol{h};\mathcal{L}) , (\boldsymbol{v},\boldsymbol{h})  \rangle \, &= \, \left\langle \textbf{J}\left(
\boldsymbol{v},\boldsymbol{h} \right)^{\top}, \left(\boldsymbol{v},\boldsymbol{h} \right)^{\top} \right\rangle \,\nonumber\\
&\geq \, \mu \left\Vert \left(
\boldsymbol{v},\boldsymbol{h}
\right) \right\Vert^2
\end{align}
Intuitively a function with a large strong-convexity constant $\mu$ is very stitched around its minimum, hence yielding a good theoretical convergence rate toward the minimizer. 
The strong-convexity parameter is represented in Fig.~\ref{fig:mu_j1_j2} as a function of the range of scales involved in the estimation procedure (cf. \eqref{eq:defF}). 

\begin{figure}[h!]
\centering
\includegraphics[width = 0.5\linewidth]{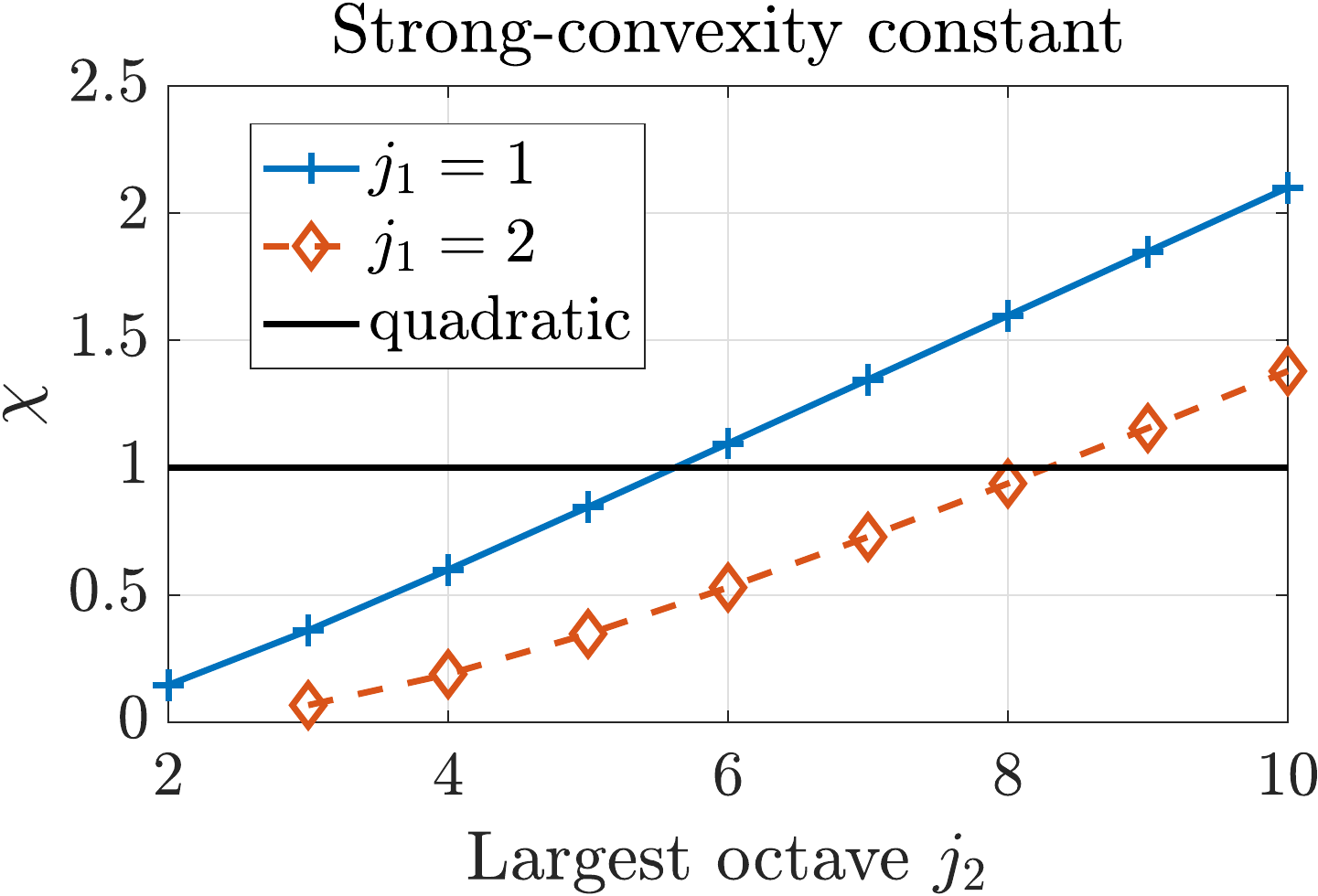}
\caption{\label{fig:mu_j1_j2} {\bf Strong convexity constant} as a function of the range of scales involved in the estimation.}
\end{figure}

The iterations of the devised fast primal-dual algorithms are detailed in Algorithms~\ref{alg:pdj} and \ref{alg:pdc}. 
Convergence guarantees are specified in Theorem~\ref{th:pdj} (resp. Theorem~\ref{th:pdc}).
\begin{theorem}[Convergence of $\mathrm{PD}_{\mathrm{J}}$]
\label{th:pdj}
The sequence $(\boldsymbol{v}^{[\ell]},\boldsymbol{h}^{[\ell]})$ generated by Algorithm~\ref{alg:pdj} converges towards a solution of the \textit{joint estimation problem~\eqref{eq:jointtv}}.
\end{theorem}
\begin{theorem}[Convergence of $\mathrm{PD}_{\mathrm{C}}$]
\label{th:pdc}
The sequence $(\boldsymbol{v}^{[\ell]},\boldsymbol{h}^{[\ell]})$ generated by Algorithm~\ref{alg:pdc} converges towards a solution of the \textit{joint estimation problem~\eqref{eq:coupledtv}}.
\end{theorem}
\begin{proof}
\textit{The proofs of Theorems~\ref{th:pdj}~and~\ref{th:pdc} come directly from the choice of descent parameters proposed in Algs.~\ref{alg:pdj}~and~\ref{alg:pdc} following the reasoning of~\cite{chambolle2011first}.\\}
\end{proof}

\begin{algorithm}
$\begin{array}{ll} \textbf{Initialization:} &\mbox{Set} \; \boldsymbol{v}^{[0]} \in \mathbb{R}^{\lvert \Upsilon \rvert}, \,\boldsymbol{u}^{[0]}  =  \textbf{D}\boldsymbol{v}^{[0]}, \, \bar{ \boldsymbol{u}}^{[0]} =  \boldsymbol{u}^{[0]};\\
&\mbox{Set} \; \boldsymbol{h}^{[0]} \in \mathbb{R}^{\lvert \Upsilon \rvert}, \,\boldsymbol{\ell}^{[0]}  =  \textbf{D}\boldsymbol{h}^{[0]}, \, \bar{\boldsymbol{\ell}}^{[0]} = \boldsymbol{\ell}^{[0]}; \\
& \mbox{Set} \; \alpha>0\;\mbox{and}\;\lambda>0;\\
&\mbox{Set} \; (\delta_0 ,\nu_0)\, \mbox{such that}\, \delta_0 \nu_0 \lVert \textbf{D} \rVert^2 < 1;\\\\
\end{array}$\\
\vspace{0.3cm}

\For{$t \in \mathbb{N}^*$}{
\,
\underline{\textit{Primal variable update:}}\\
\,
$\begin{pmatrix}
\boldsymbol{v}^{[t+1]} \\
\boldsymbol{h}^{[t+1]}
\end{pmatrix} = \mathrm{prox}_{\delta_t \Phi} \left( \begin{pmatrix}
\boldsymbol{v}^{[t]} \\
\boldsymbol{h}^{[t]}
\end{pmatrix} - \delta_t \begin{pmatrix}
\textbf{D}^* \bar{\boldsymbol{u}}^{[t]}\\
\textbf{D}^* \bar{\boldsymbol{\ell}}^{[t]}
\end{pmatrix} \right)$\\
\,
\underline{\textit{Dual variable update:}}\\
\,
$\boldsymbol{u}^{[t+1]} = \mathrm{prox}_{\nu_t \left( \lambda \lVert \cdot \rVert_{2,1}\right)^*} \left( \boldsymbol{u}^{[t]} + \nu_t \textbf{D} \boldsymbol{v}^{[t+1]} \right)$ \\
\,
$\boldsymbol{\ell}^{[t+1]} = \mathrm{prox}_{\nu_t \left( \lambda \alpha \lVert \cdot \rVert_{2,1}\right)^*} \left( \boldsymbol{\ell}^{[t]} + \nu_t \textbf{D} \boldsymbol{h}^{[t+1]} \right)$ \\
\,
\underline{\textit{Descent steps update:}}\\
\,
$\vartheta_t = \left( 1+ 2 \mu \delta_t \right)^{-1/2}$, $\delta_{t+1} = \vartheta_t \delta_t, \, \nu_{t+1} = \nu_t/\vartheta_t$\\
\,
\underline{\textit{Auxiliary variable update:}}\\
\,
$\begin{pmatrix}
\bar{\boldsymbol{u}}^{[t+1]} \\
\bar{\boldsymbol{\ell}}^{[t+1]}
\end{pmatrix} =  \begin{pmatrix}
\boldsymbol{u}^{[t+1]} \\
\boldsymbol{\ell}^{[t+1]}
\end{pmatrix} + \vartheta_t \left(\begin{pmatrix}
\boldsymbol{u}^{[t+1]} \\
\boldsymbol{\ell}^{[t+1]}
\end{pmatrix} - \begin{pmatrix}
\boldsymbol{u}^{[t]} \\
\boldsymbol{\ell}^{[t]}
\end{pmatrix}  \right)$}
\caption{\label{alg:pdj}  $\mathrm{PD}_{\mathrm{J}}$: Joint estimation (Pb.~\eqref{eq:jointtv})}
\end{algorithm}

\begin{algorithm}
$\begin{array}{ll} \textbf{Initialization:} \!\!\!\!\!\!\!\!\!\!\!\!\!\!\!\!\!\!\!\!\!\!\!\!\!\!\!\!&\\
&\mbox{Set} \; \boldsymbol{v}^{[0]} \in \mathbb{R}^{\lvert \Upsilon \rvert}, \,\boldsymbol{u}^{[0]}  =  \textbf{D}\boldsymbol{v}^{[0]}, \, \bar{ \boldsymbol{u}}^{[0]} =  \boldsymbol{u}^{[0]};\\
&\mbox{Set} \; \boldsymbol{h}^{[0]} \in \mathbb{R}^{\lvert \Upsilon \rvert}, \,\boldsymbol{\ell}^{[0]}  =  \alpha\textbf{D}\boldsymbol{h}^{[0]}, \, \bar{\boldsymbol{\ell}}^{[0]} = \boldsymbol{\ell}^{[0]}; \\
& \mbox{Set} \; \alpha>0\;\mbox{and}\;\lambda>0.\\
&\mbox{Set} \; (\delta_0 ,\nu_0)\, \mbox{such that}\, \delta_0 \nu_0 \max(1,\alpha)\lVert \textbf{D} \rVert^2 < 1;\\\\
\end{array}$\\

\For{$t \in \mathbb{N}^*$}{
\,
\underline{\textit{Primal variable update:}}\\
\,
$\begin{pmatrix}
\boldsymbol{v}^{[t+1]} \\
\boldsymbol{h}^{[t+1]}
\end{pmatrix} = \mathrm{prox}_{\delta_t \Phi} \left( \begin{pmatrix}
\boldsymbol{v}^{[t]} \\
\boldsymbol{h}^{[t]}
\end{pmatrix} - \delta_t \begin{pmatrix}
\textbf{D}^* \bar{\boldsymbol{u}}^{[t]}\\
\alpha \textbf{D}^* \bar{\boldsymbol{\ell}}^{[t]}
\end{pmatrix} \right)$\\
\,
\underline{\textit{Dual variable update:}}\\
\,
$\begin{pmatrix}
	\boldsymbol{u}^{[t+1]} \\
	\boldsymbol{\ell}^{[t+1]}
	\end{pmatrix} = \mathrm{prox}_{\nu_t \left( \lambda \lVert . \rVert_{2,1}\right)^*} \begin{pmatrix}
	\boldsymbol{u}^{[t]} + \nu_t \textbf{D} \boldsymbol{v}^{[t+1]}\\
	\boldsymbol{\ell}^{[t]} + \nu_t \alpha \textbf{D} \boldsymbol{h}^{[t+1]}
	\end{pmatrix}  $\\
\,
\underline{\textit{Descent steps update:}}\\
\,
$\vartheta_t = \left( 1+ 2 \mu \delta_t \right)^{-1/2}$, $\delta_{t+1} = \vartheta_t \delta_t, \, \nu_{t+1} = \nu_t/\vartheta_t$\\
\,
\underline{\textit{Auxiliary variable update:}}\\
\,
$\begin{pmatrix}
\bar{\boldsymbol{u}}^{[t+1]} \\
\bar{\boldsymbol{\ell}}^{[t+1]}
\end{pmatrix} =  \begin{pmatrix}
\boldsymbol{u}^{[t+1]} \\
\boldsymbol{\ell}^{[t+1]}
\end{pmatrix} + \vartheta_t \left(\begin{pmatrix}
\boldsymbol{u}^{[t+1]} \\
\boldsymbol{\ell}^{[t+1]}
\end{pmatrix} - \begin{pmatrix}
\boldsymbol{u}^{[t]} \\
\boldsymbol{\ell}^{[t]}
\end{pmatrix}  \right)$}
\caption{\label{alg:pdc}  $\mathrm{PD}_{\mathrm{C}}$: Coupled estimation (Pb.~\eqref{eq:coupledtv})}
\end{algorithm}

\subsection{Duality gap}
\label{subsec:gap}
To ensure fair comparisons between the different algorithms, we construct a stopping criterion based on the \emph{duality gap}. 
Let $\mathcal{H}, \mathcal{G}$ Hilbert spaces, $\Theta : \mathcal{H} \rightarrow ]-\infty, +\infty]$, $\Xi : \mathcal{G}  \rightarrow ]-\infty, +\infty]$ and $\textbf{L} : \mathcal{H} \rightarrow \mathcal{G}$ a bounded linear operator. From a primal optimization problem of the general form $\widehat{\boldsymbol{x}} = \underset{\boldsymbol{x} \in \mathcal{H}}{\mathrm{\arg\min}} \, \,\Theta(\boldsymbol{x}) + \Xi(\textbf{L}\boldsymbol{x})$,
and its associated dual problem $\widehat{\boldsymbol{y}} = \underset{\boldsymbol{y} \in \mathcal{G}}{\mathrm{\arg\min}} \, \, \Theta^*(-\textbf{L}^*\boldsymbol{y}) + \Xi^*(\boldsymbol{y})$, we define:
\begin{align}
\label{eq:defGamma}
\Gamma(\boldsymbol{x};\boldsymbol{y}) := \Theta(\boldsymbol{x}) + \Xi(\textbf{L}\boldsymbol{x}) + \Theta^*(-\textbf{L}^*\boldsymbol{y}) + \Xi^*(\boldsymbol{y}),
\end{align}
$\Theta^*$ (resp. $\Xi^*$), being the Fenchel conjugate of $\Theta$ (resp. $\Xi$).
\begin{definition}
The \emph{duality gap} $\delta \Gamma$ associated to primal and dual optimization problems is defined as the infimum:
\begin{align}
\label{eq:definf}
\delta \Gamma := \underset{(\boldsymbol{x},\boldsymbol{y}) \in \mathcal{H}\times \mathcal{G}}{\inf} \Gamma(\boldsymbol{x}; \boldsymbol{y}).
\end{align}
\end{definition}
Under some loose assumptions on $\Theta$, $\Xi$ and $\textbf{L}$, referred as \emph{strong duality} in~\cite{Bauschke_H_2011_book_con_amo}, the duality gap $\delta\Gamma$ is zero and the respective solutions $\widehat{\boldsymbol{x}}$ and $\widehat{\boldsymbol{y}}$ of primal and dual problems are characterized by $\Gamma(\widehat{\boldsymbol{x}}; \widehat{\boldsymbol{y}}) = \delta\Gamma = 0$.

Primal-dual algorithms presented above built a minimizing sequence $(\boldsymbol{x}^{[t]}, \boldsymbol{y}^{[t]})_{t\in\mathbb{N}}$ for $\Gamma$,
therefore converging towards the unique point achieving the infimum of Eq.~\eqref{eq:definf}: $(\widehat{\boldsymbol{x}}, \widehat{\boldsymbol{y}})$. Then 
the convergence of $\boldsymbol{x}^{[t]}$ (resp. $\boldsymbol{y}^{[t]}$) towards $\widehat{\boldsymbol{x}}$ (resp. $\widehat{\boldsymbol{y}}$) can be measured evaluating the primal-dual functional $\Gamma$, as:
\begin{align}
\Gamma(\boldsymbol{x}^{[t]}; \boldsymbol{y}^{[t]}) \underset{t \rightarrow \infty}{\longrightarrow} 0.
\end{align}

The ``\textit{joint}"~\eqref{eq:jointtv} and ``\textit{coupled}"~\eqref{eq:coupledtv} optimization problems, with primal variable $\boldsymbol{x} = (\boldsymbol{v}, \boldsymbol{h}) \in \mathbb{R}^{ \lvert \Upsilon \rvert} \times \mathbb{R}^{ \lvert \Upsilon \rvert}$, share the same data fidelity term $\Theta (\boldsymbol{x}) = \Phi(\boldsymbol{v}, \boldsymbol{h}; \boldsymbol{\mathcal{L}})$ and a penalization based on the mixed norm $\Xi = \lambda \lVert \cdot \rVert_{2,1}$. Yet, the linear operator \textbf{L} differ between \textit{joint} and \textit{coupled} formulations, as:\\ 
\noindent \textit{i)} \,$\textbf{L}\boldsymbol{x} = [\textbf{D}\boldsymbol{v}, \alpha \textbf{D}\boldsymbol{h}] \in \mathbb{R}^{2 \times 2\lvert \Upsilon \rvert}$, for ``\textit{joint}"~\eqref{eq:jointtv},\\
\noindent \textit{ii)} $\textbf{L}\boldsymbol{x} = [\textbf{D}\boldsymbol{v} ; \alpha \textbf{D}\boldsymbol{h}] \in \mathbb{R}^{4 \times \lvert \Upsilon \rvert}$, for ``\textit{coupled}"~\eqref{eq:coupledtv}.\\
For Pb.~\eqref{eq:jointtv} (resp.~\eqref{eq:coupledtv}), denoting $\boldsymbol{y} = (\boldsymbol{u}, \boldsymbol{\ell})$ the dual variable, the evaluation of $\Gamma_{\mathrm{J \, (resp. \, C)}}(\boldsymbol{v}, \boldsymbol{h}; \boldsymbol{u}, \boldsymbol{\ell})$ requires computing: \\
\noindent \textit{i)} \, \,$\textbf{L}^*$ (straightforward from $\textbf{D}^*$),\\
\noindent \textit{ii)} \, $\Xi^*=(\lambda \lVert \cdot \rVert_{2,1})^* = \iota_{\Vert \cdot \Vert_{2,+\infty}\leq \lambda}$ (direct computation),\\
\noindent \textit{iii)} $\Theta^* = \Phi^*$ which is devised in Proposition~\ref{prop:ccphi}.
\begin{proposition}
\label{prop:ccphi}
Let $\textbf{J}$ and $\left(\boldsymbol{\mathcal{S}}, \boldsymbol{\mathcal{T}}\right)$ defined in \eqref{eq:sm} and \eqref{eq:ST},
\begin{align}
\label{eq:phi_conj}
\Phi^*(\boldsymbol{v}, \boldsymbol{h}; \boldsymbol{\mathcal{L}}) &= \frac{1}{2} \langle (\boldsymbol{v}, \boldsymbol{h})^\top, \textbf{J}^{-1}(\boldsymbol{v}, \boldsymbol{h})^\top \rangle  + \langle (\boldsymbol{\mathcal{S}},\boldsymbol{\mathcal{T}})^\top, \textbf{J}^{-1}(\boldsymbol{v}, \boldsymbol{h})^\top \rangle + \boldsymbol{\mathcal{C}},  \\
\boldsymbol{\mathcal{C}} &= \frac{1}{2} \langle (\boldsymbol{\mathcal{S}}, \boldsymbol{\mathcal{T}})^\top ,  \textbf{J}^{-1}(\boldsymbol{\mathcal{S}}, \boldsymbol{\mathcal{T}})^\top \rangle - \frac{1}{2} \sum_j (\log_2 \boldsymbol{\mathcal{L}}_{j} )^2. \nonumber
\end{align}
\end{proposition}
\begin{proof}
By definition of the Fenchel conjugate, 
\begin{align}
F^*(\boldsymbol{v},\boldsymbol{h};\boldsymbol{\mathcal{L}}) = \underset{\widetilde{\boldsymbol{v}} \in \mathbb{R}^{\lvert \Upsilon \rvert},\widetilde{\boldsymbol{h}} \in \mathbb{R}^{\lvert \Upsilon \rvert}}{\sup} \langle \widetilde{\boldsymbol{v}},\boldsymbol{v} \rangle + \langle \widetilde{\boldsymbol{h}}, \boldsymbol{h} \rangle - F(\widetilde{\boldsymbol{v}},\widetilde{\boldsymbol{h}};\boldsymbol{\mathcal{L}}).
\end{align}
The supremum is obtained at $(\bar{\boldsymbol{v}},\bar{\boldsymbol{h}})$ such that, for every $\underline{n}\in \Omega$,
\begin{align}
\begin{cases}
 v_{\underline{n}} - \sum_j \left( \bar{v}_{\underline{n}} + j\bar{h}_{\underline{n}} - \log_2 \mathcal{L}_{j,\underline{n}} \right) = 0 \\
h_{\underline{n}} - \sum_j j\left( \bar{v}_{\underline{n}} + j\bar{h}_{\underline{n}} - \log_2 \mathcal{L}_{j,\underline{n}} \right) = 0.
\end{cases}
\end{align}
or equivalently,
\begin{align}
\left\lbrace \begin{array}{l}
R_0 \bar{v}_{\underline{n}} + R_1 \bar{h}_{\underline{n}} = v_{\underline{n}} + \mathcal{S}_{\underline{n}} \\
R_1 \bar{v}_{\underline{n}} + R_2 \bar{h}_{\underline{n}} = h_{\underline{n}} + \mathcal{T}_{\underline{n}} \\
\end{array} \right. 
\end{align}
that yields to
\begin{align}
\begin{pmatrix}
\bar{v}_{\underline{n}} \\
\bar{h}_{\underline{n}}
\end{pmatrix} = \textbf{J}^{-1} \begin{pmatrix}
v_{\underline{n}} + \mathcal{S}_{\underline{n}} \\
h_{\underline{n}} + \mathcal{T}_{\underline{n}}
\end{pmatrix}
\end{align}
and it is then necessary to re-inject this expression into the explicit expression of $\Psi$.
\end{proof}

\section{Piecewise homogeneous fractal textures}
\label{sec:synthesis}

\noindent {\bf Homogeneous fractal textures.}~ 
Numerous models of homogeneous textures were proposed in the literature, mostly consisting of 2D extensions (referred to as fractional Brownian field fBf) of fractional Brownian motion, the unique univariate Gaussian exactly selfsimilar process with stationary increments~\cite{cohen2013fractional}.
However, in most applications, real world textures are better modeled by stationary processes, hence by increments of fBf, referred to as fractional Gaussian fields. 
Further, to construct a formally relevant definition for piecewise homogeneous fractal textures, as originally proposed here, it is easier to \emph{manipulate} stationary processes. 

For Gaussian processes, space or frequency domain definitions of fBf are theoretically equivalent and several variations were proposed in either domain, mostly consisting in different tuning of the 2D extension of the fractional integration kernel underlying selfsimilar textures (cf. e.g.,~\cite{abry1996wavelet,cohen2013fractional,pereira2016dissipative}).
Yet, in practice, numerical issues needs to be accounted for, such as discrete sampling from a continuous process, or computation of integrals defined from a theoretically infinite support kernel and related border effects.
Further, fast circulant embedding matrix algorithms were proposed, that yet do not permit to reach the full range $0<H<1$ for the selfsimilarity parameter~\cite{stein2002fast}.
Therefore, we will make use of a self-customized construction that combines an effective implementation scheme with excellent theoretical and practical control of the local variance and regularity, while permitting direct extension to piecewise homogeneous construction. 

Following e.g.,~\cite{Bierme2007,Roux2013,Didier2018}, we start from an harmonizable representation of fBf $B_{\underline{n}},$ $\underline{n} \in \mathbb{R}^{\lvert \Upsilon \rvert}$, $\Upsilon = \lbrace 1, \hdots, N \rbrace^2$: 
\begin{equation}
\label{eq:fbf}
B_{\underline{n}} = \frac{\Sigma}{C(H)} \int \frac{\mathrm{e}^{-\mathrm{i}\underline{f} . \underline{n}}- 1}{\lVert \underline{f} \rVert^{H+1}} \mathrm{d}\widetilde{G}(\underline{f}),
\end{equation}
with $\mathrm{d}\widetilde{G}(\underline{f})$ the Fourier transform of a white Gaussian noise and 
\begin{align*}
C(H) &=
\frac{\pi^{1/2} \Gamma (H+ 1/2)}{2^{d/2}H\Gamma(2H) \sin(\pi H) \Gamma(H + d/2)}.
\end{align*}
The covariance function of the zero mean process $B_{\underline{n}}$ reads: 
\begin{align}
\label{eq:cov_fbm}
\mathbb{E} \left[ B_{\underline{n}} B_{\underline{m}}\right] = \frac{\Sigma^2}{2} \left( \lVert\underline{n} \rVert^{2H} +  \lVert\underline{m} \rVert^{2H} - \lVert \underline{n} - \underline{m} \rVert^{2H}\right). 
\end{align}

Texture $Y_{\underline{n}} $ is constructed from increments of $B_{\underline{n}}$ as (with $\textbf{e}_1$,  $\textbf{e}_2$, unitary vectors in horizontal and vertical directions): 
\begin{equation}
\label{eq:propfield}
Y_{\underline{n}} = \frac{\Sigma}{2 \delta^H \sqrt{ 1 - 2^{H-2}}}( \underbrace{B_{\underline{n}+ \delta \textbf{e}_1} - B_{\underline{n}} }_{\text{horizontal increment}} + \underbrace{ B_{\underline{n}+ \delta \textbf{e}_2} - B_{\underline{n}} }_{\text{vertical increment}}),
 \end{equation}
	
\begin{proposition}
The field $Y = (Y_{\underline{n}})_{\underline{n}\in \Upsilon}$ 
is a zero-mean Gaussian process with variance $ \mathbb{E} \left[Y_{\underline{n}}^2\right] = \Sigma^2$ and  covariance:\\
$ \displaystyle
\mathbb{E} \left[ Y_{\underline{n}+\Delta \underline{n}}Y_{\underline{n}} \right] = \frac{\Sigma^2 \delta^{-2H}}{4 - 2^H }( \lVert  \Delta \underline{n}  + \delta \textbf{e}_1 \rVert^{2H}\\
\vspace{1mm}
 + \lVert  \Delta \underline{n}  - \delta \textbf{e}_1 \rVert^{2H} + \lVert  \Delta \underline{n}  + \delta \textbf{e}_2 \rVert^{2H}
 + \lVert  \Delta \underline{n}  - \delta \textbf{e}_2 \rVert^{2H} - 3  \lVert  \Delta \underline{n}  \rVert^{2H}\\
 \vspace{1mm}
 - \frac{1}{2} \lVert  \Delta \underline{n}  + \delta \textbf{e}_1 -  \delta \textbf{e}_2 \rVert^{2H}
 - \frac{1}{2} \lVert   \Delta \underline{n}  - \delta \textbf{e}_1 +  \delta \textbf{e}_2 \rVert^{2H}).
$
\end{proposition}
\vspace{1mm}
\noindent {\bf Piecewise homogeneous fractal textures.}~ Piecewise homogeneous fractal textures $X= (X_{\underline{n}})_{\underline{n}\in \Upsilon}$ are defined as the concatenation of $M$ (pairwise disjoint) regions, denoted $(\Upsilon_m)_{0\leq m \leq M-1}$, such that $\Upsilon = \cup_m \Upsilon_m$, with textures in each region consisting of an homogeneous fractal $Y_{\underline{n}}^{(m)}$ ($0 \leq m \leq M-1$), with variance and regularity $(\Sigma_{m}^2, H_{m})$: 
\begin{align*}
X_{\underline{n}} = Y_{\underline{n}}^{(m)}, \, \, \mathrm{ when } \, \underline{n} \in \Upsilon_m.
\end{align*}
Examples of piecewise fractal textures are shown in Fig.~\ref{fig:exfGf}.

\begin{figure}[t]
\centering
\begin{tabular}{cccc}
 \includegraphics[width = 0.2\linewidth]{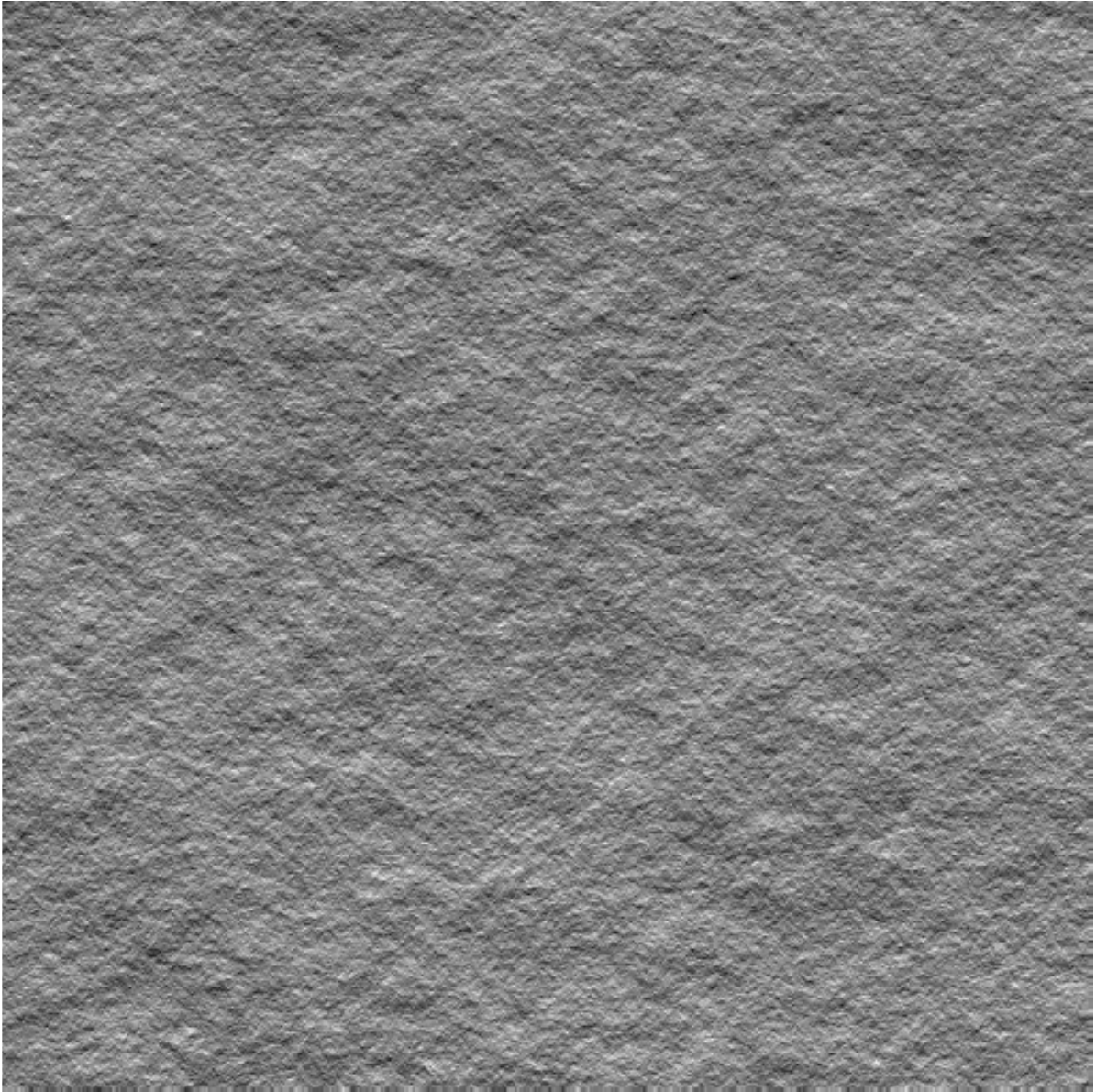}   &\includegraphics[width = 0.2\linewidth]{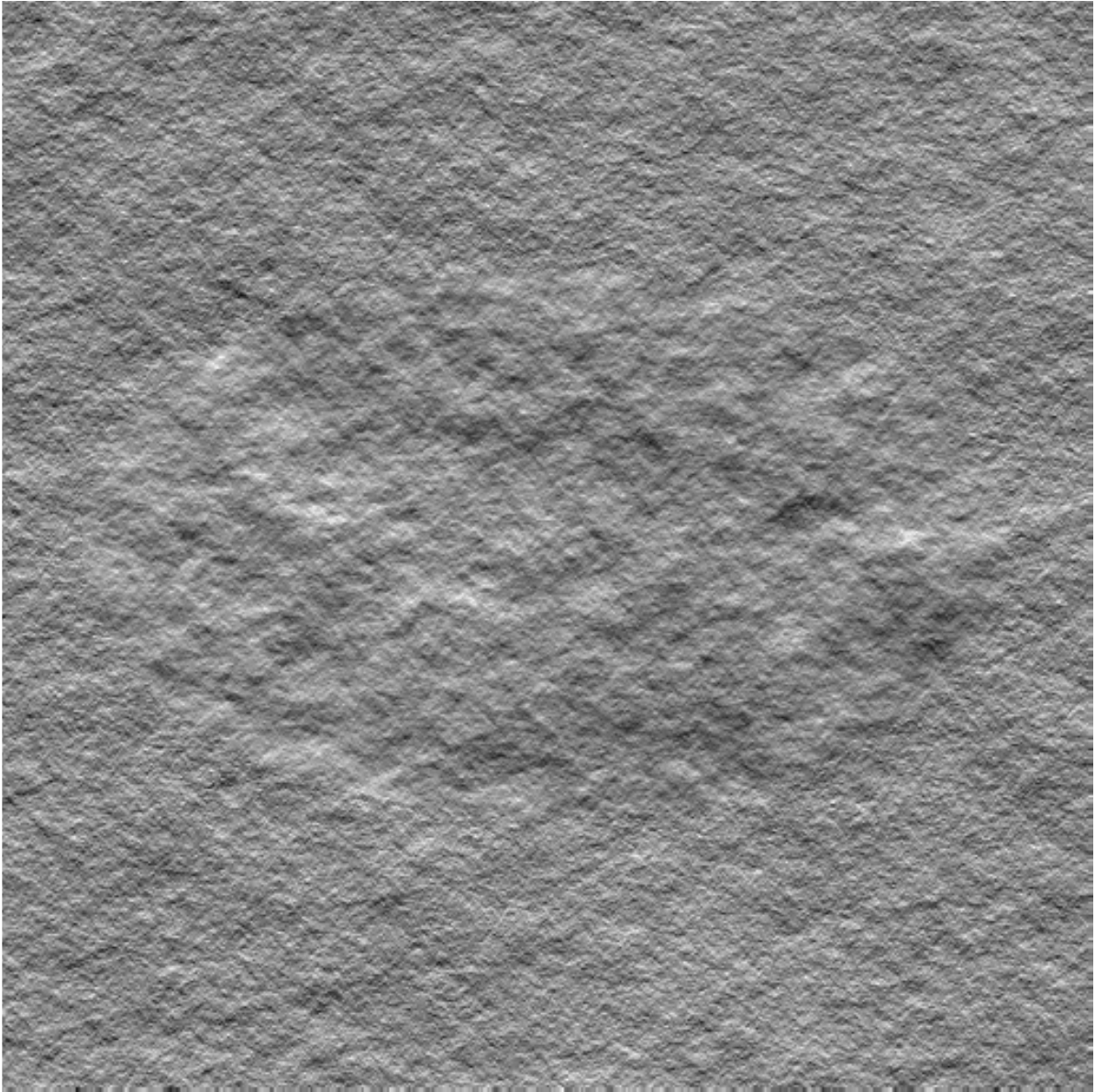} &  \includegraphics[width = 0.2\linewidth]{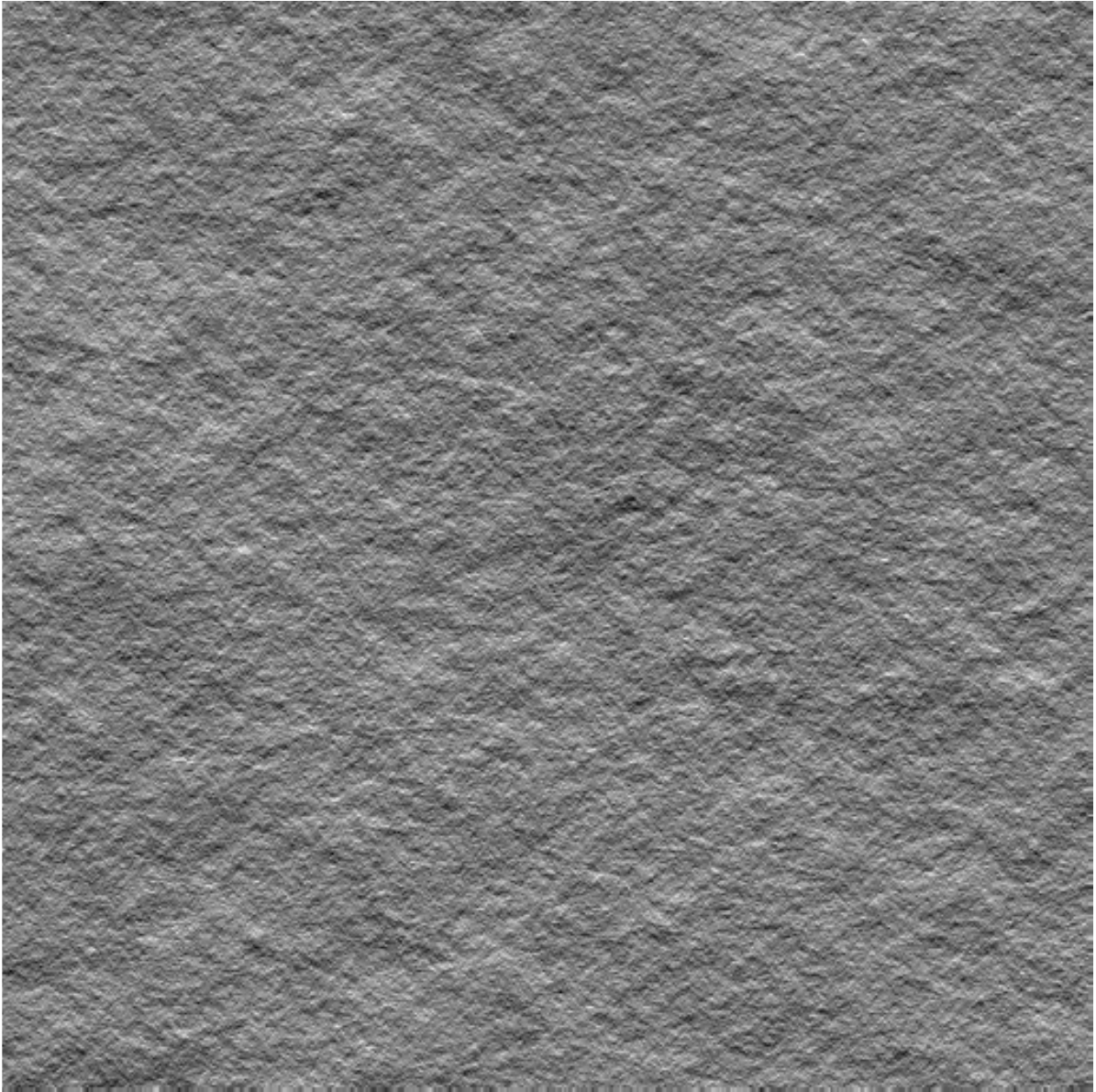}  &\includegraphics[width = 0.2\linewidth]{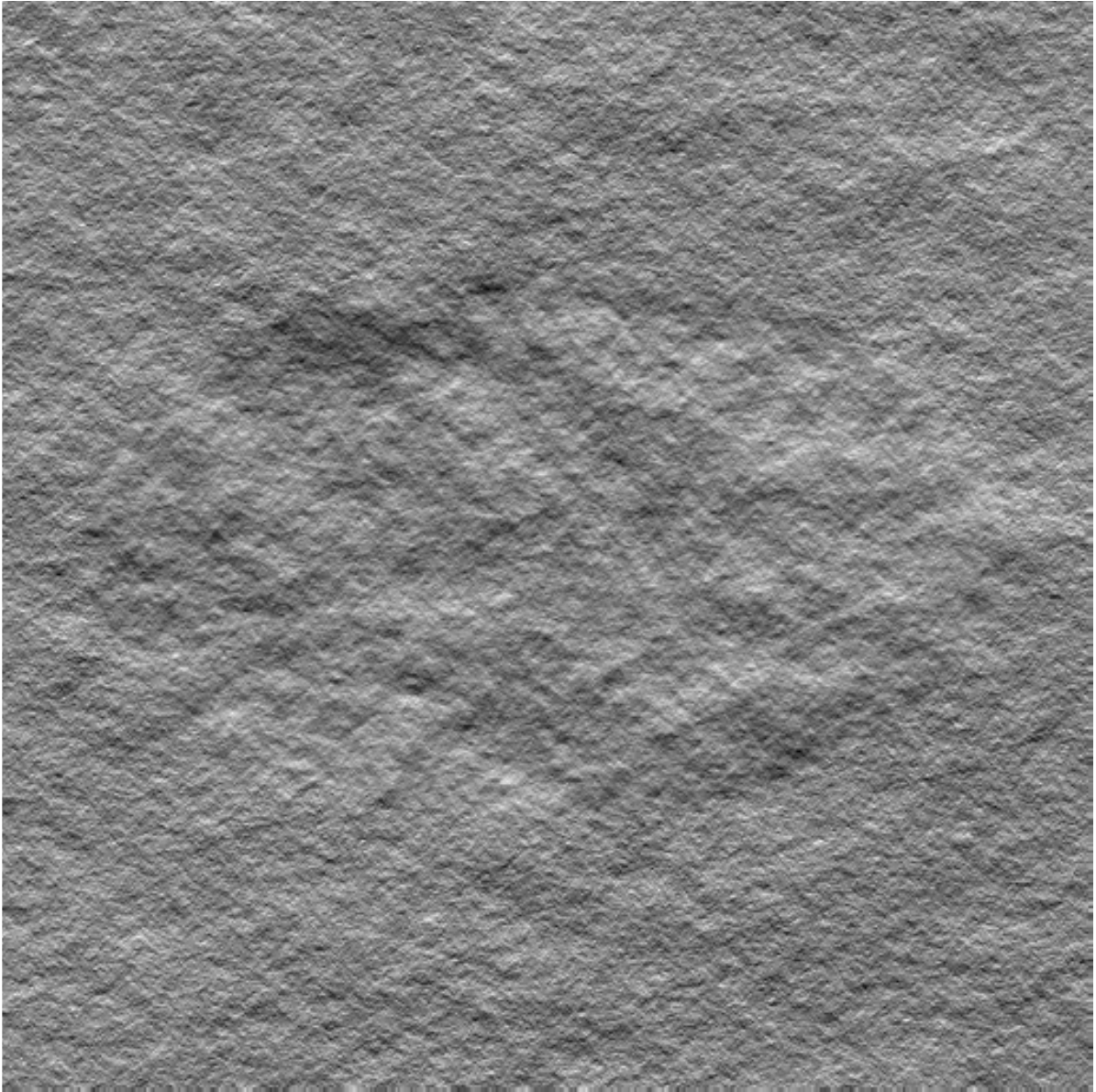} \\
\end{tabular}
\caption{{\bf Piecewise homogeneous fractal textures} $X$ generated using the mask displayed in Fig.~\ref{subfig:mask}, with parameters $(H_0, \Sigma_0)=(0.5, 0.6)$ and $(H_1,\Sigma_1))$ chosen as: (a) $(0.5, 0.6)$, no change ;  (b)  $(0.5, 0.65)$, change in  $\Sigma$ ;  (c)  $(0.8, 0.6)$ change in $H$  ; (d)  $(0.8, 0.65)$, change in $\Sigma$ and $H$. 
\label{fig:exfGf} }
\end{figure}

\section{Estimation/Segmentation performance}
\label{sec:performance}

\subsection{Monte-Carlo simulation set-up} 

The proposed \textit{joint} and \textit{coupled} segmentations are now compared in terms of performance and computational costs to the state-of-the-art T-ROF procedure, in the context of a two-region segmentation, by means of Monte-Carlo simulations.

\subsubsection{Synthetic textures}
Piecewise homogeneous fractal textures are generated as defined in Section~\ref{sec:synthesis}, using the mask shown in Fig.~\ref{subfig:mask}, with $N=512$. 
They consist of two regions: i)~a background, with variance and local regularity $\left(\Sigma^2_0, H_0\right)= (0.6,0.5)$ kept fixed, ii)~a central ellipse, for which variance and local regularity $\left(\Sigma^2_1, H_1\right)$ are varied, as illustrated in Fig.~\ref{fig:perf7config}. 
For each configuration, 5 realizations of homogeneous fractal are generated and analyzed~; performance are reported as averages over realizations.

\subsubsection{Wavelet transform} 2D-Wavelet decompositions are performed using tensor product wavelets constructed from 1D-Daubechies orthonormal least asymmetric wavelets with $N_\psi = 3$ vanishing moments~\cite{Mallat_S_1997_book_wav_tsp}.  
Wavelet leaders are computed as in Section~\ref{sec:wavleaders}.
Local estimate $\widehat{\boldsymbol{h}}_{\mathrm{LR}}$ in  T-ROF is computed as in Sec.~\ref{sec:estimation}. 
Estimates from all procedures involve octaves $(j_1,j_2) = (2,5)$. 
Octave $j_1 = 1 $ is a priori excluded as leaders are biased~\cite{Wendt2007,Wendt2009b}.

\subsubsection{Optimization algorithm parameters} To achieve best convergence of the optimization schemes, 
descent steps are chosen as large as permitted: 
\begin{itemize}
\item Alg.~\ref{alg:fistaj} ({\small FISTA$_{\mathrm{J}}$}): $\gamma = 0.99/\left(\lVert \textbf{J}^{-1} \rVert \lVert \textbf{D} \rVert^2\right)$,
\item Alg.~\ref{alg:fistac} ({\small FISTA$_{\mathrm{C}}$}): $\gamma = 0.99/\left( \max( 1, \alpha)  \, \lVert \textbf{J}^{-1} \rVert \lVert \textbf{D} \rVert^2 \right)$,
\item Alg.~\ref{alg:pdj} ({\small PD$_{\mathrm{J}}$}): $\delta_0 = \nu_0 = 0.99 /\lVert \textbf{D} \rVert$.
\item Alg.~\ref{alg:pdc} ({\small PD$_{\mathrm{C}}$}): $\delta_0 = \nu_0 = 0.99 / \left( \max(1,\sqrt{\alpha})\lVert \textbf{D} \rVert \right)$, 
\end{itemize}
where $\lVert \textbf{D} \rVert = 2\sqrt{2}$, with \textbf{D} defined in Eq.~\eqref{eq:defD}, and where $ \lVert \textbf{J}^{-1} \rVert$  depends on the octaves involved in estimation (\textbf{J} defined in Eq.~\eqref{eq:sm}):
With $ (j_1,j_2) = (2,5)$, $\lVert \textbf{J}^{-1} \rVert \simeq 2.88$.
the inertia parameter is set to $b = 4$.

\subsection{Issues in performance and algorithm comparison}

\subsubsection{Stopping criteria for proximal algorithms}

To ensure fair comparisons either between algorithms or between \emph{joint} and \emph{coupled} formulations (Pb.~\eqref{eq:jointtv}~vs. Pb.~\eqref{eq:coupledtv}), an effective stopping criterion is needed.
Fig.~\ref{fig:conv_crits} illustrates 
the convergence toward zero of two potential candidates, for the particular case, chosen as representative example, of the \emph{joint} estimation using primal-dual Alg.~\ref{alg:pdj}, and for several choices of hyperparameters. 

First, the normalized increments of objective function, as a function of the number of iterations, are observed (Fig.~\ref{fig:conv_crits}a) to decrease with an extremely irregular behavior, which makes them ill-suited to serve as effective stopping criterion.

Second, the \textit{normalized} primal-dual functional (with notations as in Sec.~\ref{subsec:gap})
\begin{align*}
\widetilde{\Gamma}^{[t]} := \frac{\Gamma^{[t]}}{\left\lvert \Theta(\boldsymbol{x}^{[t]}) + \Xi(\textbf{L}\boldsymbol{x}^{[t]}) \right\rvert + \left\lvert \Theta^*(-\textbf{L}^* \boldsymbol{y}^{[t]}) + \Xi^*(\boldsymbol{y}^{[t]}) \right\rvert},
\end{align*}
is observed (Fig.~\ref{fig:conv_crits}b) to decrease smoothly.
Systematic inspections of such decreases together with that of the corresponding achieved solutions lead us to devise a stopping criterion, which is effective both for primal-dual and forward-backward algorithms and for \emph{joint} and \emph{coupled} estimations, as,
\begin{align}
\label{eq:stop_crit}
\widetilde{\Gamma}_{\mathrm{\bullet}}^{[t]} < 5.10^{-3}, \, \mathrm{for}\, \bullet \in \lbrace \mathrm{ROF}, \, \mathrm{J} \rbrace, \, \mathrm{and} \,\widetilde{\Gamma}_{\mathrm{\mathrm{C}}}^{[t]} < 10^{-4}.
\end{align}
The use of the normalized quantities, $\widetilde{\Gamma}^{[t]}$, makes this stopping criterion robust to variations of the hyperparameters.
For practical purposes, we also impose an upper limit on the number of iterations: $t <2.5 \, 10^5$.

\begin{figure}[h!]
\centering
 \includegraphics[width = 0.75\linewidth]{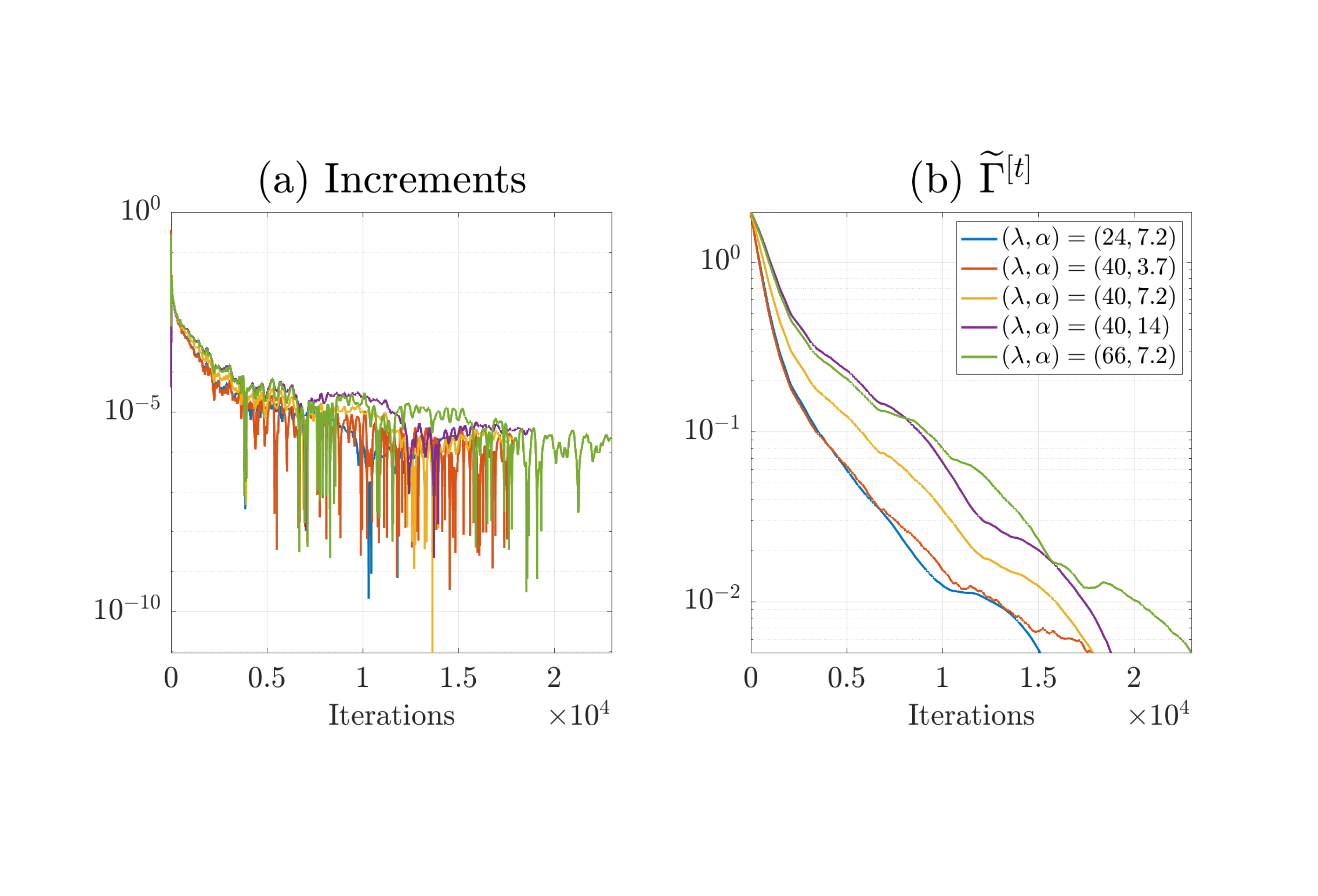}
\caption{\label{fig:conv_crits} \textbf{Stopping criteria for proximal algorithms.} Convergence of Alg.~\ref{alg:pdj} solving \textit{joint} problem~\eqref{eq:jointtv} for five pairs of hyperparameters $(\lambda, \alpha)$ evaluated with: (a)~(Normalized) increments of objective function, 
(b)~Normalized primal-dual functional $\widetilde{\Gamma}_{\mathrm{J}}^{[t]}$.}
\end{figure}

\subsubsection{Choice of regularization hyperparameters}  

The choice of regularization parameters ($\lambda$, $\alpha$) appearing in Pb~\eqref{eq:tvsep},~\eqref{eq:jointtv},~\eqref{eq:coupledtv} is of prime importance as $\lambda$ tunes the trade-off between fidelity to the fractal model~\eqref{eq:model} and expected piecewise constancy, while $\alpha$ controls the relative weight given to local wavelet log-variance $\boldsymbol{v}$ compared to local regularity $\boldsymbol{h}$, in the (\textit{joint}~\eqref{eq:deftvj} and \textit{coupled}~\eqref{eq:deftvc}) total variation penalization.
The automated choice of the regularization parameters is a difficult issue, beyond the scope of the present work. 
In this study, a grid search strategy is used to find the parameters $\lambda$ and $\alpha$ achieving the best segmentation.
In practice,  logarithmically spaced ranges are used, from $10^{-1}$ to $10^3$ for $\lambda$  and from $10^{-2}$ to $10^3$ for $\alpha$.

\subsubsection{Performance assessment}  

A natural performance criterion consists in comparing the achieved classification, denoted $\widehat{\boldsymbol{M}}_{\mathrm{ROF}}$, $\widehat{\boldsymbol{M}}_{\mathrm{J}}$ and $\widehat{\boldsymbol{M}}_{\mathrm{C}}$ respectively for the three segmentation procedures compared here, (ROF, \textit{joint} and \textit{coupled}), to the mask in Fig~\ref{subfig:mask}, regarded as ground truth.
It leads to define the \textit{classification score} as the percentage of correctly labeled pixels. 
Classification scores for $\widehat{\boldsymbol{M}}_{\mathrm{ROF}}$, $\widehat{\boldsymbol{M}}_{\mathrm{J}}$ and $\widehat{\boldsymbol{M}}_{\mathrm{C}}$, applied to different configurations of  piecewise fractal textures are reported in Table~\ref{tab:perfs}, together with the difference between the a posteriori global estimates obtained for each segmented regions $\Upsilon_0$ and $\Upsilon_1$ (cf. Section~\ref{sec:posterior}):
\begin{align*}
\widehat{\Delta H}_{\bullet} := \widehat{H}_{1,\bullet} - \widehat{H}_{0,\bullet}, \, \mathrm{for} \,  \bullet \in \lbrace \mathrm{ROF}, \, \mathrm{J} , \, \mathrm{C} \rbrace.
\end{align*}

\subsection{Performance comparisons}

\subsubsection{Segmentation and estimation performance}

Fig.~\ref{fig:perf7config} and Table~\ref{tab:perfs} report segmentation and estimation performance for 7 different configurations and for the optimal set of hyperparameters (i.e., those that maximize the classification scores).

Configurations~I, III, V, VI correspond to a decrease in the difference between the regularity of each region of the piecewise fractal texture: $\Delta H = H_1-H_0$, (hence to an increase in difficulty) for a fixed $\Delta \Sigma^2 = 0.1$. 
While the segmentation performance of the three procedures (T-ROF,  T-\textit{joint} and T-\textit{coupled}) are comparable for \emph{easy} configuration, those of T-ROF decrease drastically when $\Delta H$ decreases while those of  T-\textit{joint} and T-\textit{coupled} decrease significantly less.
Along the same line, the estimation of $\Delta H $ remains more satisfactory at small $\Delta H$ for T-\textit{joint} and T-\textit{coupled} than for T-ROF. 
It can also be observed that the performance of T-\textit{coupled} degrade slightly less than those T-\textit{joint}. 

Configurations~II, III, IV correspond to a decrease in variance, $\Delta \Sigma^2$, (hence to an increase in difficulty) for a fixed $\Delta H = 0.1$ which can already be regarded as a difficult case. 
As expected, T-ROF is not helped by the increase of variance between IV and II as
estimation of local regularity does not depend on variance~\cite{Wendt2009b,pustelnik_combining_2016}, and T-ROF segmentation results are not satisfactory. 
It can also be observed that performance of T-\textit{joint} and T-\textit{coupled} improve when $\Delta \Sigma^2$ increase, and again that the improvement is slightly larger for  T-\textit{coupled} than for  T-\textit{joint}.

These results permit to draw two clear conclusions. 
First, there are quantifiable benefits in using the side information brought by $\Delta \Sigma^2$, notably when the changes in regularity become small (low $\Delta H$): T-\textit{joint} and T-\textit{coupled} outperform T-ROF.
Second, T-\textit{coupled}, that, by principle, favor co-localized changes in regularity and variance, shows overall better performance than T-\textit{joint}, that does not favor co-localized changes. 
This is a satisfactory outcome as all the configurations chosen follow the a priori intuition, relevant for real-world applications, that changes of textures naturally imply co-localized changes in local variance and local regularity. 

\begin{figure}
\resizebox{0.15\textwidth}{!}{\begin{subfigure}{0.2\linewidth}
\includegraphics[width = \linewidth]{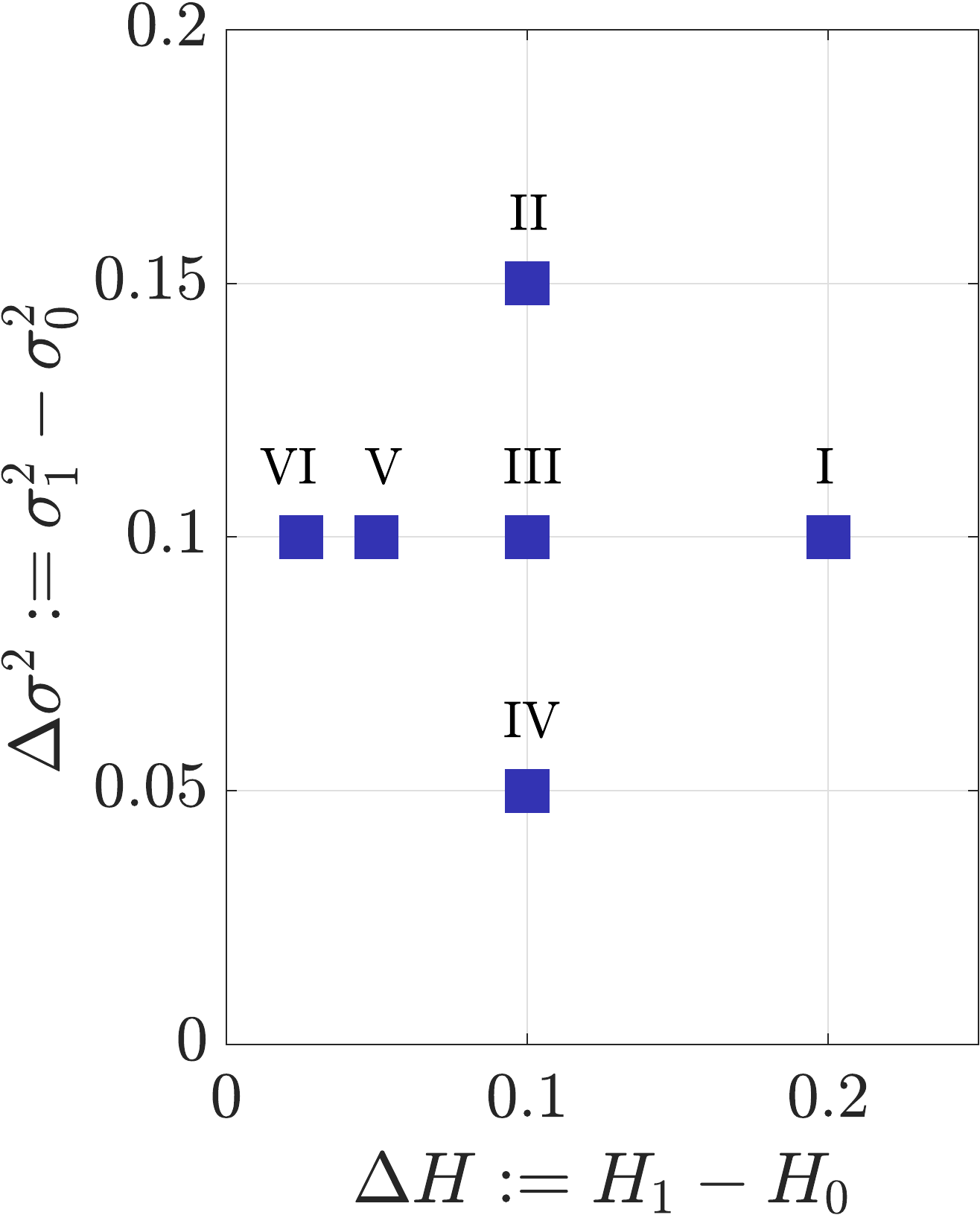}
\subcaption{\label{fig:explored_DS_DH}Explored $\Delta \Sigma^2$ and $\Delta H$.}
\end{subfigure}}\hspace{3mm}
\resizebox{0.57\textwidth}{!}{\begin{subfigure}{0.77\linewidth}
\centering
\begin{footnotesize}
\begin{tabular}{ccccccc}
 & I & II & III & IV & V & VI \\
\hline
\noalign{\vskip 1mm}   
 & \begin{minipage}[c]{16mm}
$\Delta \Sigma^2 = 0.1$\\
$\Delta H = 0.2$
\end{minipage} & \begin{minipage}[c]{17mm}
$\Delta \Sigma^2 = 0.15$\\
$\Delta H = 0.1$
\end{minipage}  & \begin{minipage}[c]{17mm}
$\Delta \Sigma^2 = 0.1$\\
$\Delta H = 0.1$
\end{minipage} & \begin{minipage}[c]{17mm}
$\Delta \Sigma^2 = 0.05$\\
$\Delta H = 0.1$
\end{minipage} & \begin{minipage}[c]{17mm}
$\Delta \Sigma^2 = 0.1$\\
$\Delta H = 0.05$
\end{minipage} & \begin{minipage}[c]{17mm}
$\Delta \Sigma^2 = 0.1$\\
$\Delta H = 0.025$
\end{minipage} \\
\noalign{\vskip 1mm}   
\hline 
\noalign{\vskip 1mm}   
\vspace{1mm}
\rotatebox{90}{Texture}
& \raisebox{-4mm}{\includegraphics[width = 17mm]{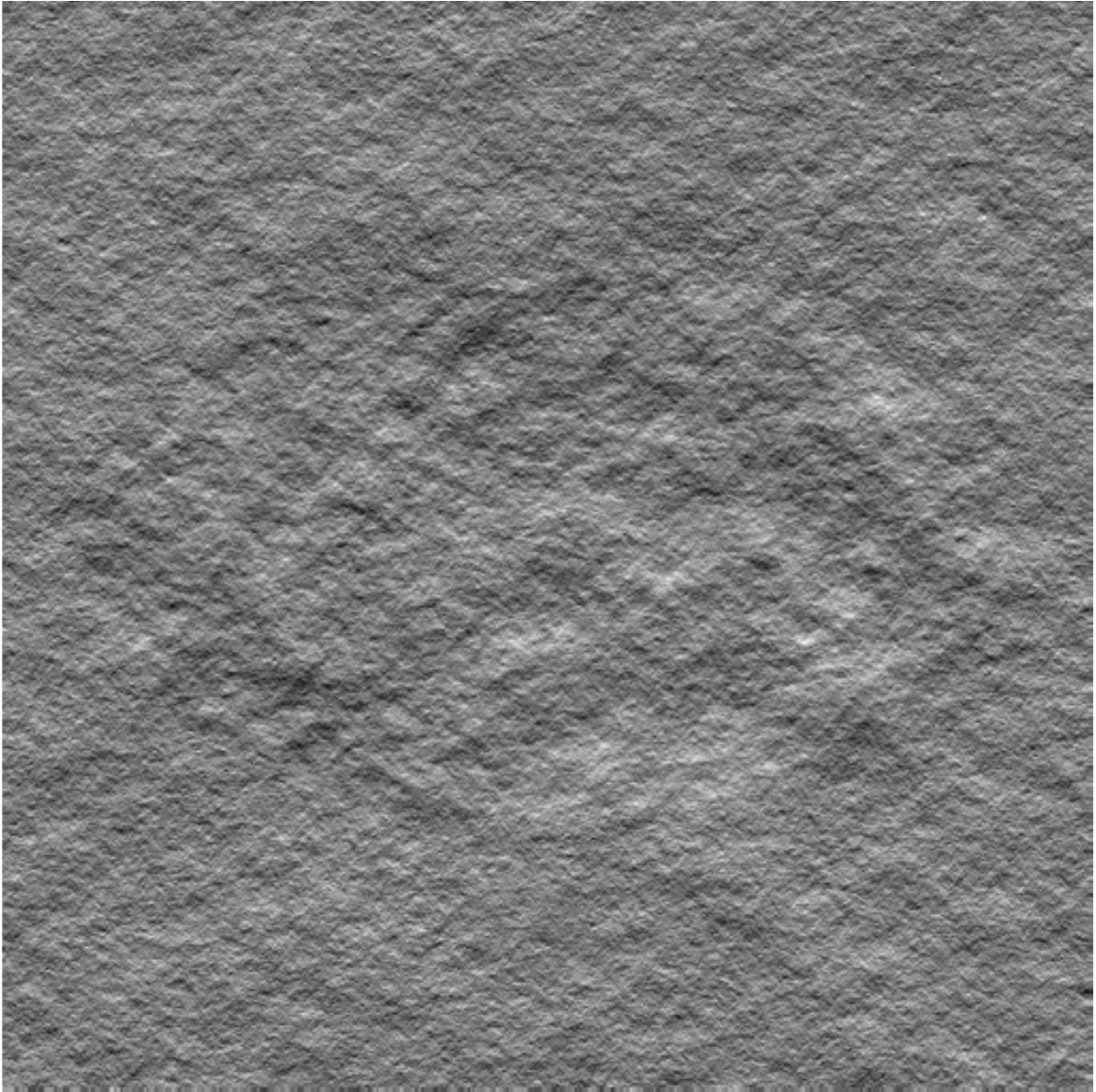}} 
& \raisebox{-4mm}{\includegraphics[width = 17mm]{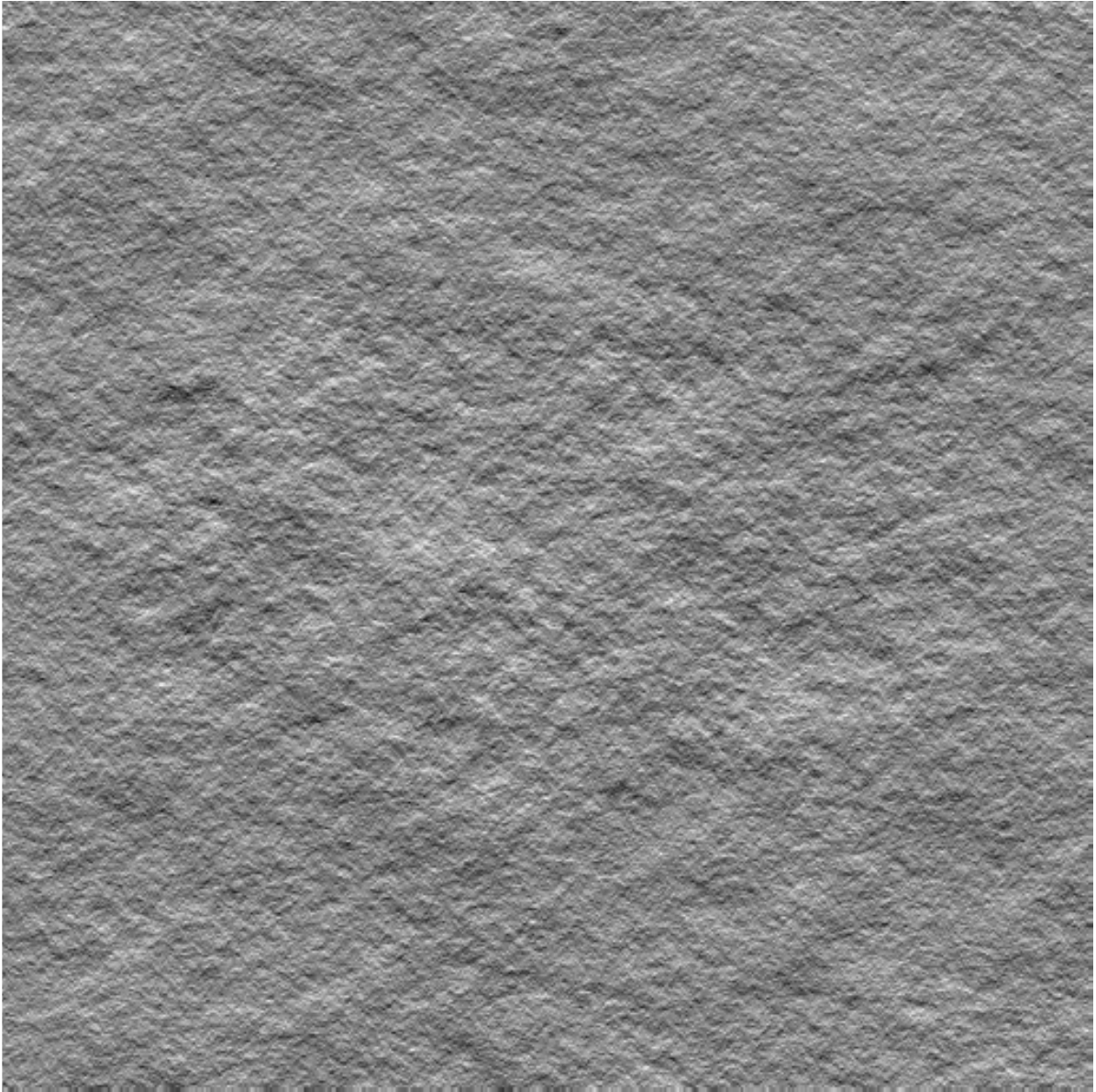}} 
& \raisebox{-4mm}{\includegraphics[width = 17mm]{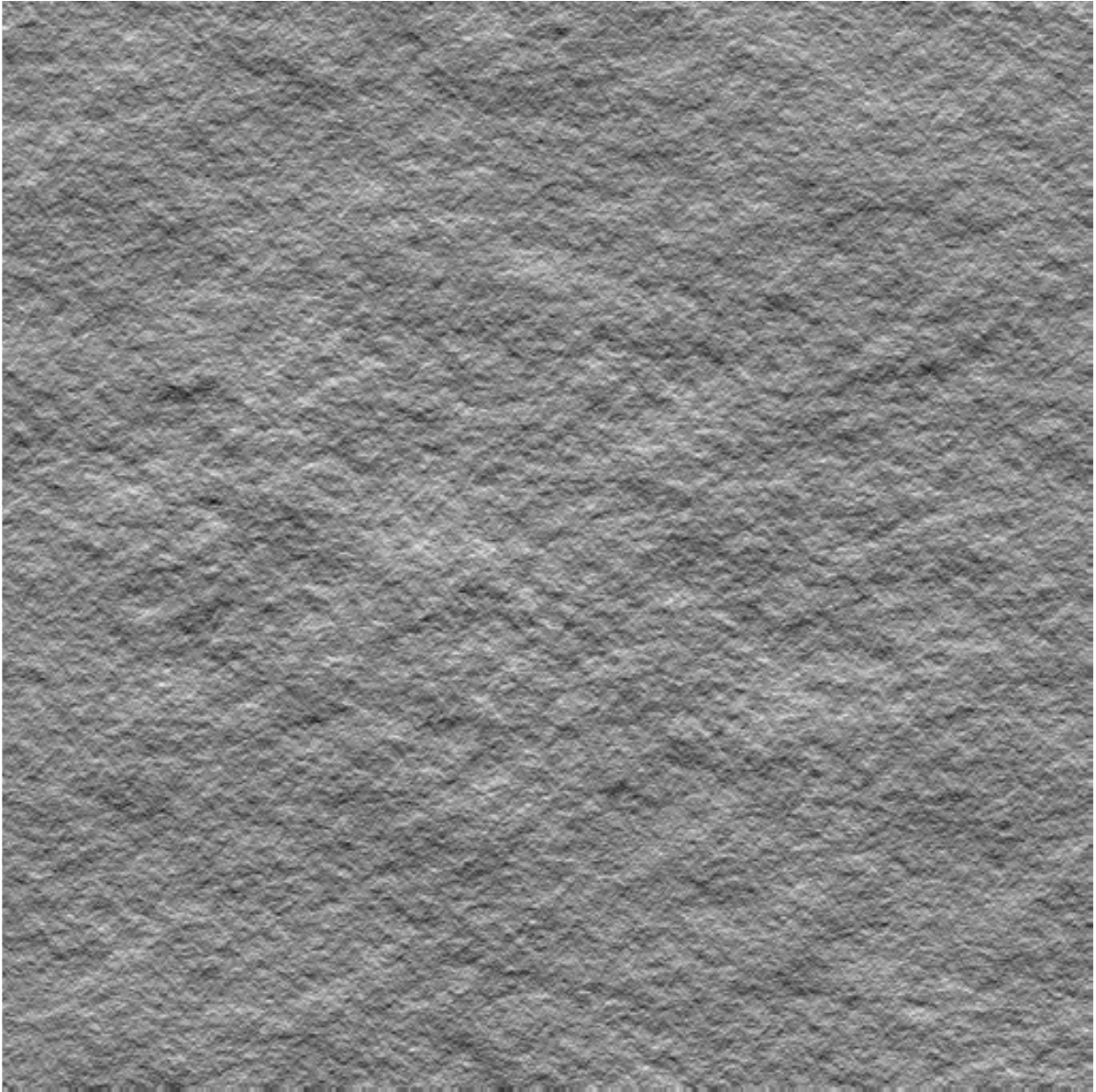}} 
& \raisebox{-4mm}{\includegraphics[width = 17mm]{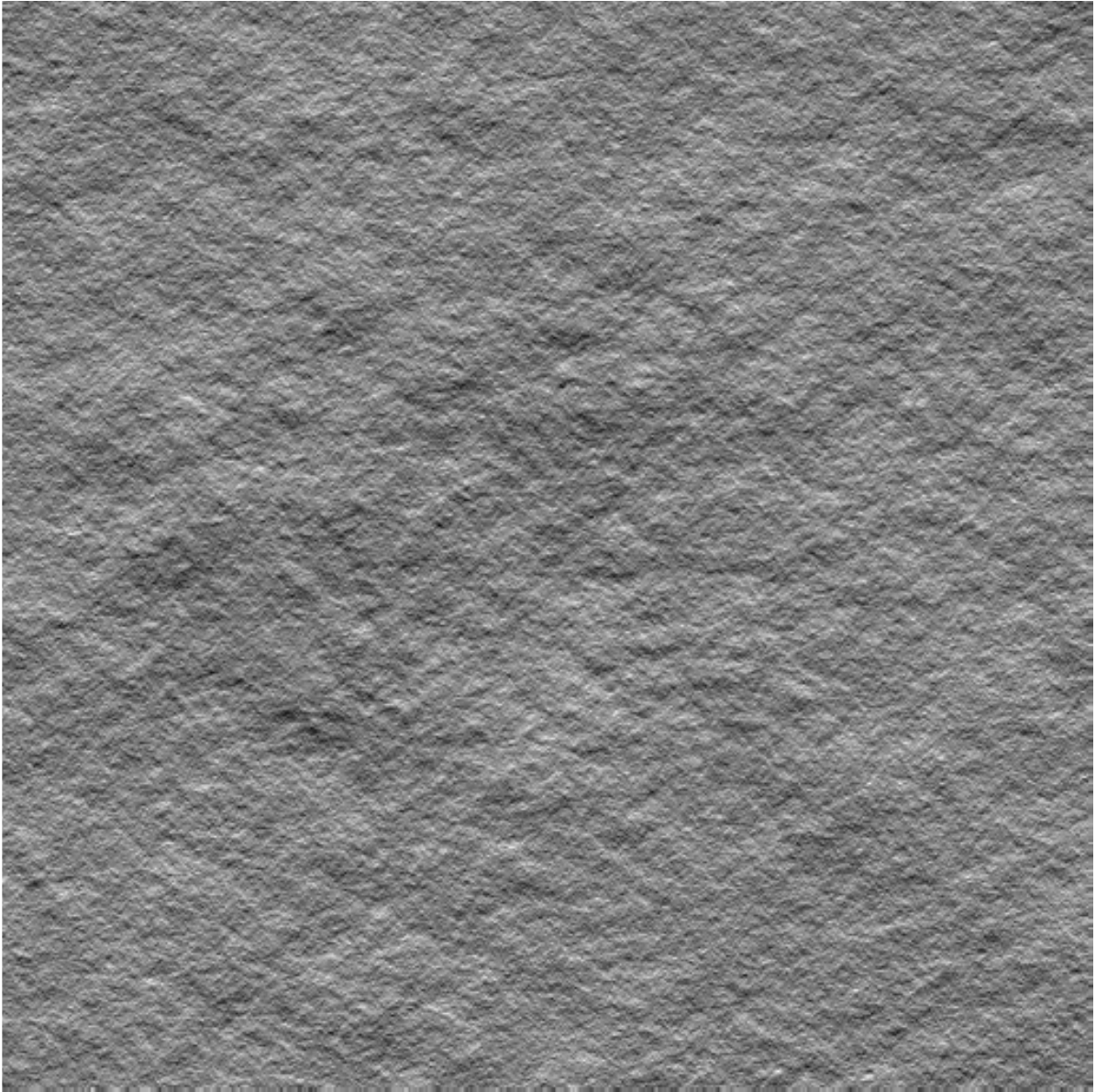}} 
& \raisebox{-4mm}{\includegraphics[width = 17mm]{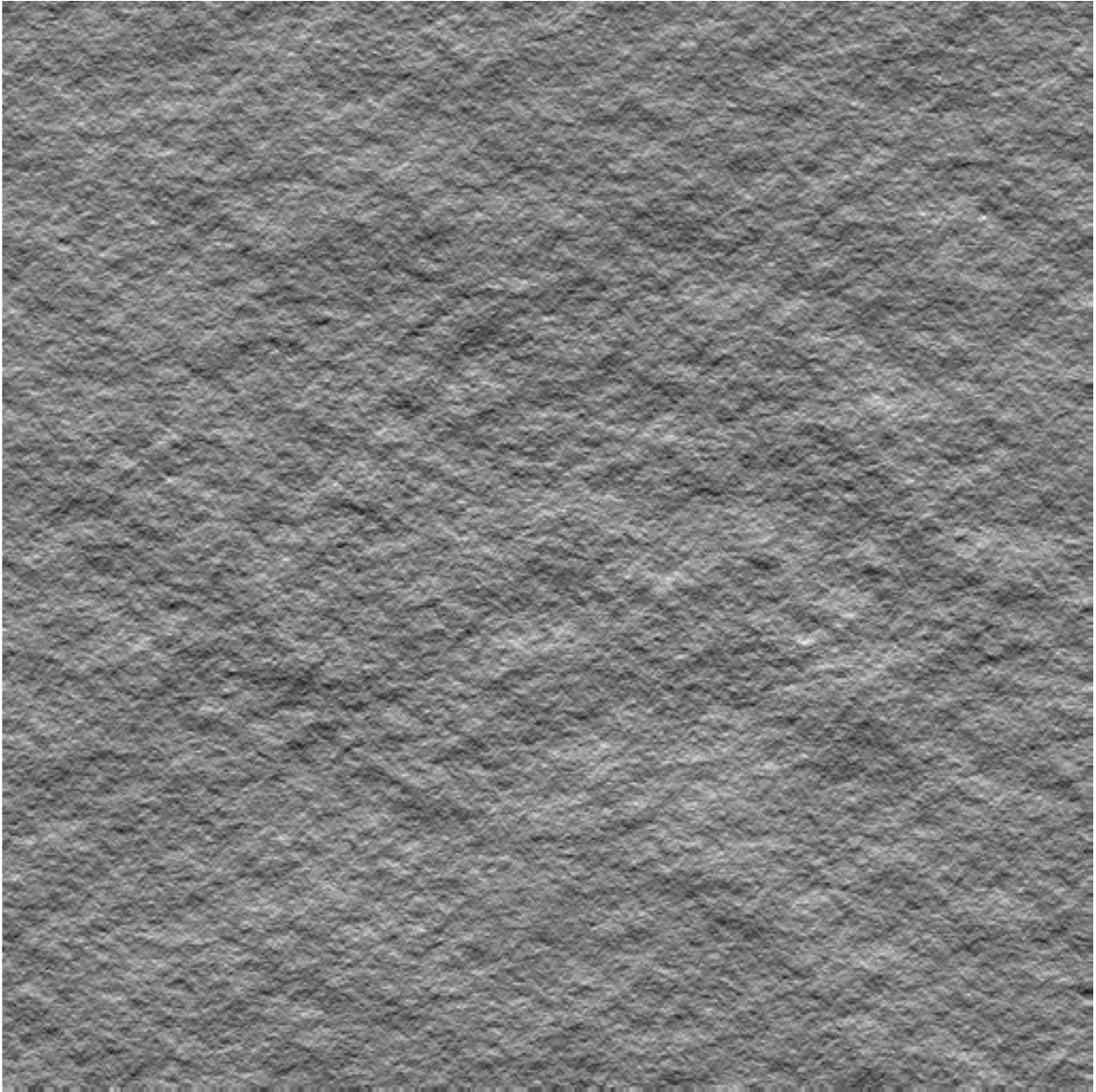}} 
& \raisebox{-4mm}{\includegraphics[width = 17mm]{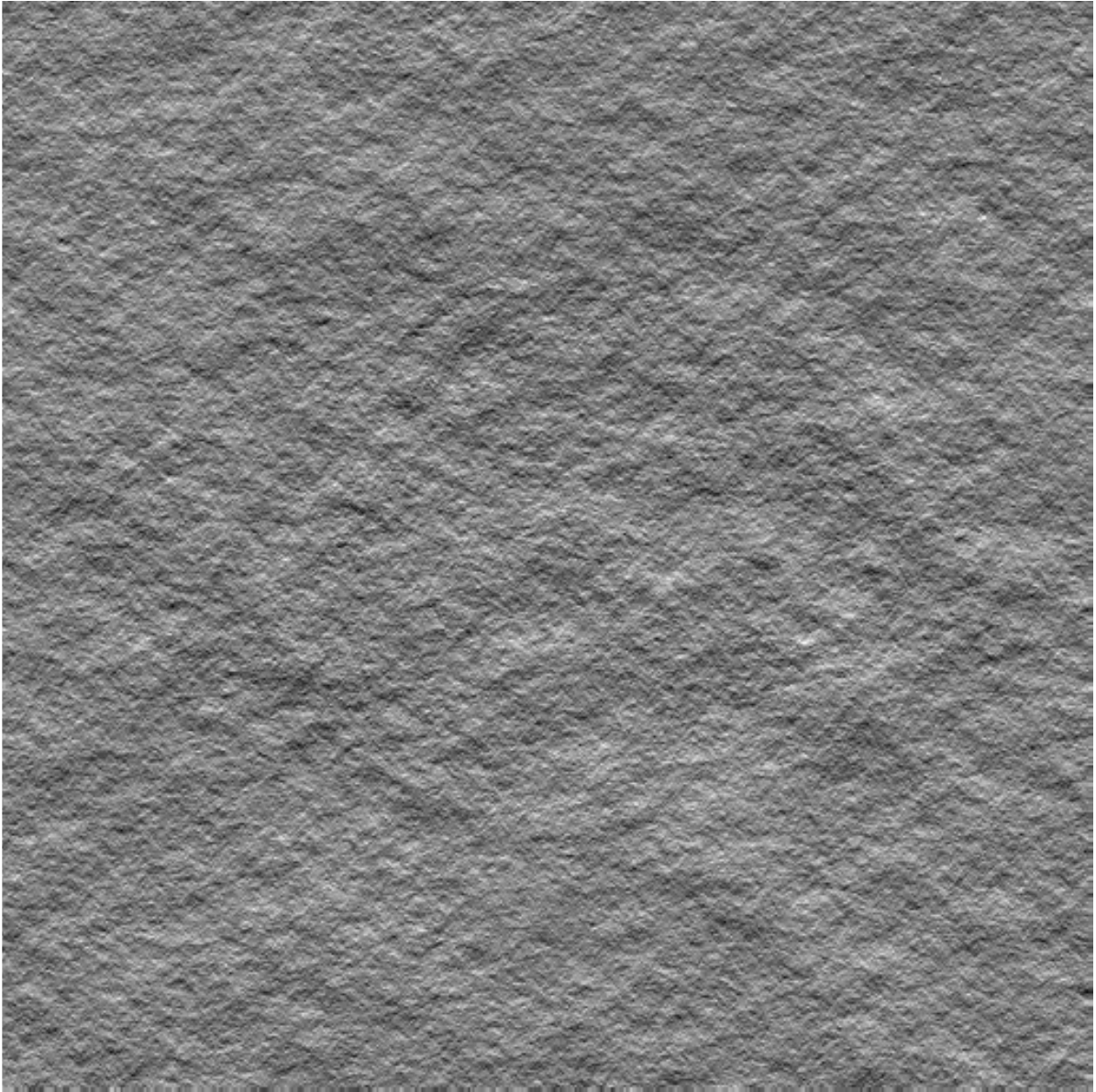}} \\
\vspace{1mm}
 \rotatebox{90}{T-ROF }
 & \raisebox{-5mm}{\includegraphics[width = 17mm]{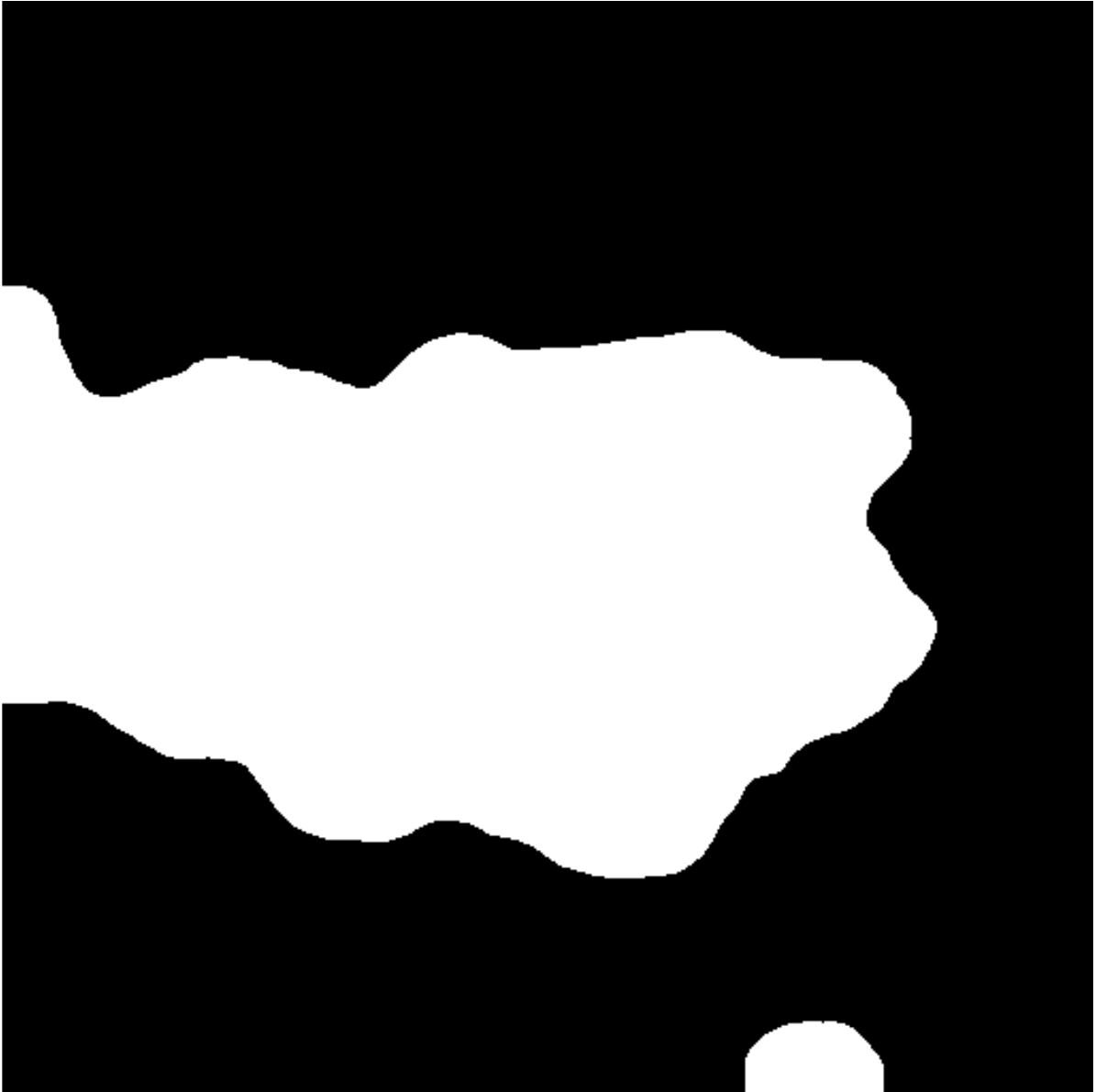}} 
 & \raisebox{-5mm}{\includegraphics[width = 17mm]{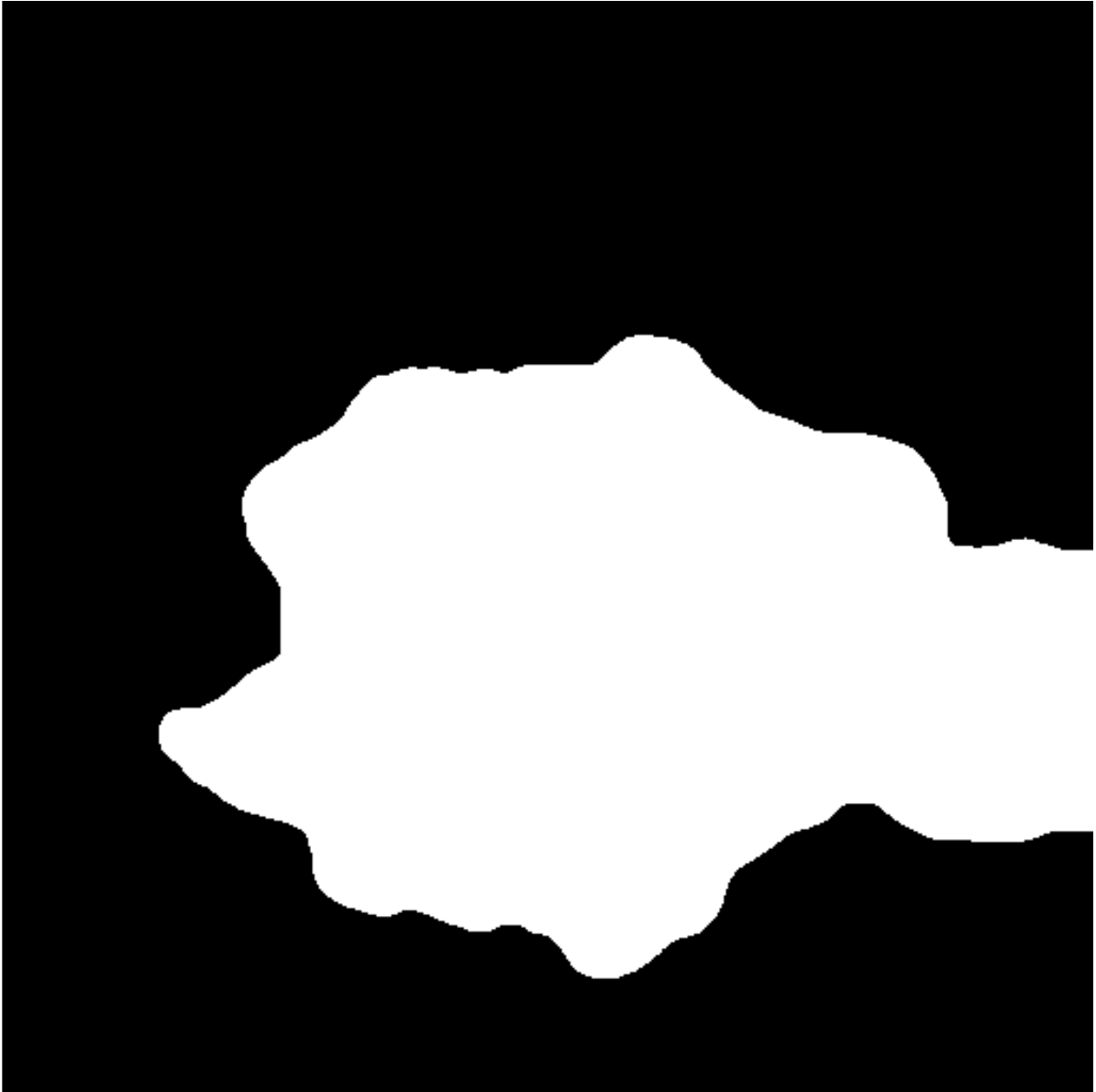}} 
 & \raisebox{-5mm}{\includegraphics[width = 17mm]{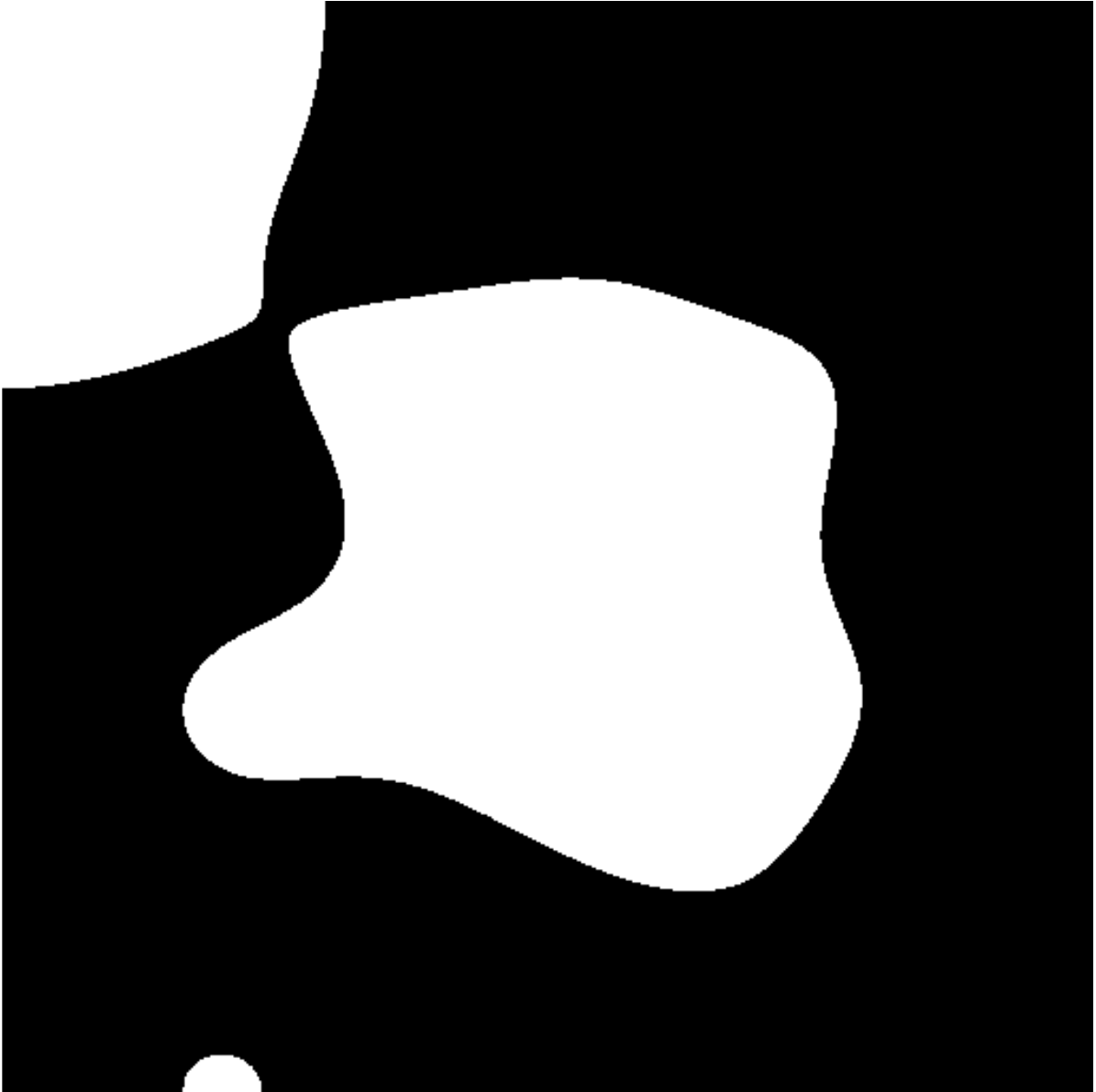}} 
 & \raisebox{-5mm}{\includegraphics[width = 17mm]{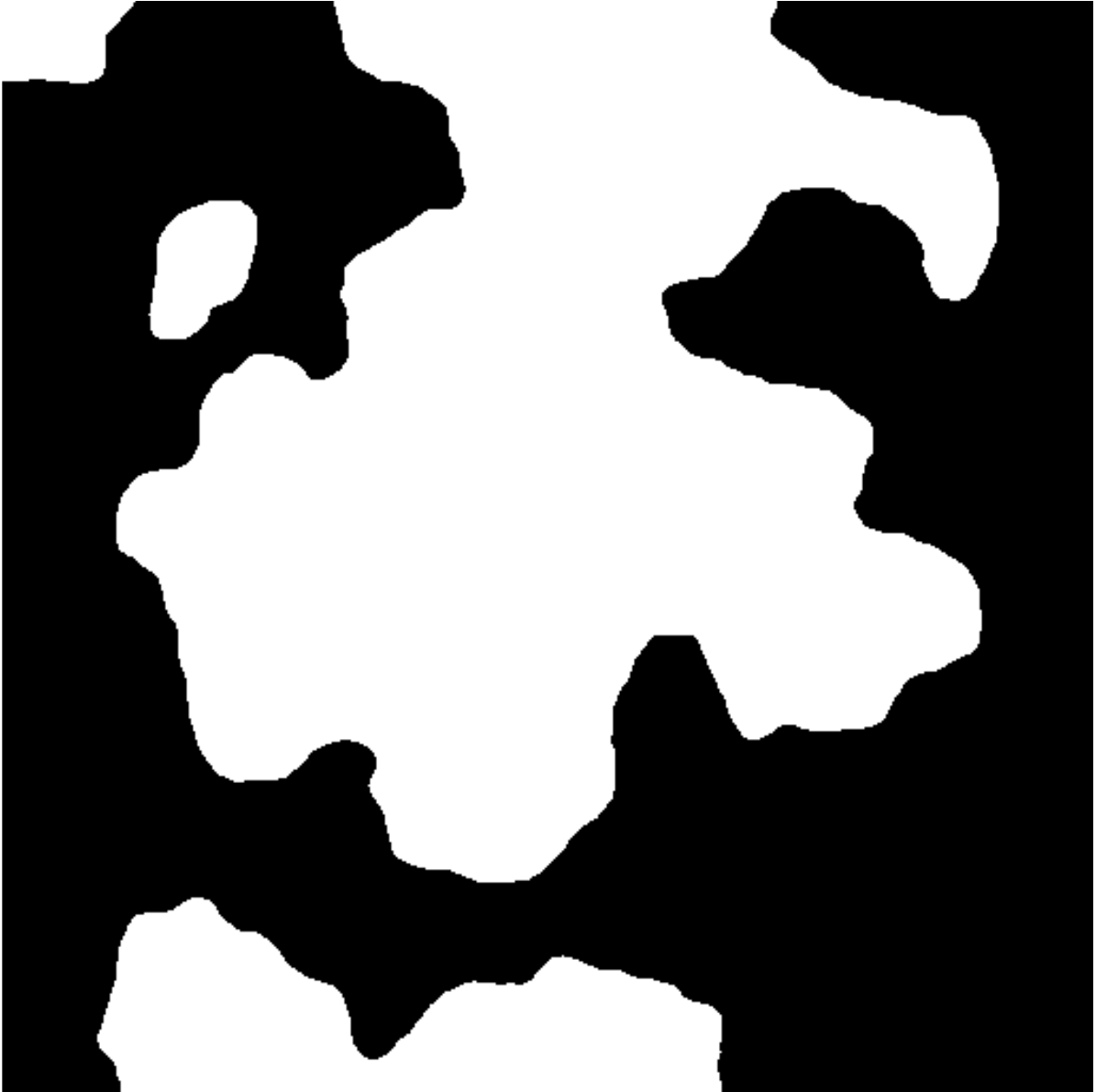}} 
 & \raisebox{-5mm}{\includegraphics[width = 17mm]{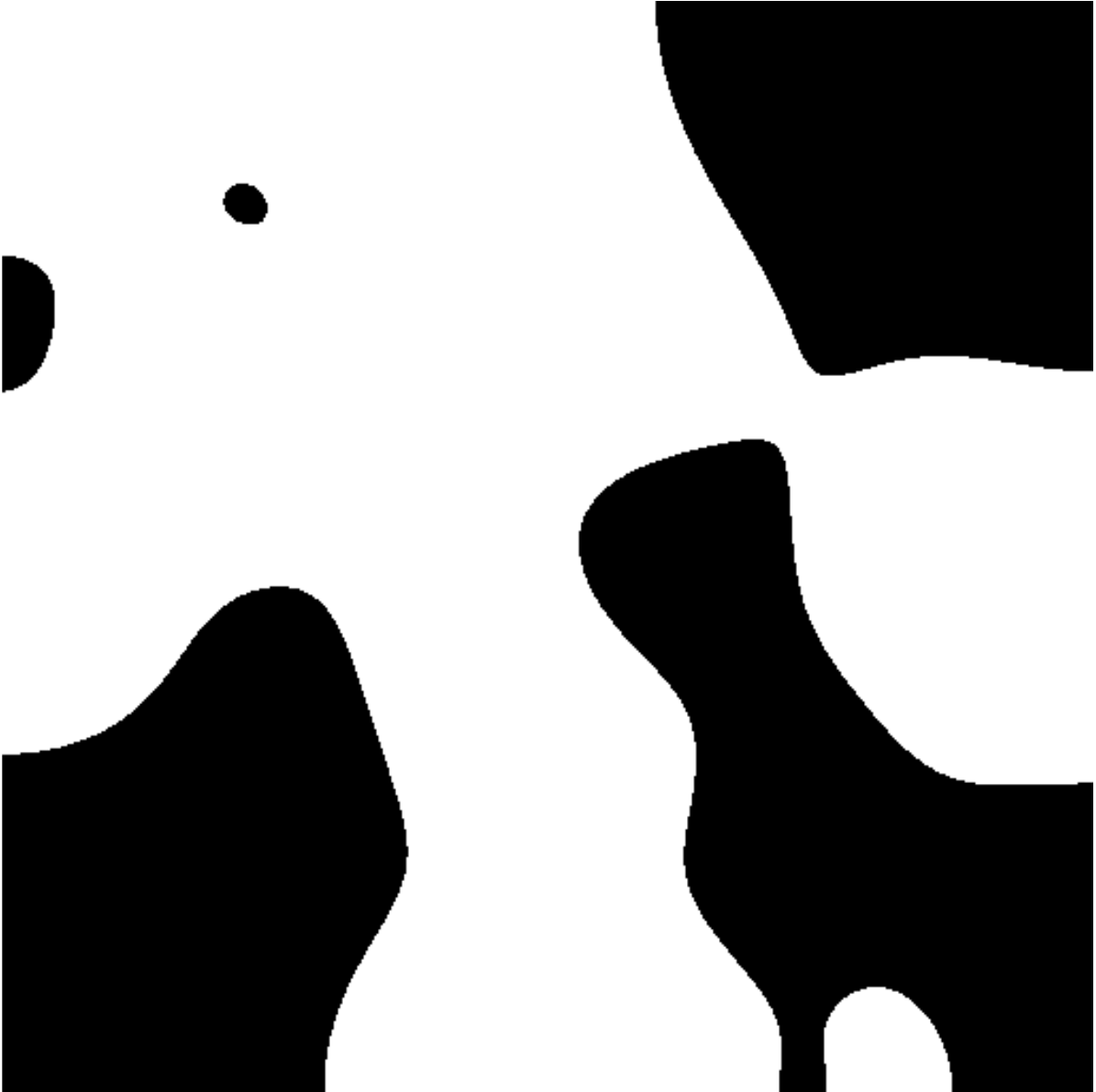}} 
 & \raisebox{-5mm}{\includegraphics[width = 17mm]{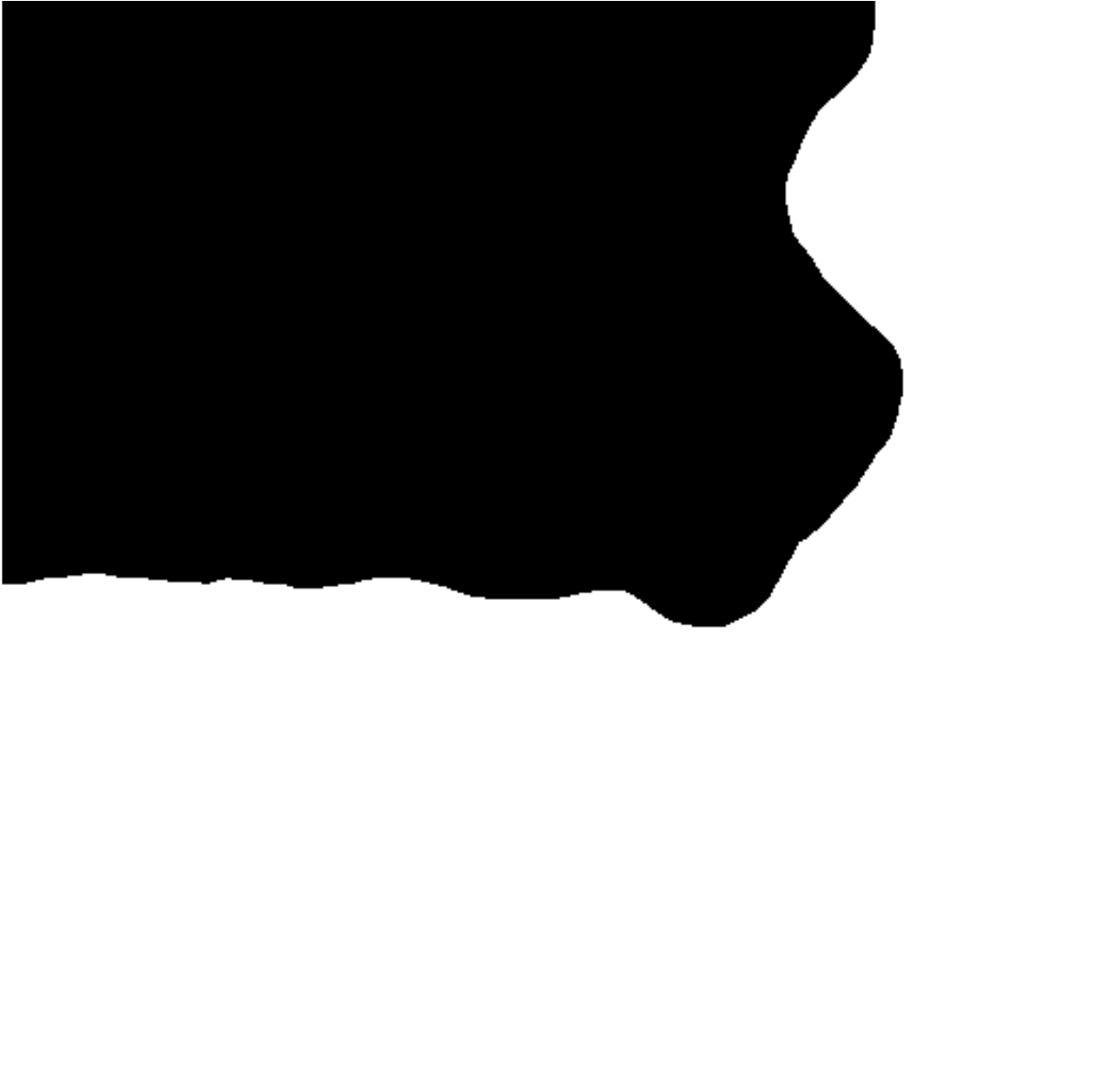}} \\
 \vspace{1mm}
\rotatebox{90}{T-\textit{joint}  }
& \raisebox{-5mm}{\includegraphics[width = 17mm]{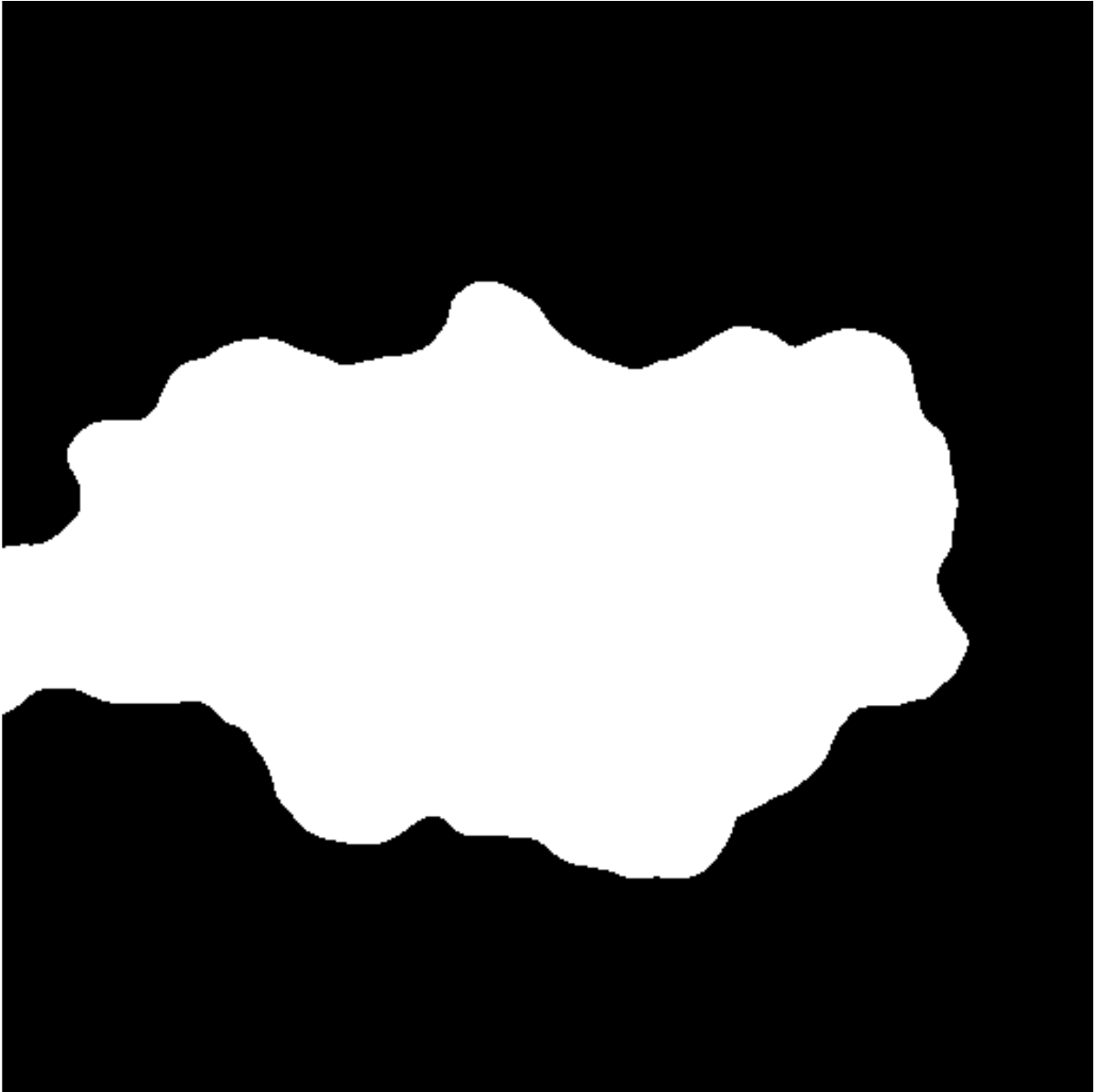}} 
& \raisebox{-5mm}{\includegraphics[width = 17mm]{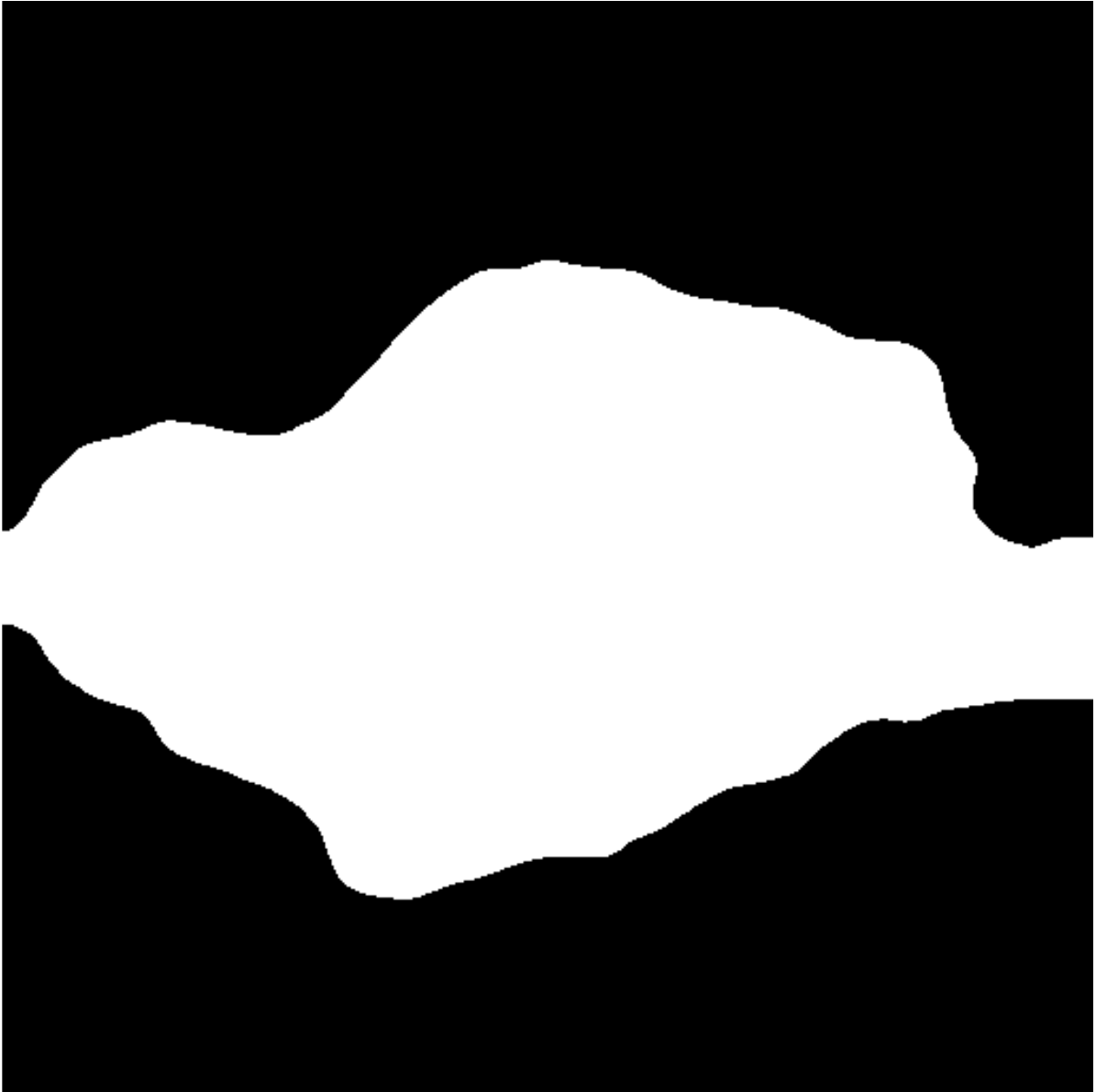}} 
& \raisebox{-5mm}{\includegraphics[width = 17mm]{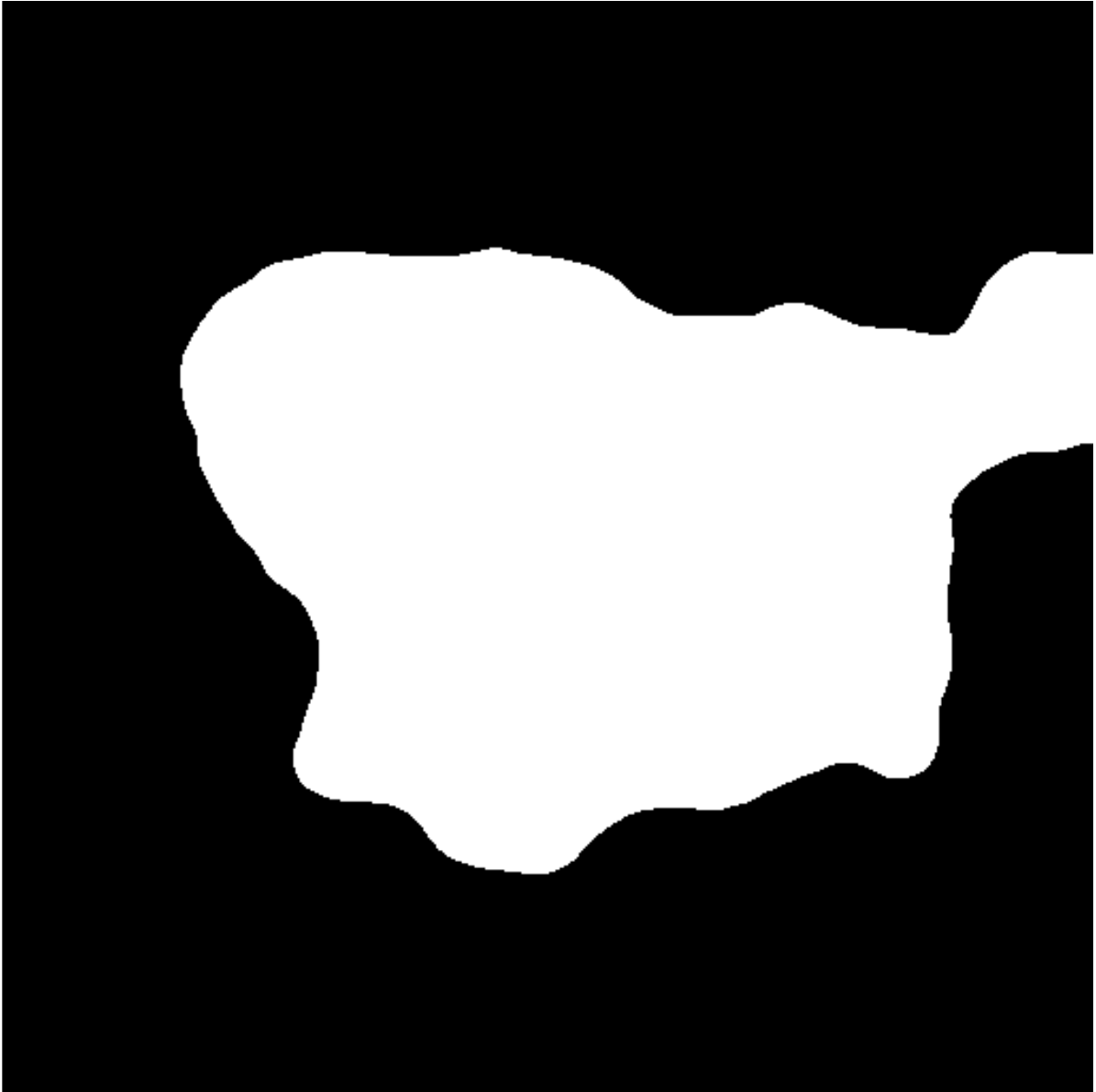}} 
& \raisebox{-5mm}{\includegraphics[width = 17mm]{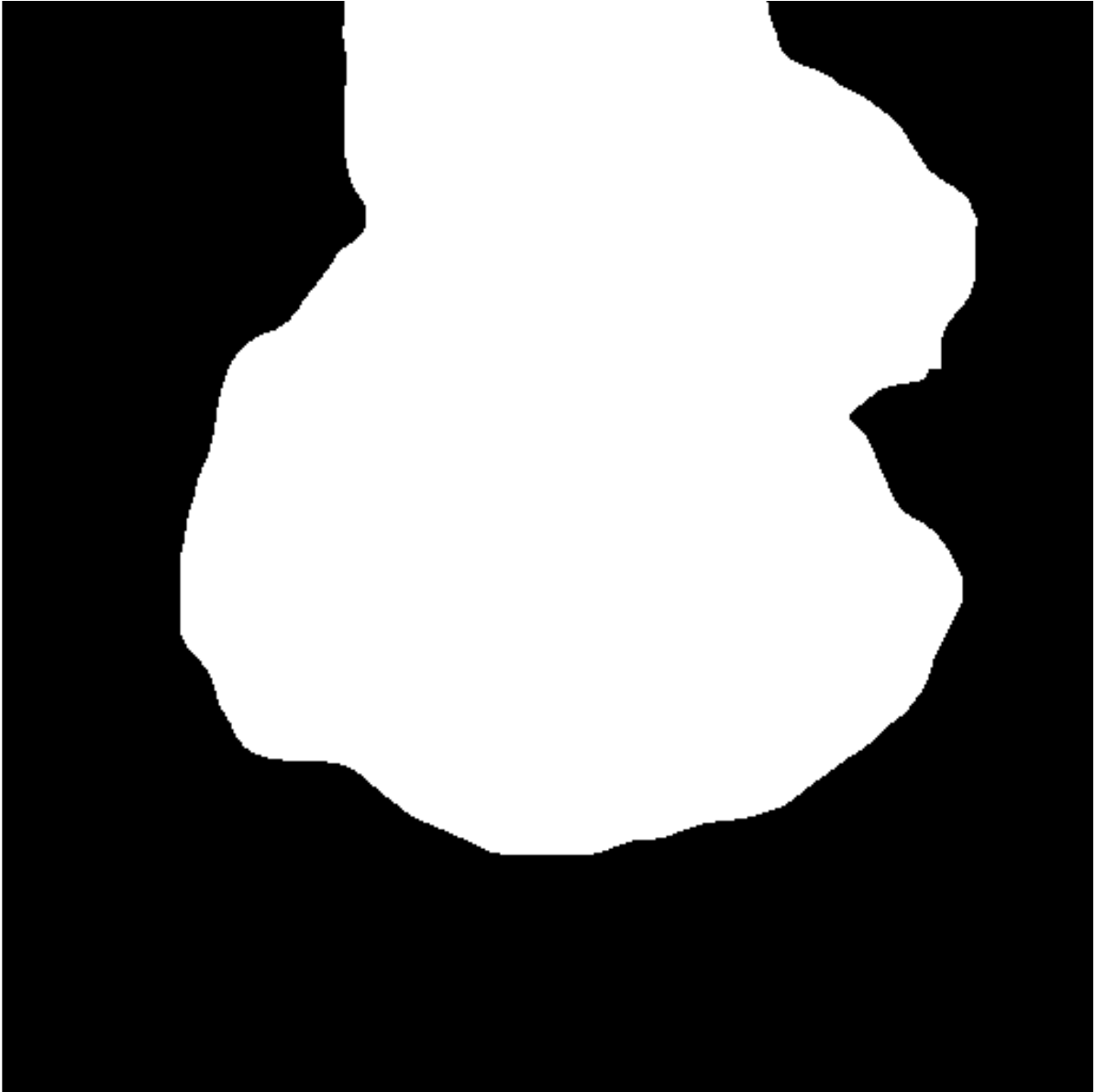}} 
& \raisebox{-5mm}{\includegraphics[width = 17mm]{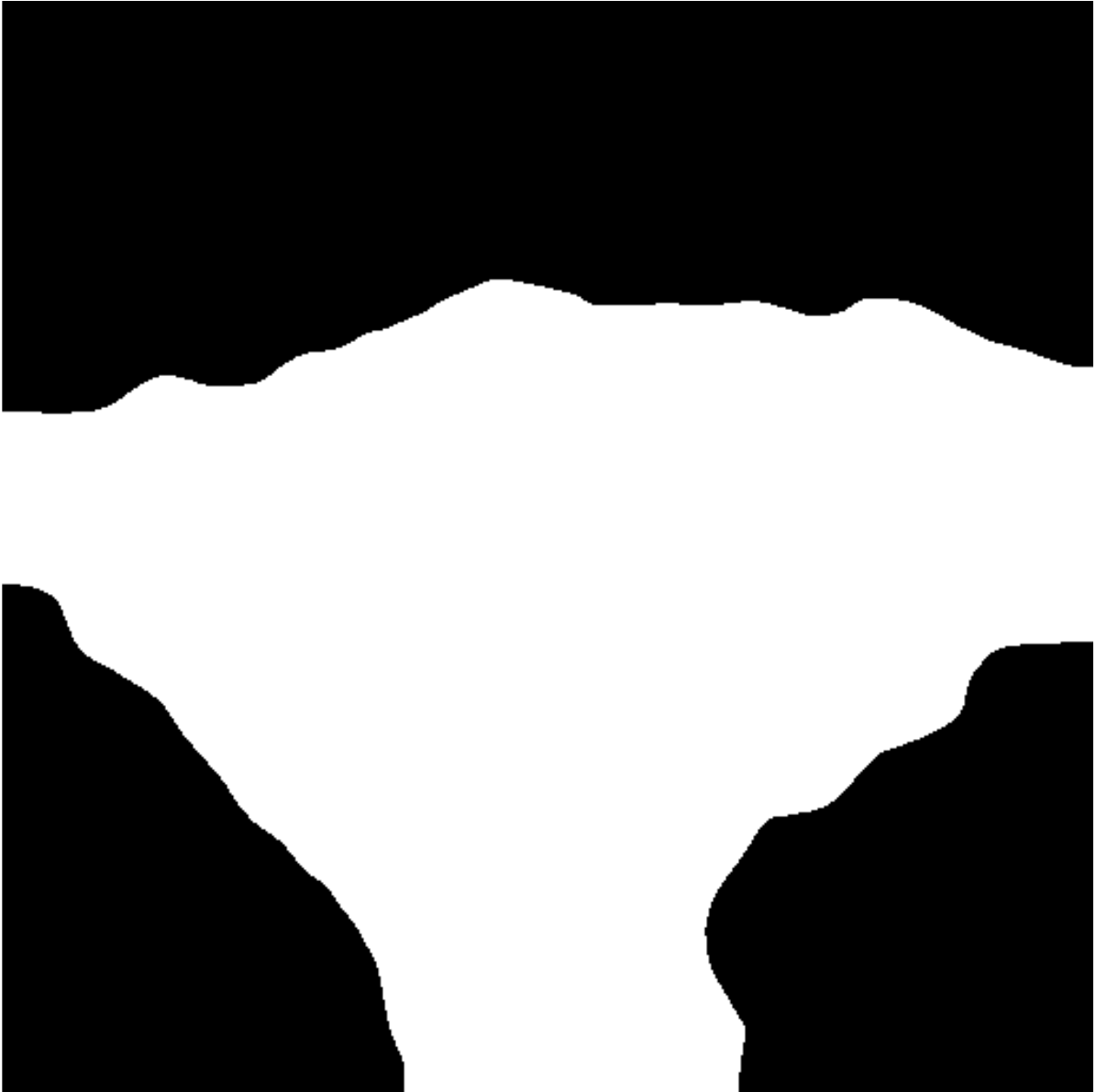}} 
& \raisebox{-5mm}{\includegraphics[width = 17mm]{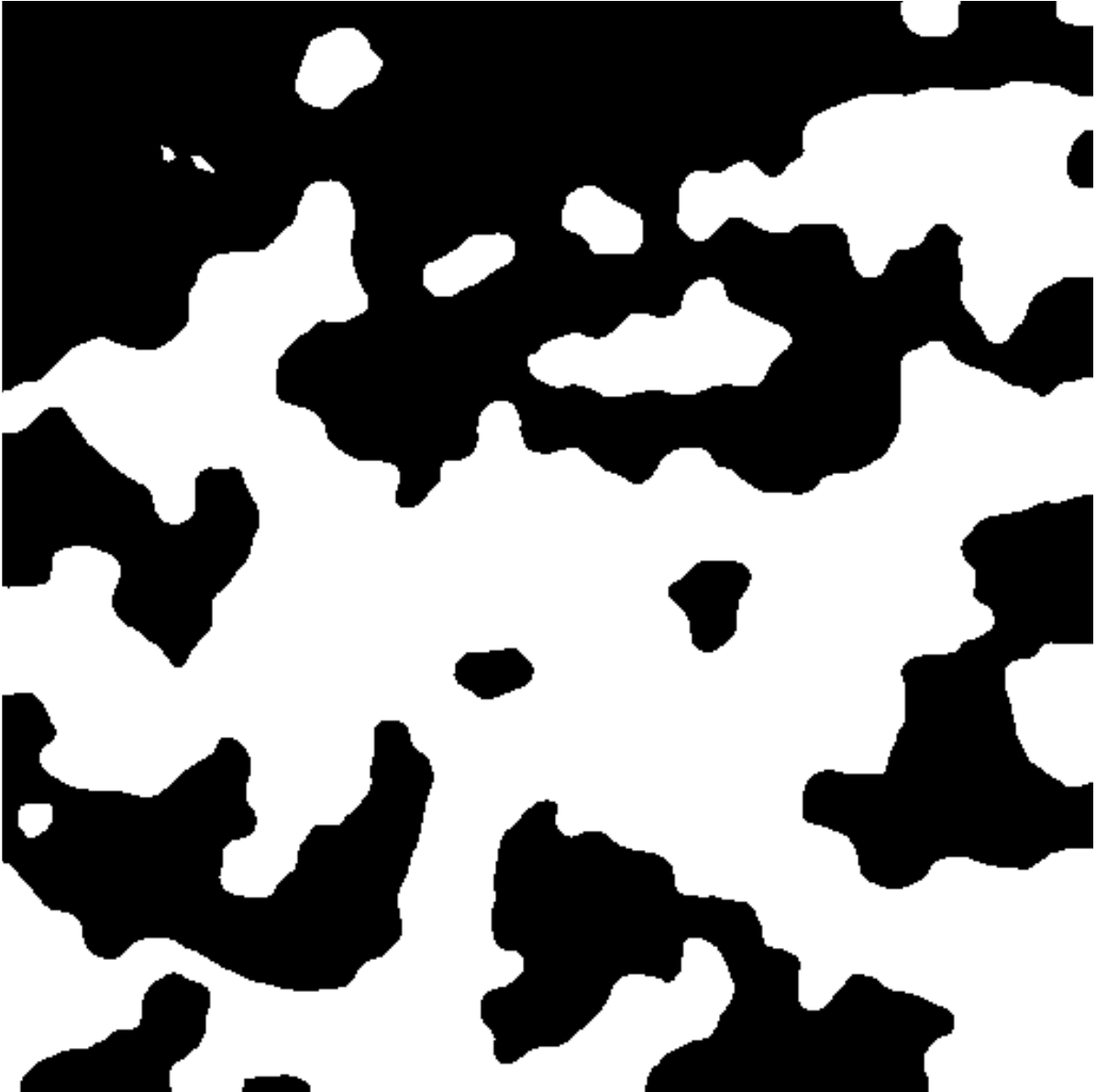}} \\
\rotatebox{90}{T-\textit{coupled} }
& \raisebox{-3mm}{\includegraphics[width = 17mm]{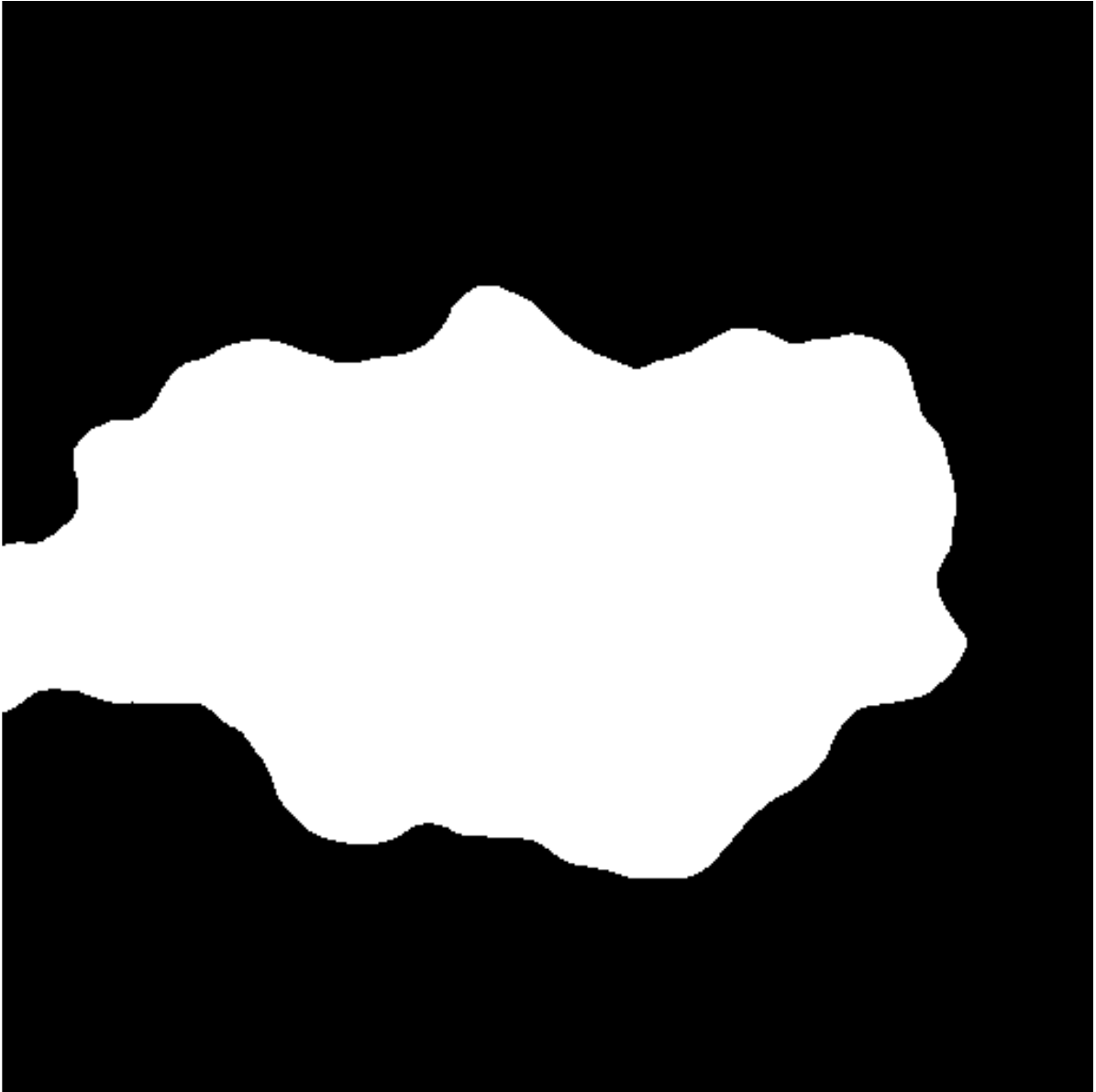}} 
& \raisebox{-3mm}{\includegraphics[width = 17mm]{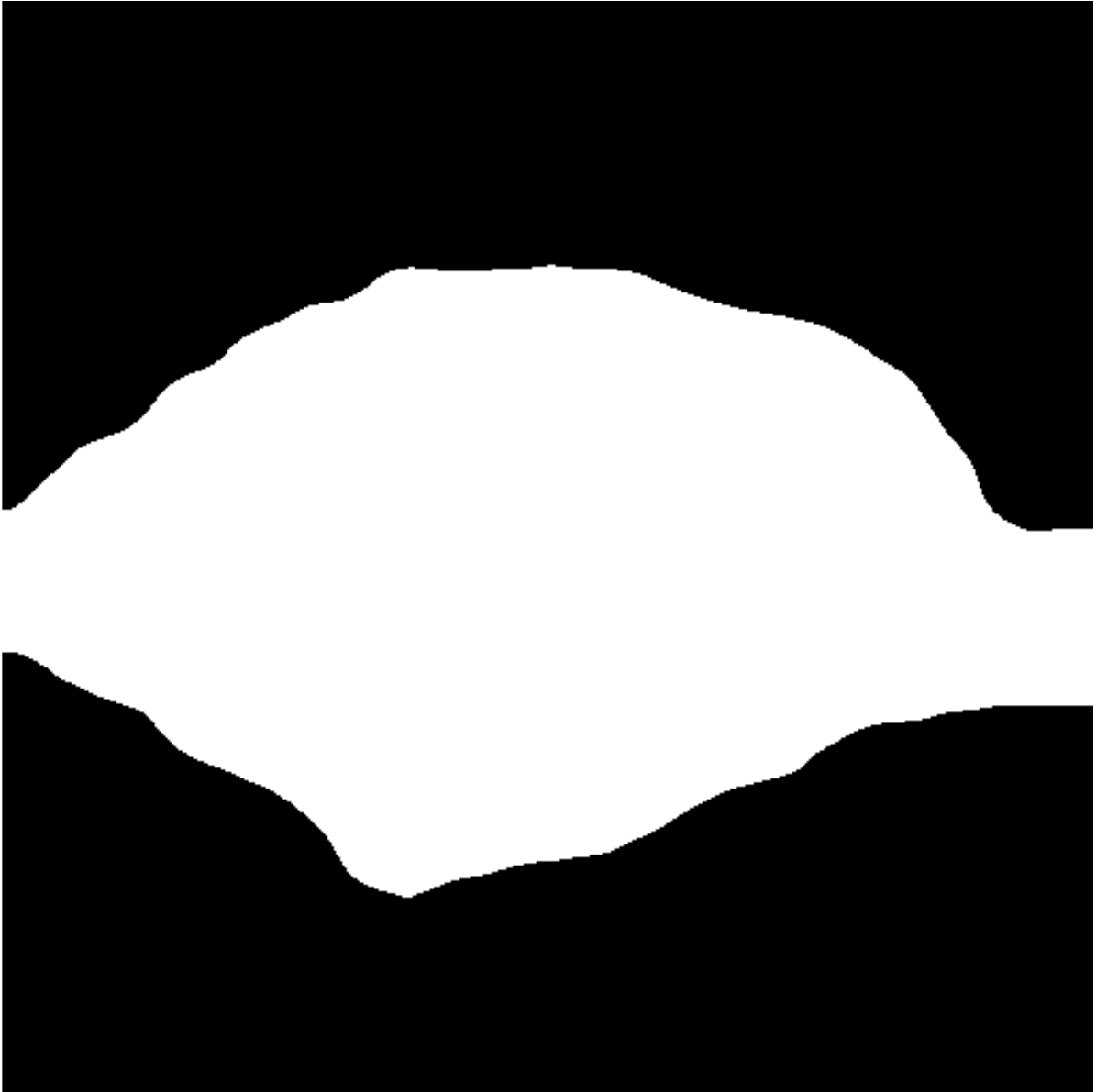}} 
& \raisebox{-3mm}{\includegraphics[width = 17mm]{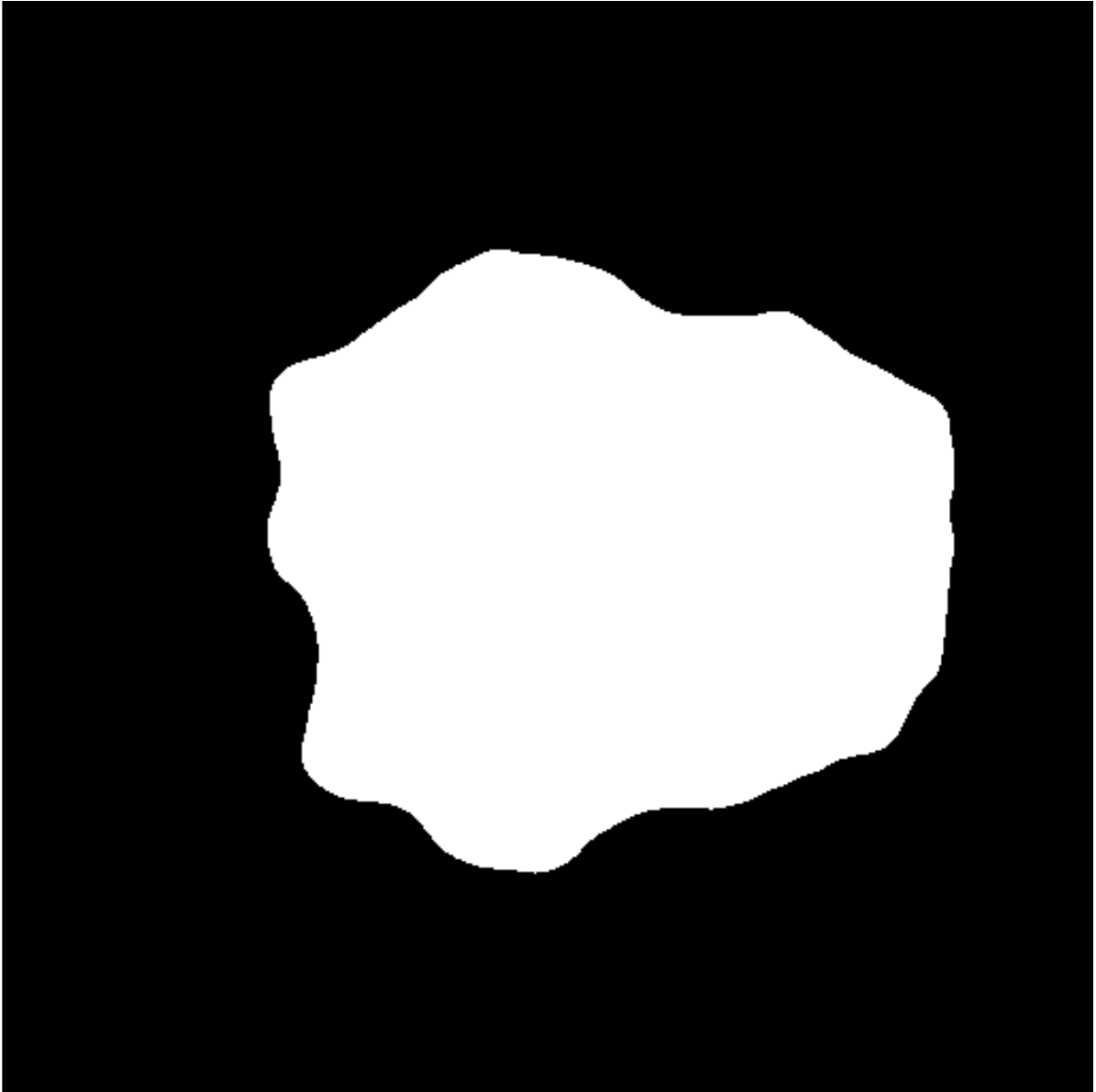}} 
& \raisebox{-3mm}{\includegraphics[width = 17mm]{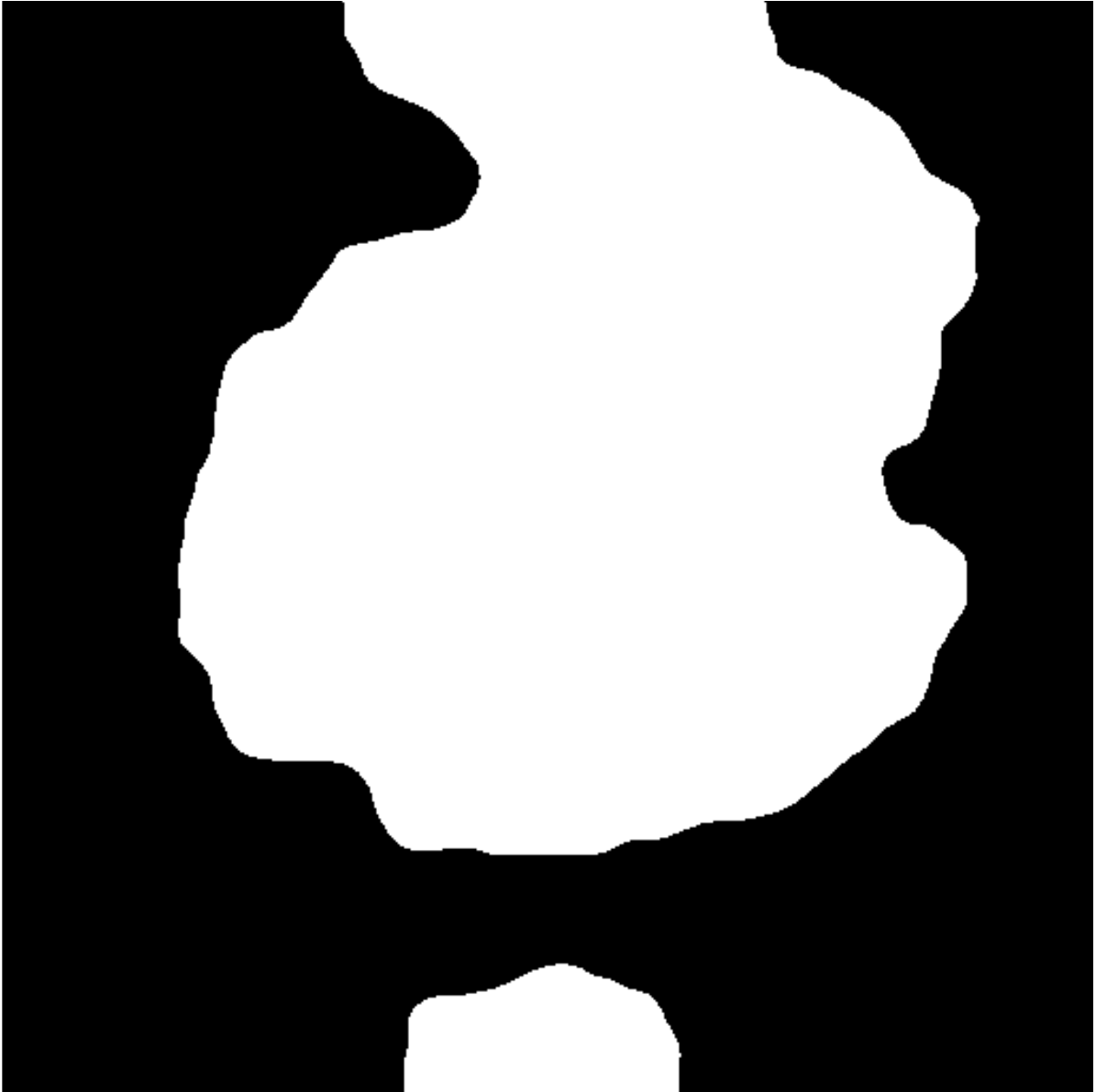}} 
& \raisebox{-3mm}{\includegraphics[width = 17mm]{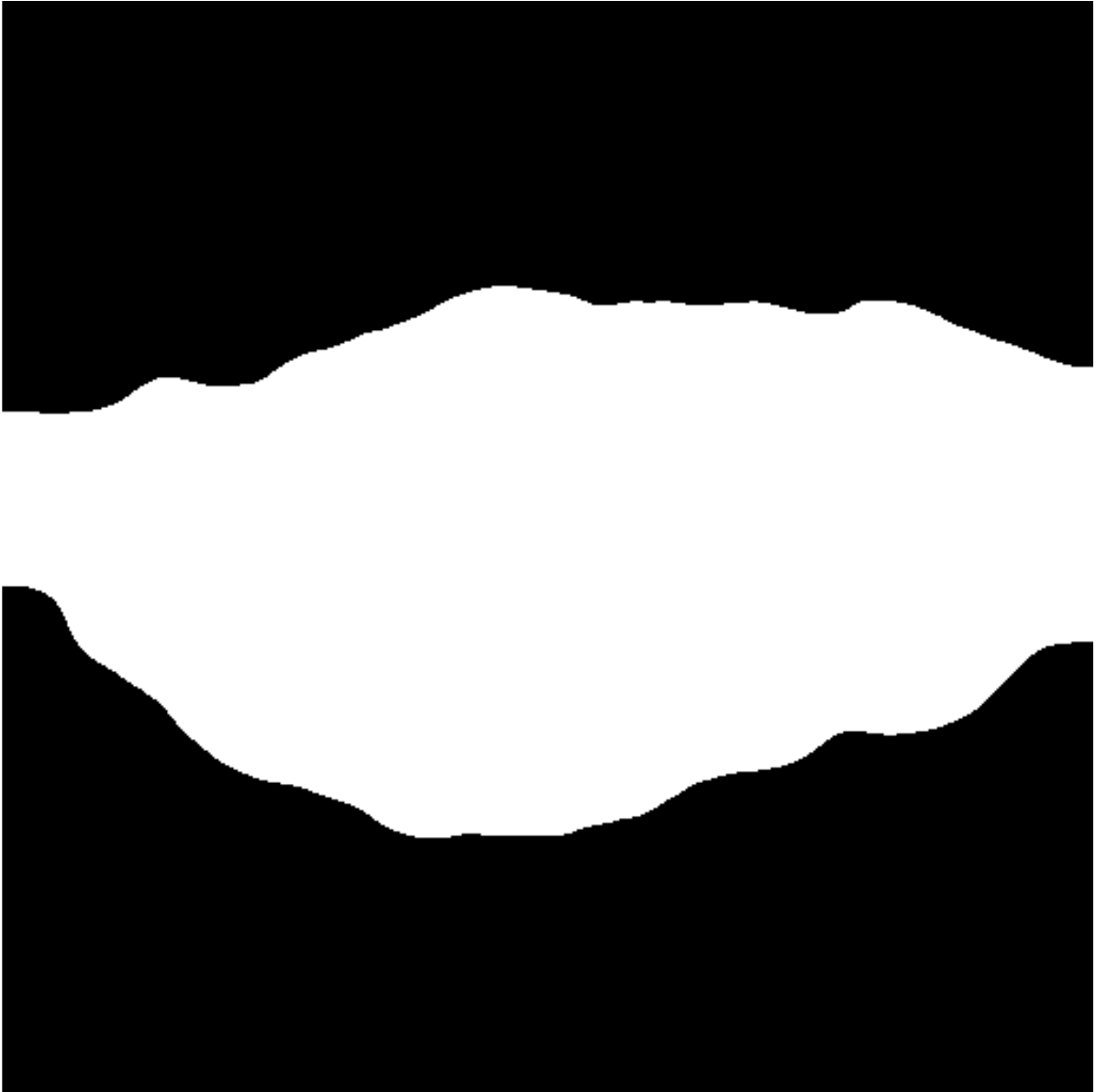}} 
& \raisebox{-3mm}{\includegraphics[width = 17mm]{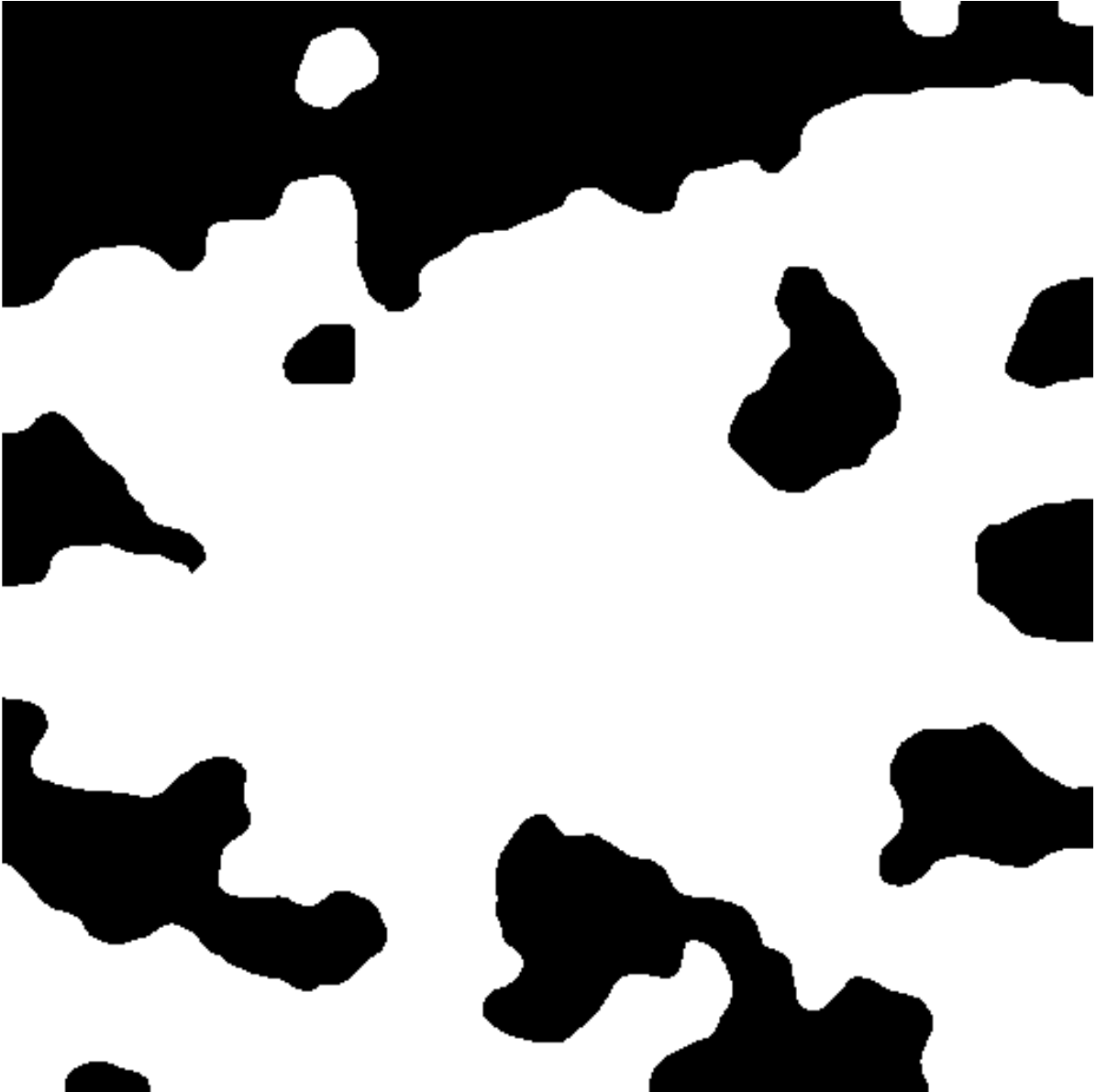}} 
\end{tabular}
\end{footnotesize}
\end{subfigure}}
\caption{\label{fig:perf7config}\textbf{Compared optimal segmentation.} Piecewise fractal textures are  characterized by $\Sigma_0^2 = 0.6$, $H_0 = 0.5$ and different $\Delta \Sigma^2$, $\Delta H$ as sketched in Fig.~\ref{subfig:mask}. First row: T-ROF segmentation $\widehat{\boldsymbol{M}}_{\mathrm{ROF}}$.
Second row: T-\textit{joint} segmentation $\widehat{\boldsymbol{M}}_{\mathrm{J}}$.
Third row: T-\textit{coupled} segmentation $\widehat{\boldsymbol{M}}_{\mathrm{C}}$.
}
\end{figure}

\setlength{\tabcolsep}{0.8mm}
\begin{table}
\centering
\resizebox{\textwidth}{!}{\begin{tabular}{cccccccc}
 & & I & II & III & IV & V & VI \\ 
\hline
\noalign{\vskip 1mm}   
 & & \begin{minipage}[c]{18mm}
$\Delta \Sigma^2 = 0.1$\\
$\Delta H = 0.2$
\end{minipage} & 
\begin{minipage}[c]{18mm}
$\Delta \Sigma^2 = 0.15$\\
$\Delta H = 0.1$
\end{minipage}  & \begin{minipage}[c]{18mm}
$\Delta \Sigma^2 = 0.1$\\
$\Delta H = 0.1$
\end{minipage} & \begin{minipage}[c]{18mm}
$\Delta \Sigma^2 = 0.05$\\
$\Delta H = 0.1$
\end{minipage} & \begin{minipage}[c]{18mm}
$\Delta \Sigma^2 = 0.1$\\
$\Delta H = 0.05$
\end{minipage} & \begin{minipage}[c]{19mm}
$\Delta \Sigma^2 = 0.1$\\
$\Delta H = 0.025$
\end{minipage} \\
\noalign{\vskip 1mm}   
\hline 
\noalign{\vskip 1mm}   
                                           T-ROF & \begin{minipage}{10mm}
                                           \centering
                                           Score\\
                                           $\widehat{\Delta H} $
                                           \end{minipage}
 & \begin{minipage}[c]{18mm}
 86.7 $\pm$ 2.1\%\\
 $ 0.21 \pm 0.07$
 \end{minipage}
 & \begin{minipage}[c]{18mm}
 79.5 $\pm$ 1.2\% \\
 $ 0.05 \pm 0.02$
 \end{minipage}
 & \begin{minipage}[c]{18mm}
 78.5 $\pm$ 1.1\% \\
 $ 0.05 \pm 0.06$
 \end{minipage}
 &  \begin{minipage}[c]{18mm}
 77.5 $\pm$ 2.9\% \\
 $ 0.07 \pm 0.04$
 \end{minipage}
 &  \begin{minipage}[c]{19mm}
 69.9 $\pm$ 7.1\%  \\
 $ 0.01 \pm 0.06$
 \end{minipage}
 & \begin{minipage}[c]{18mm}
 59.5 $\pm$ 2.4\%  \\
 $ 0.05 \pm 0.07$
 \end{minipage} \\
 \noalign{\vskip 1mm}   
\hline 
\noalign{\vskip 1mm}   
                                           T-\textit{joint} & \begin{minipage}{10mm}
                                           \centering
                                           Score\\
                                           $\widehat{\Delta H} $
                                           \end{minipage}
& \begin{minipage}[c]{18mm}
 91.6 $\pm$ 1.7\% \\
 $ 0.21 \pm 0.06$
 \end{minipage}
& \begin{minipage}[c]{18mm}
 91.5 $\pm$ 2.0\% \\
 $ 0.07 \pm 0.03$
 \end{minipage}
& \begin{minipage}[c]{18mm}
 90.2 $\pm$ 1.9\% \\
 $ 0.10 \pm 0.02$
 \end{minipage}
&  \begin{minipage}[c]{18mm}
 84.2 $\pm$ 4.5\% \\
 $ 0.04 \pm 0.07$
 \end{minipage}
&  \begin{minipage}[c]{18mm}
 84.3 $\pm$ 3.2\% \\
 $ 0.05 \pm 0.02$
 \end{minipage}
& \begin{minipage}[c]{18mm}
 74.7 $\pm$ 8.2\%  \\
 $0.11 \pm 0.28$
 \end{minipage} \\
 \noalign{\vskip 1mm}   
\hline 
\noalign{\vskip 1mm}   
                                           T-\textit{coupled} 
& \begin{minipage}{10mm}
\centering
                                           Score\\
                                           $\widehat{\Delta H} $
                                           \end{minipage}
& \begin{minipage}[c]{18mm}
 91.7 $\pm$ 1.7\%\\
 $ 0.20 \pm 0.05$
 \end{minipage}
& \begin{minipage}[c]{18mm}
 91.9 $\pm$ 4.0\%\\
 $0.06 \pm 0.04$
 \end{minipage}
& \begin{minipage}[c]{18mm}
 91.1 $\pm$ 1.5\%\\
 $ 0.10 \pm 0.02$
 \end{minipage}
&  \begin{minipage}[c]{18mm}
 85.5 $\pm$ 3.8\%\\
 $ 0.08 \pm 0.04$
 \end{minipage}
&  \begin{minipage}[c]{18mm}
 86.1 $\pm$ 4.3\% \\
 $ 0.05 \pm 0.02$
 \end{minipage}
& \begin{minipage}[c]{18mm}
 74.3 $ \pm$ 8.2\%  \\
 $0.06 \pm 0.04$
 \end{minipage} \\
\end{tabular}}
\caption{\label{tab:perfs}\textbf{Optimal segmentation performance for different configurations of fractal textures, averaged over $5$ realizations.} Piecewise fractal textures are characterized by $(\Sigma_0^2, H_0) = (0.6, 0.5)$ and different $(\Delta \Sigma^2, \Delta H)$ as sketched in Fig.~\ref{subfig:mask}. 
First row: T-ROF segmentation.
Second row: T-\textit{joint} segmentation. 
Third row: T-\textit{coupled} segmentation.}
\end{table}

\subsubsection{Computational costs}

Comparisons in terms of computational costs both between the three approaches, and between the two classes of proximal algorithms, dual forward-backward, standard and accelerated (FISTA), vs. primal-dual, standard and accelerated by strong convexity (cf. Sec.~\ref{sec:optimization}), are reported in Table~\ref{tab:compcosts}, for configurations I and III (regarded as \emph{easy} and \emph{difficult}) considered as representative. 
Computational costs are reported in number of iterations actually used to reach the stopping criterion and in real time, for the optimal set of hyperparameters and averaged over $5$ realizations.

Table~\ref{tab:compcosts} shows first that, as expected,  accelerated algorithms are always require less iterations than non accelerated ones, thus generally leading to lower computational times (though this is not always the case with FISTA whose complexity per iteration is larger).
Also, T-ROF shows always lower computational costs compared to T-\textit{joint} and T-\textit{coupled}.
This is expected as T-ROF only works with the regularity and do not use variance.

\textit{FISTA vs. Accelerated primal-dual.} For T-ROF, FISTA is overall preferable to the accelerated primal-dual algorithm, as both show equivalent computational costs for Configuration I but FISTA is ten times faster (both in number of iterations and computation time) for Configuration III. 
For T-\textit{joint} and T-\textit{coupled}, for both configurations, accelerated primal-dual is faster than FISTA. 
For T-\textit{coupled}, in configuration III, FISTA has actually not converged when meeting the the upper limit of iterations.
Therefore, FISTA is to be preferred for T-ROF, while accelerated primal-dual algorithms are more relevant for T-\textit{joint} and T-\textit{coupled}. 

\textit{T-\textit{joint} vs. T-\textit{coupled}.} Focusing on T-\textit{joint} and T-\textit{coupled} and thus on the accelerated primal-dual algorithm that is faster for these two methods, Table~\ref{tab:compcosts} shows that  T-\textit{joint} is solved 3 to 4 times faster (both in number of iterations and computational cost) than 
T-\textit{coupled}.

\setlength{\tabcolsep}{0.7mm}
\begin{table}[h!]
\centering
\resizebox{\textwidth}{!}{\begin{tabular}{ccccc|ccc}
 & & \multicolumn{2}{r}{\begin{minipage}{3cm} \centering Configuration I \\ \vspace{1mm} \end{minipage}} & &  \multicolumn{2}{r}{\begin{minipage}{32mm} \centering Configuration III  \\ \vspace{1mm} \end{minipage}}  \\
 & & T-ROF  & T-\emph{joint}	& T-\emph{coupled} &   T-ROF  & T-\emph{joint}	& T-\emph{coupled}   \\
 \midrule
 \multirow{4}{*}{\rotatebox{90}{\begin{minipage}[c]{1.5cm}
 \centering
 Iterations
 \\ ($10^3$ it.)
 \end{minipage} }} & \rotatebox{0}{DFB} & \begin{minipage}[c]{21mm}
\centering
$96  \pm 48 $ \\
\end{minipage}
& \begin{minipage}[c]{21mm}
\centering
$>250$  \\
\end{minipage}
& \begin{minipage}[c]{21mm}
\centering
$>250$  \\
\end{minipage}
& \begin{minipage}[c]{21mm}
\centering
$241  \pm 18 $  \\
\end{minipage}
& \begin{minipage}[c]{21mm}
\centering
$>250$  \\
\end{minipage}
& \begin{minipage}[c]{21mm}
\centering
$>250$  \\
\end{minipage}
\\ 
 & \rotatebox{0}{FISTA} & \begin{minipage}[c]{21mm}
\centering
$1.7  \pm 0.4 $  \\
\end{minipage}
& \begin{minipage}[c]{21mm}
\centering
$50.2  \pm 21.0 $  \\
\end{minipage}
& \begin{minipage}[c]{21mm}
\centering
$231  \pm 37$  \\
\end{minipage}
& \begin{minipage}[c]{21mm}
\centering
$\mathbf{3.7 \pm 0.7}$  \\
\end{minipage}
& \begin{minipage}[c]{21mm}
\centering
$48.1  \pm 3.4$  \\
\end{minipage}
& \begin{minipage}[c]{21mm}
\centering
$>250$  \\
\end{minipage}
\\
 & \rotatebox{0}{PD} &
 \begin{minipage}[c]{21mm}
\centering
$31.8  \pm 17.0 $  \\
\end{minipage}
& \begin{minipage}[c]{21mm}
\centering
$>250$  \\
\end{minipage}
& \begin{minipage}[c]{21mm}
\centering
$>250$  \\
\end{minipage}
& \begin{minipage}[c]{21mm}
\centering
$201  \pm 69 $  \\
\end{minipage}
& \begin{minipage}[c]{21mm}
\centering
$>250$  \\
\end{minipage}
& \begin{minipage}[c]{21mm}
\centering
$>250$  \\
\end{minipage}
\\
 & \rotatebox{0}{AcPD} & \begin{minipage}[c]{21mm}
\centering
$\mathbf{1.5 \pm 0.4}$ \\
\end{minipage}
& \begin{minipage}[c]{21mm}
\centering
$\mathbf{31.4  \pm 4.6 }$  \\
\end{minipage}
& \begin{minipage}[c]{21mm}
\centering
$\mathbf{125  \pm 67 }$  \\
\end{minipage}
& \begin{minipage}[c]{21mm}
\centering
$45.2  \pm 43$ \\
\end{minipage}
& \begin{minipage}[c]{21mm}
\centering
$\mathbf{40.5  \pm 2.8 }$ \\
\end{minipage}
& \begin{minipage}[c]{21mm}
\centering
$\mathbf{121  \pm 42 }$\\
\end{minipage}\\
 \midrule
 \multirow{4}{*}{\rotatebox{90}{Time (s)} }
  & \rotatebox{0}{DFB} & \begin{minipage}[c]{21mm}
\centering
$1,090  \pm 520$\\
\end{minipage}
& \begin{minipage}[c]{21mm}
\centering
$4,840  \pm 15$\\
\end{minipage}
& \begin{minipage}[c]{21mm}
\centering
$4,210  \pm 76$\\
\end{minipage}
& \begin{minipage}[c]{21mm}
\centering
$2,010 \pm 73$\\
\end{minipage}
& \begin{minipage}[c]{21mm}
\centering
$4,810  \pm 215$\\
\end{minipage}
& \begin{minipage}[c]{21mm}
\centering
$4,200  \pm 76$\\
\end{minipage}
\\ 
 & \rotatebox{0}{FISTA} &
 \begin{minipage}[c]{21mm}
\centering
$16 \pm 4$\\
\end{minipage}
& \begin{minipage}[c]{21mm}
\centering
$1,030  \pm 410 $\\
\end{minipage}
& \begin{minipage}[c]{21mm}
\centering
$4,800 \pm 560 $\\
\end{minipage}
& \begin{minipage}[c]{21mm}
\centering
$\mathbf{30 \pm 5}$\\
\end{minipage}
& \begin{minipage}[c]{21mm}
\centering
$989 \pm 64$\\
\end{minipage}
& \begin{minipage}[c]{21mm}
\centering
$5,110 \pm 340 $\\
\end{minipage}
\\
 & \rotatebox{0}{PD} &
 \begin{minipage}[c]{21mm}
\centering
$297 \pm 150$\\
\end{minipage}
& \begin{minipage}[c]{21mm}
\centering
$4,180  \pm 69$\\
\end{minipage}
& \begin{minipage}[c]{21mm}
\centering
$4,110 \pm 43$\\
\end{minipage}
& \begin{minipage}[c]{21mm}
\centering
$1,580  \pm 490$\\
\end{minipage}
& \begin{minipage}[c]{21mm}
\centering
$4,150  \pm 18$\\
\end{minipage}
& \begin{minipage}[c]{21mm}
\centering
$4,100  \pm 15$\\
\end{minipage}
\\
 & \rotatebox{0}{AcPD} &
 \begin{minipage}[c]{21mm}
\centering
$\mathbf{15 \pm 4}$\\
\end{minipage}
& \begin{minipage}[c]{21mm}
\centering
$\mathbf{619 \pm 96}$\\
\end{minipage}
& \begin{minipage}[c]{24mm}
\centering
$\mathbf{2,420 \pm 1,300 }$\\
\end{minipage}
& \begin{minipage}[c]{21mm}
\centering
$349 \pm 330$\\
\end{minipage}
& \begin{minipage}[c]{21mm}
\centering
$\mathbf{785 \pm 59}$\\
\end{minipage}
& \begin{minipage}[c]{21mm}
\centering
$\mathbf{2,320 \pm 790 }$\\
\end{minipage}
\end{tabular}}
\caption{\label{tab:compcosts} Number of iterations and computational time 
necessary to reach Condition~\eqref{eq:stop_crit} for the different proximal algorithms investigated, illustrated on two configurations I~($\Delta H = 0.2$, $\Delta \Sigma^2 = 0.1$) and III~($\Delta H = 0.1$, $\Delta \Sigma^2 = 0.1$). \textbf{DFB}: Dual Forward-Backward, \textbf{FISTA}: inertial acceleration of DFB, \textbf{PD}: primal-dual, \textbf{AcPD}: strong-convexity based acceleration of PD.}
\end{table}

\subsubsection{Overall comparison}

As an overall conclusion, results reported above show that there are benefits to use together local regularity and variance, compared to regularity only, when the changes in regularity become small. 
This implies switching from accelerated dual forward-backward algorithms (FISTA, for T-ROF) to accelerated primal-dual algorithms (for T-\textit{joint} and T-\textit{coupled}). 

For difficult configurations, T-\textit{coupled} (slightly) outperforms T-\textit{joint} in terms of segmentation performance, at the price of non negligible increases of computational costs. 
In that sense,  T-\textit{joint} can be considered a reasonable trade-off between too poor segmentation performance (as those of T-ROF) and too large computational costs (as those of T-\textit{coupled}). 

\subsection{Real-world textures}
\label{sec:random_contours}

The actual potential of the proposed segmentation strategies is further illustrated on real-world textures, arbitrarily chosen from the commonly used large, documented and publicly available \emph{University of Maryland, HighResolution (UMD HR)}\footnote{{\tt http://legacydirs.umiacs.umd.edu/~fer/website-texture/texture.htm}} texture dataset.
This dataset consists of 50 classes of homogeneous textures, each consisting of 40 images of same textures under different imaging conditions (light, angle,\ldots). 
From arbitrarily chosen pairs of images, piecewise homogeneous textures are constructed by including textures one into the other using the ellipse-shaped mask used above, as sketched and illustrated in Fig.~\ref{fig:umd_textures}. 
Homogeneous textures are centered and normalized in variance independently and prior to inclusion, to avoid that detection could be based on sole change of means or variances.

The T-ROF, T-\textit{Joint} and T-\textit{Coupled} segmentation approaches proposed here are applied to and compared on these real-world piecewise homogeneous textures.
Further, they are compared against one of the state-of-the-art texture segmentation procedure, based on matrix factorization and local spectral histograms clustering, proposed in~\cite{yuan_factorization-based_2015} and being a joint feature selection and segmentation procedure, and hereafter referred to as \emph{LSHC-MF}. 

Fig.~\ref{fig:res_umd} reports and compares segmentation performance, in terms of percentage of pixels correctly labelled. 
For the T-ROF, T-\textit{Joint} and T-\textit{Coupled} approaches, optimal performance are reported, after a grid search 
for the selection of the hyperparameters $\lambda$ and $\alpha$. 
For \emph{LSHC-MF}, optimal performance are also reported, after a grid search on the window size, within which local histogram statistics are computed.   

Fig.~\ref{fig:res_umd} clearly shows that the \emph{LSHC-MF} approach fails to segment correctly the piecewise homogeneous texture and is outperformed by the T-ROF, T-\textit{Joint} and T-\textit{Coupled} strategies.
It also shows that T-ROF, while not being inconsistant, does not yield satisfactory segmentation, whereas T-\textit{Joint} and T-\textit{Coupled} do. 
Further, it is worth noting that not only the \textit{Coupled} strategy reaches (slightly) better segmentation performance compared to the \textit{Joint} one, but that it also yields smoother boundaries, an outcome of utmost 
 relevance for decision/diagnostic in numerous real-world applications, cf. e.g.~\cite{pascal2018joint}. 
Such significant improvement in segmentation performance is at the cost tough of significantly higher computational cost: 
When \emph{LSHC-MF} requires of the order of $1$s to perform its most accurate segmentation, as in Fig.~\ref{fig:res_umd}(d), The T-ROF, T-\textit{Joint} and T-\textit{Coupled} segmentation approaches require of the orders of $10$s, $600$s $1800$s respectively, in agreement with computational cost evaluation reported in Tab.~\ref{tab:compcosts}. 
Consistent conclusions can be drawn for other choices of classes of textures in the UMD HR dataset.

\begin{figure}
\centering
\begin{subfigure}{0.24\linewidth}
\centering
\includegraphics[width = 2.9cm]{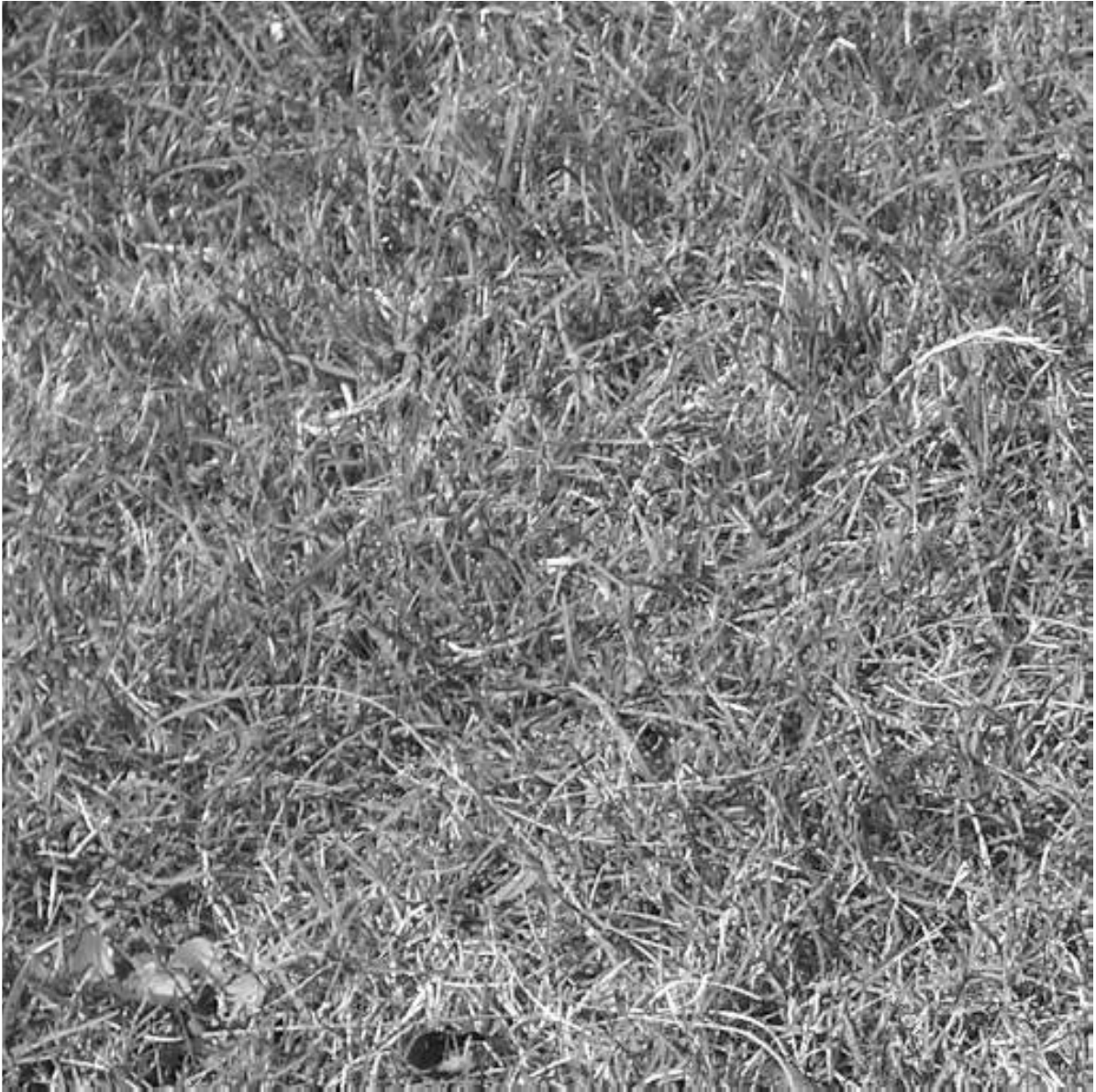}
\subcaption{Texture \textbf{A}}
\end{subfigure}
\begin{subfigure}{0.24\linewidth}
\centering
\includegraphics[width = 2.9cm]{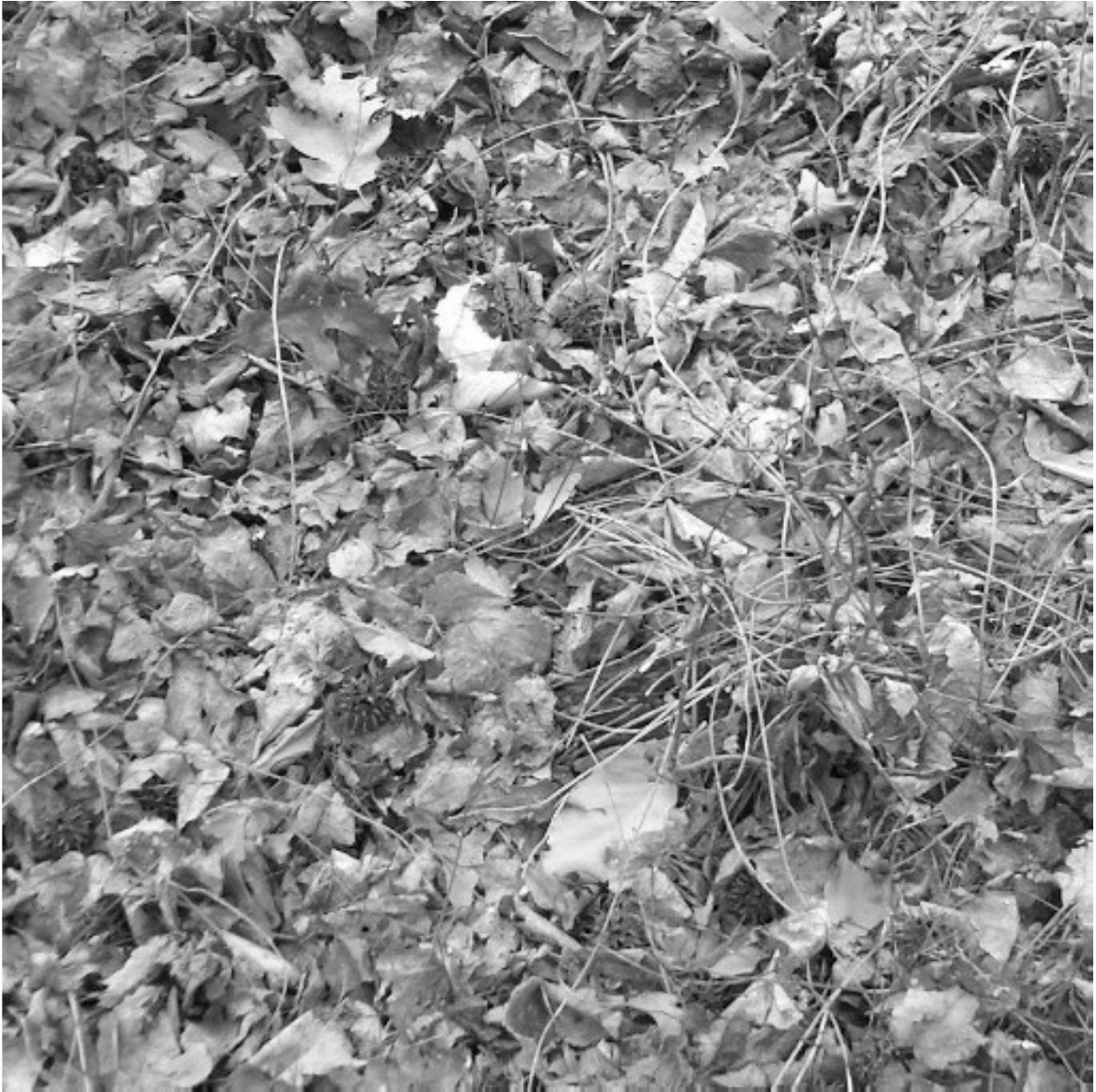}
\subcaption{Texture \textbf{B}}
\end{subfigure}
\begin{subfigure}{0.24\linewidth}
\centering
\includegraphics[width = 2.9cm]{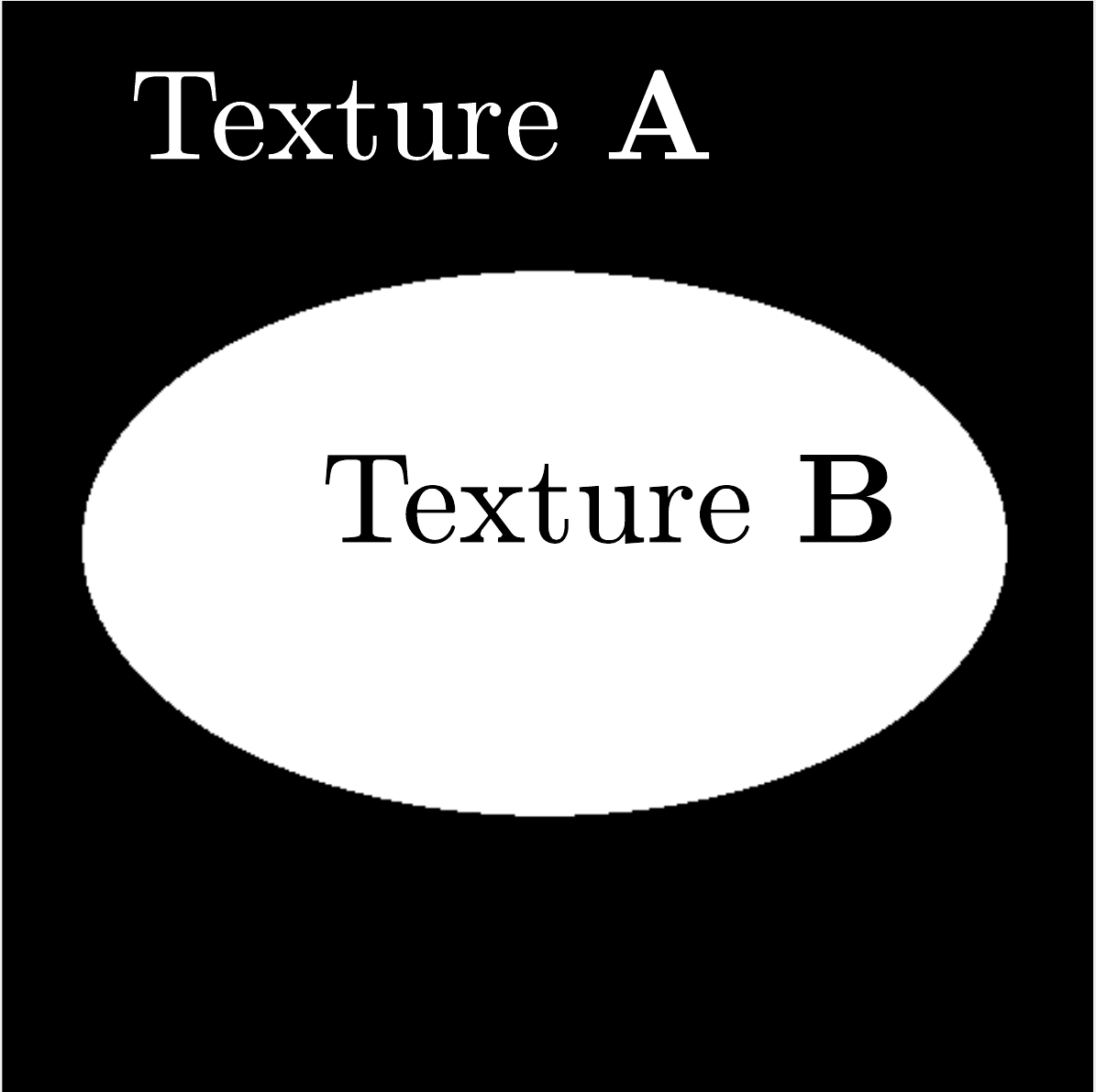}
\subcaption{Mask}
\end{subfigure}
\begin{subfigure}{0.24\linewidth}
\centering
\includegraphics[width = 2.9cm]{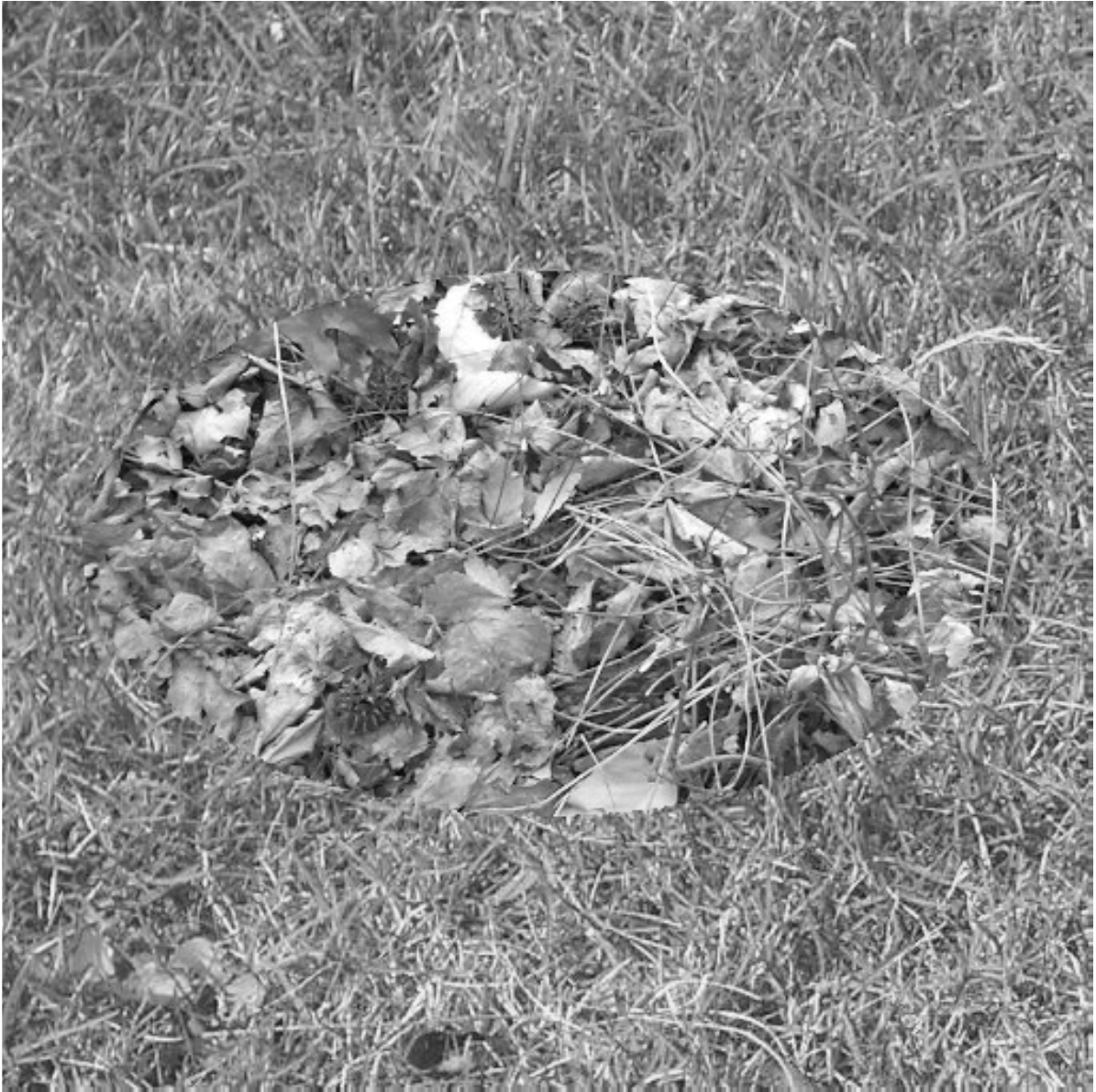}
\subcaption{Mixture \textbf{A}-\textbf{B}}
\end{subfigure}
\caption{\label{fig:umd_textures} \textbf{Real world textures.} two samples of real-world textures taken from the UMD HR texture dataset ((a) and (b)), piecewise constant mask (c) and piecewise homogeneous texture (d).}
\end{figure}

\begin{figure}
\centering
\begin{tabular}{cccc}
\includegraphics[width = 2.9cm]{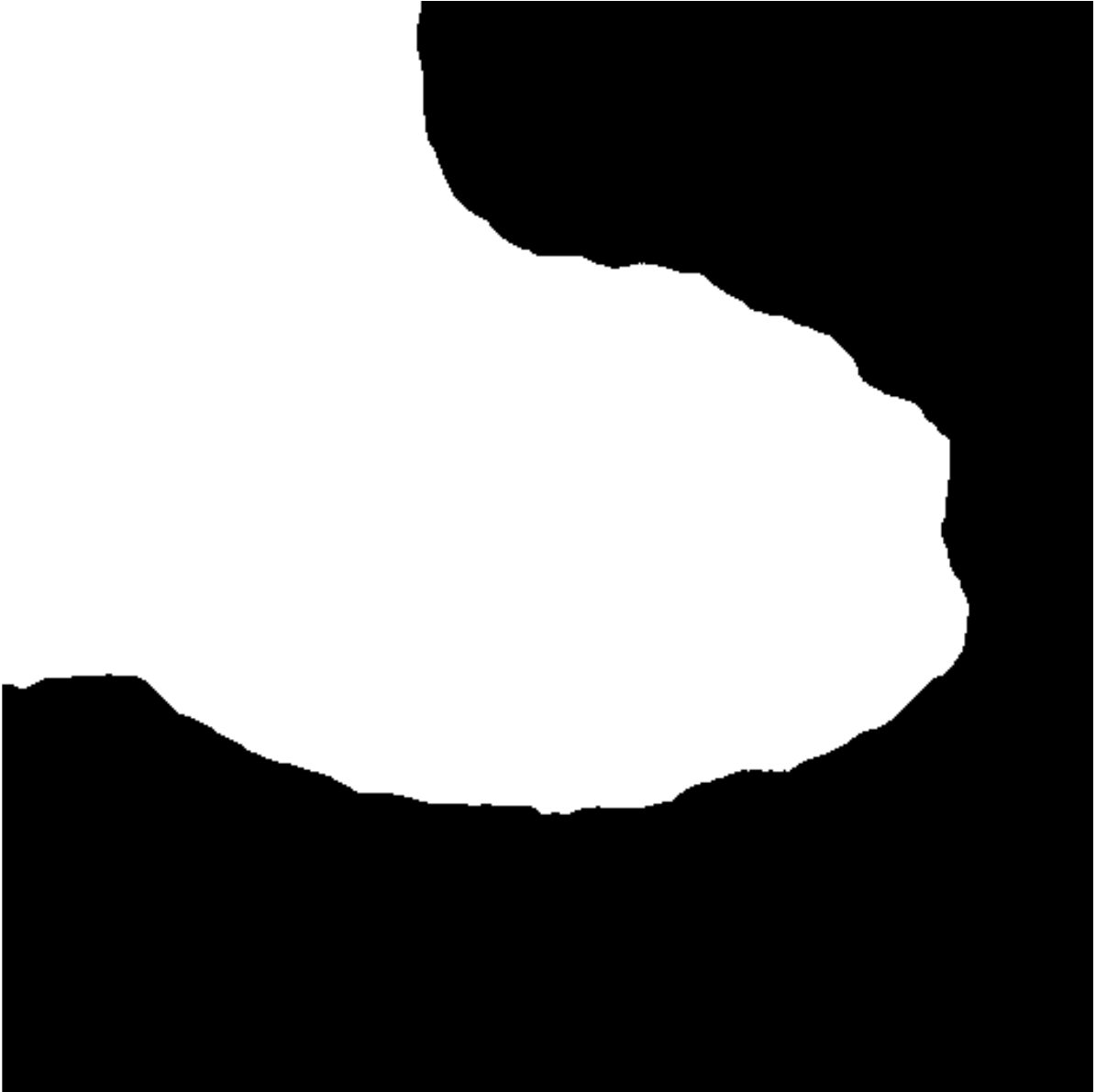} & \includegraphics[width = 2.9cm]{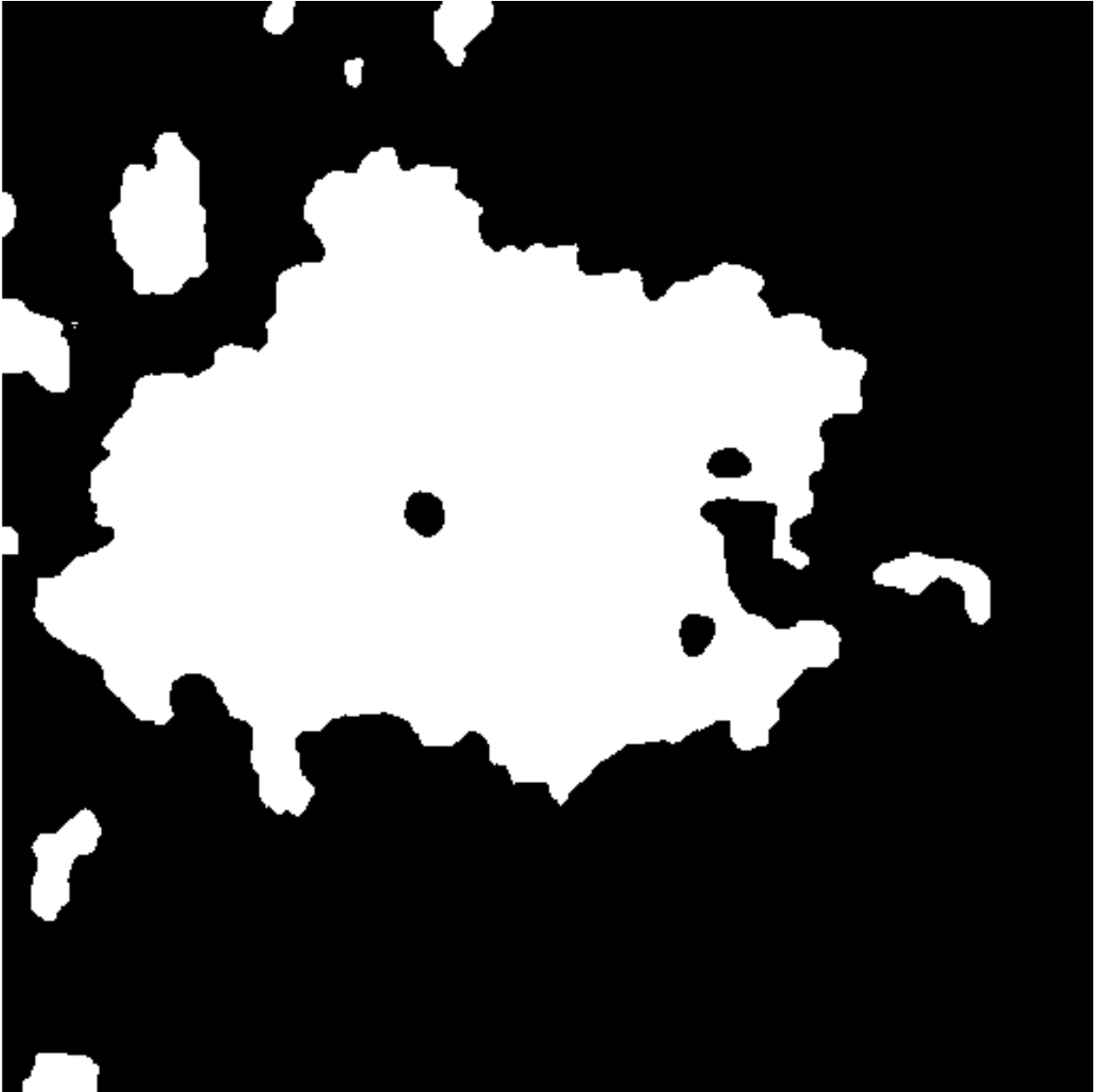}
& \includegraphics[width = 2.9cm]{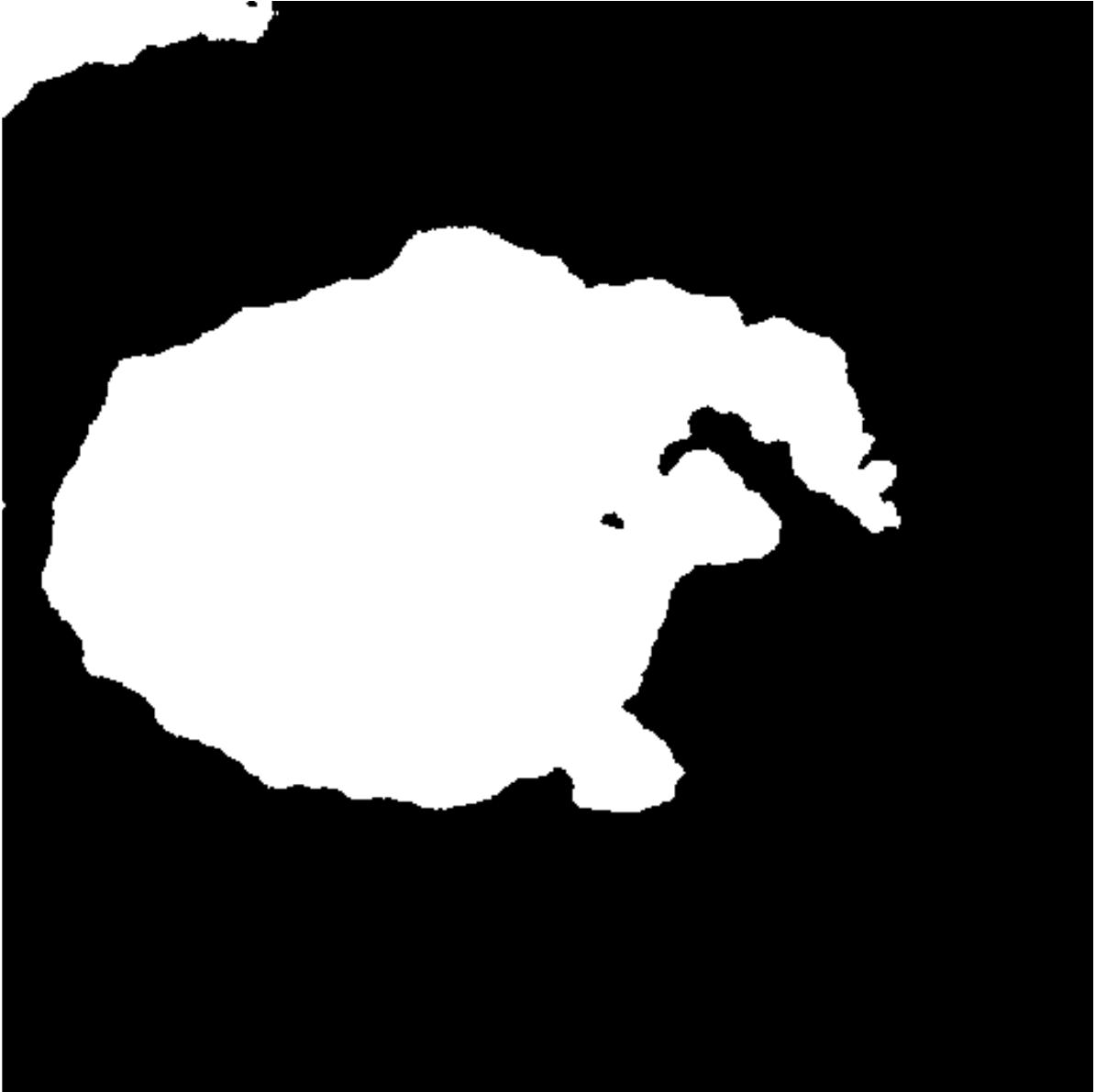} & \includegraphics[width = 2.9cm]{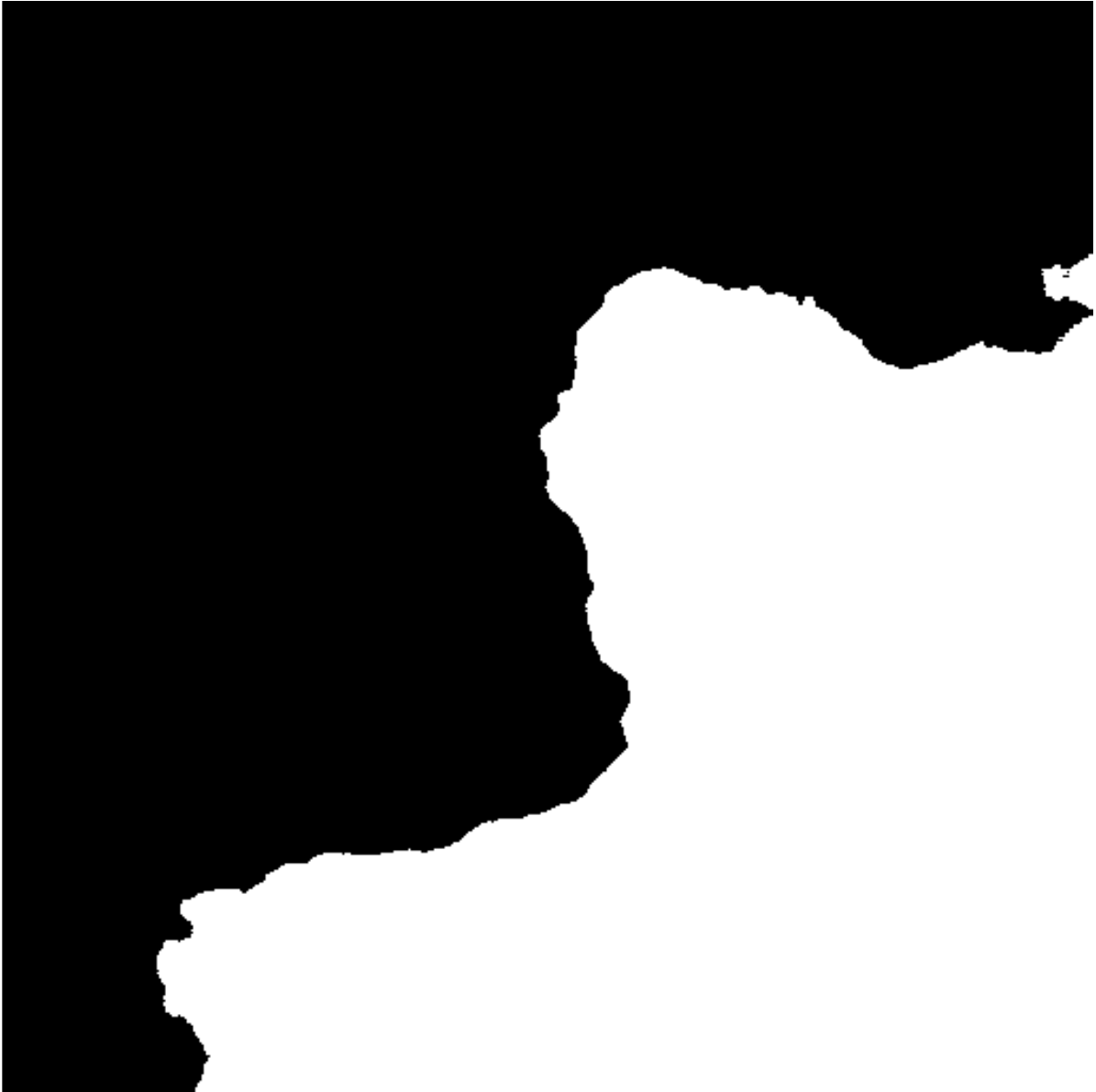}\\
$82.6\%$ & $86.1\%$ & $87.9\%$ & $55.2 \%$ \\
(a) T-ROF & (b) Joint & (c) Coupled & (d) LSHC-MF \\
\end{tabular}
\caption{\label{fig:res_umd} \textbf{Segmentation performance on real-world textures.}
For the piecewise homogeneous texture shown in Fig.~\ref{fig:umd_textures}(d), performance are computed in term of the percentage of well-classified pixels for T-ROF (a),  \textit{Joint} (b),  \textit{Coupled} (c) and LSHC-MF (d). 
}
\end{figure}

\section{Conclusion and future work}
\label{sec:conc}
The present article has significantly advanced the state-of-the-art in the segmentation of piecewise fractal textures. 

First, it has been proposed to base the segmentation of fractal textures not only on the estimation of the sole local regularity parameter, but to use an additional local parameter, the log-wavelet variance, tightly related to the local variance of the textures. 
Two variations were investigated, \emph{coupled} and \emph{joint}, that respectively enforce or not co-localized changes in regularity and variance.
It has been shown, using large size Monte Carlo simulations, that the use of this additional features improves drastically segmentation performance when the difference in regularity becomes negligible. 
This yet comes at the price of a non negligible increase in computational costs. \\
\indent Therefore, a second contribution has been to construct accelerated primal-dual algorithms, requiring the explicit calculation of the strong convexity constant underlying the data fidelity term form. 
The achieved substantial reduction in computational costs has turned critical both to be able to conduct large size Monte Carlo simulation and to perform the greedy search of the optimal set of hyper--parameters. 
This low computational cost is also crucial for application on real-world data. \\
\indent The investigations reported here have permitted to show that accelerated primal-dual algorithms outperform accelerated dual forward-backward (FISTA-type) algorithms for piecewise fractal texture segmentation as soon as the joint use of regularity and variance is required. 
Further, they showed that the \emph{coupled} formulation, that favor co-localized changes in regularity and variance, performs better than the \emph{joint} formulation, yet at the price of a significantly larger computational cost. 
Thus, depending on budget constraints on time and requested quality of the solution, the \emph{joint} formulation can be regarded as an effective trade-off. \\
\indent  The proposed theoretical formulations for piecewise fractal texture segmentation and the corresponding accelerated algorithms are matured enough for applications on real-world data and competitive with respect to state-of-the-art as shown in Sec.~\ref{sec:random_contours}. 
Application to the segmentation of multiphasic flows is under current investigations.
The automation of the tuning of the hyperparameters, which is a crucial point for real-world applications, as underlined in Sec.~\ref{sec:random_contours},  is also being investigated. 
Extensions to piecewise multifractal textures are further targeted. \\
\indent A {\sc matlab} toolbox implementing the analysis and synthesis procedures devised here will be made freely and publicly available at the time of publication.

\bibliographystyle{plain}
\bibliography{abbr,pami}

\end{document}